\documentclass{amsart}

\usepackage{amssymb}
\usepackage{mathrsfs}
\usepackage{tikz-cd}
\usepackage{array,booktabs,longtable}
\usepackage{microtype}
\usepackage{enumitem}
\usepackage{needspace}
\usepackage{xcolor}
\usepackage[bookmarksnumbered=true,bookmarksopen=true,bookmarksopenlevel=2,bookmarksdepth=3]{hyperref}

\definecolor{AMSDarkBlue}{RGB}{0,51,102}
\hypersetup{
  colorlinks=true,
  linkcolor=AMSDarkBlue,
  citecolor=AMSDarkBlue,
  urlcolor=AMSDarkBlue,
  breaklinks=true,
  pdfpagemode=UseOutlines,
  pdftitle={Rational functions over finite fields with Galois closure of genus zero},
  pdfauthor={Xiang Fan}
}
\numberwithin{equation}{section}

\newtheorem{theorem}{Theorem}[section]
\newtheorem{proposition}[theorem]{Proposition}
\newtheorem{corollary}[theorem]{Corollary}
\newtheorem{lemma}[theorem]{Lemma}
\theoremstyle{definition}
\newtheorem{definition}[theorem]{Definition}
\newtheorem{example}[theorem]{Example}
\theoremstyle{remark}
\newtheorem{remark}[theorem]{Remark}

\title[Rational functions with Galois closure of genus zero]{Rational functions over finite fields with Galois closure of genus zero}
\author{Xiang Fan}
\address{School of Mathematics, Sun Yat-sen University, Guangzhou 510275, China}
\email{fanx8@mail.sysu.edu.cn}
\date{}
\subjclass[2020]{Primary 14H30; Secondary 14G15, 20B15, 12F10, 11T06}
\keywords{rational functions over finite fields,
  genus-zero Galois closures, Frobenius descent, monodromy groups,
  exceptional covers, permutation rational functions,
  functional decomposition}

\begin{document}

\begin{abstract}
Let $k=\mathbb F_q$. We classify, up to pre- and post-composition by
$k$-M\"obius transformations, all separable $k$-indecomposable rational
functions $f\in k(X)$ of degree greater than one whose Galois closure has
genus zero. The classification is valid in arbitrary characteristic and
includes exact arithmetic conditions and class counts over the prescribed
field. Semilinear Frobenius descent determines the finite-field forms and
which geometric decompositions descend to $k$.

For every separable $f\in k(X)$ of degree greater than one with Galois
closure of genus zero and every $m\geqslant1$, we prove that $f$ permutes
$\mathbf P^1(\mathbb F_{q^m})$ if and only if it is exceptional over
$\mathbb F_{q^m}$, meaning that it permutes $\mathbf P^1(L)$ for infinitely
many finite extensions $L/\mathbb F_{q^m}$. No indecomposability assumption
or lower bound on $q$ is needed. The argument combines fixed-point
averaging with ramification on the Galois-closure curve. If the full
constant field is $\mathbb F_{q^d}$, these properties depend only on
$\gcd(m,d)$. For each classified family we determine the permutation
extension degrees explicitly and characterize polynomial representatives.
\end{abstract}

\maketitle

\section{Introduction and main results}

Let $k=\mathbb F_q$ be a finite field of characteristic $p$.
Fix an algebraic closure $\bar k$ of $k$ and let $\mathbf x$ be
transcendental over $\bar k$.  Throughout, we work inside a fixed algebraic
closure of $\bar k(\mathbf x)$.  A nonconstant rational function $f\in k(X)$
will also be viewed as the corresponding $k$-morphism
$f:\mathbf P^1_k\to\mathbf P^1_k$.  We call $f$ \emph{separable} if the
extension $k(\mathbf x)/k(f(\mathbf x))$ is separable.  For separable $f$, let $\Omega_f$ denote the Galois closure of the finite
separable extension $k(\mathbf x)/k(f(\mathbf x))$, taken inside our fixed
algebraic closure of $\bar k(\mathbf x)$.  Thus
$k(\mathbf x)\subseteq\Omega_f$ and $\Omega_f/k(f(\mathbf x))$ is a finite
Galois extension.  The \emph{full constant field} of $\Omega_f$ is
$\kappa_f:=\Omega_f\cap\bar k$.  Put
$\overline{\Omega}_f:=\Omega_f\bar k$; this is the geometric Galois closure
of $\bar k(\mathbf x)/\bar k(f(\mathbf x))$.  Let $C_f/\kappa_f$ be the
smooth projective curve with function field $\Omega_f$.  We say that $f$ has
\emph{Galois closure of genus zero} when $C_f$ has genus zero.

When $f$ has Galois closure of genus zero, the curve $C_f$ is a projective
line over $\kappa_f$, so the geometric normal closure is governed by finite
automorphism groups of $\mathbf P^1$. Yet the cover over $\kappa_f$ does not
determine the rational function over the prescribed field $k$.
Frobenius descent determines the covers over $k$, can prevent geometric
intermediate covers from descending to $k$, and controls the extension degrees
on which permutation occurs.  We make this arithmetic layer explicit and
classify it completely.

A rational function is \emph{$k$-indecomposable} if it cannot be written
$f=g\circ h$ with $g,h\in k(X)$ of degrees greater than $1$.  Here
\emph{$\bar k$-indecomposable} is also called
\emph{geometrically indecomposable}.  We write $\operatorname{PGL}_2(k)$ for the group of all degree-one rational
functions over $k$, under composition.
Two rational functions $f,g\in k(X)$ are called
\emph{$k$-M\"obius equivalent} if $g=\alpha\circ f\circ\beta$ for some
$\alpha,\beta\in\operatorname{PGL}_2(k)$. We write $f\sim_k g$ in this
case, and call the resulting equivalence class a \emph{$k$-M\"obius class}.

Let $\pi:X\to Y$ be a finite separable morphism of smooth projective
curves over a field $F$, let $P\in X(\overline F)$, and put $Q=\pi(P)$.
A \emph{local parameter} at $Q$ is a rational function $u_Q$ on
$Y_{\overline F}$ having a simple zero at $Q$, equivalently
$\operatorname{ord}_Q(u_Q)=1$.
The \emph{ramification index} of $\pi$ at $P$ is
\[
 e_P(\pi):=\operatorname{ord}_P(u_Q\circ\pi);
\]
it is independent of the choice of $u_Q$.
We call $P$ a \emph{ramification point}, or say that $P$ is ramified,
if $e_P(\pi)>1$, and we call a point of the target a \emph{branch point}
if some point above it is ramified.
The set of ramification points is the \emph{ramification locus} of
$\pi$, and the set of branch points is its \emph{branch locus}.
We use the term \emph{branch divisor} for the reduced effective divisor
supported on the branch locus.
For $Q\in Y(\overline F)$, the \emph{ramification partition} is
\[
 \operatorname{Ram}_Q(\pi)
 :=\{\!\{\,e_P(\pi):P\in\pi^{-1}(Q)\,\}\!\}.
\]
Since $\sum_{P\in\pi^{-1}(Q)}e_P(\pi)=\deg\pi$, this is a partition
of $\deg\pi$.
We call a separable rational function $f\in k(X)$ \emph{tame} if all
its geometric ramification indices are prime to $p$, and \emph{wild}
otherwise.
Tameness and wildness are invariant under $k$-M\"obius equivalence;
accordingly, a $k$-M\"obius class represented by a separable rational
function will also be called tame or wild.

\Needspace{5\baselineskip}
\medskip
\noindent\textit{Main results.}

\begin{theorem}[Complete finite-field classification]
The $k$-M\"obius classes represented by separable
$k$-indecomposable rational functions $f\in k(X)$ of degree greater than $1$
whose Galois closure has genus zero are exactly the classes in the tame and
wild classifications of Theorems~\ref{thm:intro-tame}
and~\ref{thm:intro-wild}.  After the explicit small-group identifications in
Section~\ref{sec:wild-master}, these lists are mutually exclusive.  Every
listed class occurs exactly under the stated arithmetic conditions, and the
number of classes over the prescribed field $k$ is given explicitly.
\end{theorem}

At the level of families, the tame classes are the prime-degree cyclic and
dihedral families, the quartic $V_4/A_4$, tetrahedral and octahedral classes,
and the icosahedral classes of degrees $5,6,10$.  The wild classes are affine,
small-characteristic dihedral or $A_5$, or of
$\operatorname{PSL}_2/\operatorname{PGL}_2$ type over finite fields of
characteristic $p$.

The classification is not merely a list of geometric monodromy groups and
point stabilizers; it also
records the distinction between $k$-indecomposability and geometric
indecomposability.  The $V_4/A_4$ quartic is the only tame
$k$-indecomposable class that is geometrically decomposable.  It illustrates a
phenomenon that also occurs in the wild classification: Frobenius can eliminate
every geometric intermediate field, so a rational function may be
indecomposable over $k$ while decomposing over $\bar k$.

For a finite field $K$, we call a rational function $f\in K(X)$
\emph{exceptional} if it permutes $\mathbf P^1(L)$ for infinitely many finite
extensions $L/K$.

\begin{theorem}[Permutation--exceptionality equivalence and exact extension degrees]
Let $f\in k(X)$ be separable of degree greater than $1$ and suppose that its
Galois closure has genus zero.  Let $\kappa_f=\mathbb F_{q^d}$ be the full
constant field of the Galois closure, and put
\[
 \mathcal E_f:=\{\,g\geqslant1:g\mid d,\ f\text{ is exceptional over }\mathbb F_{q^g}\,\}.
\]
For every $m\geqslant1$, the following are equivalent:
\begin{enumerate}[label=\textup{(\roman*)},ref=\roman*]
\item $f$ permutes $\mathbf P^1(\mathbb F_{q^m})$;
\item $f$ is exceptional over $\mathbb F_{q^m}$;
\item $(m,d)\in\mathcal E_f$.
\end{enumerate}
Moreover, $d\notin\mathcal E_f$.  If $g\in\mathcal E_f$ and $h$ is a
positive integer with $h\mid g$, then $h\in\mathcal E_f$.  Hence the common
set of permutation and exceptional extension degrees is purely periodic; if
it is nonempty, it contains degree $1$ and every degree coprime to $d$.
\end{theorem}

This second theorem applies to the entire separable genus-zero class and has no
indecomposability hypothesis.  The two headline results are therefore logically
distinct: the first classifies the indecomposable building blocks over the
prescribed field, whereas the second gives uniform extension arithmetic for the
whole separable class.  Its proof converts bijectivity into local fixed-point
constraints and combines them with ramification on the Galois-closure curve;
in the Lie-type cases, it uses Shintani descent for $\operatorname{PSL}_2$ and $\operatorname{PGL}_2$.  The decomposition
formalism then removes the indecomposability hypothesis.

\medskip
\noindent\textit{Semilinear descent.}

A \emph{$k$-form} of a prescribed geometric cover or pair is an object
over $k$ whose base change to $\bar k$ is isomorphic to that geometric
object. We use \emph{finite-field form} in the same sense.
We describe the mechanism behind the finite-field classification. To isolate
the Galois-theoretic input, let $f\in k(X)$ be a nonconstant separable rational
function; no genus-zero hypothesis is imposed at this stage. Put
$\mathbf t:=f(\mathbf x)$ and, for the arithmetic Galois closure $\Omega_f$, set
\[
\begin{alignedat}{2}
 A_f&:=\operatorname{Gal}(\Omega_f/k(\mathbf t)),\qquad
 &G_f&:=\operatorname{Gal}(\Omega_f/\kappa_f(\mathbf t)),\\
 U_f&:=\operatorname{Gal}(\Omega_f/k(\mathbf x)),
 &H_f&:=\operatorname{Gal}(\Omega_f/\kappa_f(\mathbf x)).
\end{alignedat}
\]
We call the isomorphism class of the pair $(G_f,H_f)$ the
\emph{geometric monodromy pair} of $f$, or simply its \emph{geometric pair}.
Here two pairs $(G,H)$ and $(G',H')$ are isomorphic if there is an
isomorphism $G\to G'$ carrying $H$ onto $H'$. Equivalently, the geometric
pair records the transitive coset action
$G_f\curvearrowright G_f/H_f$.
Then
\[
 A_f/G_f\cong U_f/H_f\cong\operatorname{Gal}(\kappa_f/k),
 \qquad U_f\cap G_f=H_f.
\]
For a subgroup $V\leqslant B$, put
\[
 \operatorname{core}_B(V):=\bigcap_{b\in B}bVb^{-1}.
\]
We call $V$ \emph{core-free in $B$} if
$\operatorname{core}_B(V)=\{1\}$, equivalently, if the natural action
$B\curvearrowright B/V$ is faithful.
The arithmetic point stabilizer $U_f$ is core-free in $A_f$, and the geometric
point stabilizer $H_f$ is core-free in $G_f$.
Choosing a lift $\varphi_f\in U_f$ of $q$-Frobenius on $\kappa_f$ gives
\[
 A_f=\langle G_f,\varphi_f\rangle,\qquad
 U_f=\langle H_f,\varphi_f\rangle.
\]
The subscript $f$ indicates objects arising from the actual Galois closure of
$f$; the datum below uses unsubscripted letters for the corresponding abstract
semilinear data.  Likewise, the full constant field is denoted by $\kappa_f$
on the $f$-side and by $\kappa$ on the datum side.  Their relation is made
precise in Corollary~\ref{cor:f-datum-compatibility}.

\begin{definition}[Semilinear $\mathbf P^1$-datum]
\label{def:semilinear-P1-datum}
A \emph{semilinear $\mathbf P^1$-datum over $k$} is a tuple
\[
 \mathcal D=(\kappa/k,C,G,H,\varphi)
\]
with the following properties:
\begin{enumerate}[label=\textup{(\roman*)},ref=\roman*]
\item $\kappa/k$ is a finite subextension of $\bar k/k$ of degree $d$;
\item $C$ is a smooth projective $\kappa$-curve that is isomorphic over
      $\kappa$ to $\mathbf P^1_{\kappa}$;
\item $H$ and $G$ are finite subgroups of
      $\operatorname{Aut}_{\kappa}(C)$, with $H\leqslant G$;
\item $\varphi$ is a \emph{$q$-semilinear} automorphism of $C$, meaning
      an element of $\operatorname{Aut}_k(\kappa(C))$ whose restriction to
      $\kappa$ is $a\mapsto a^q$; moreover
      \[
       \varphi G\varphi^{-1}=G,\qquad
       \varphi H\varphi^{-1}=H,\qquad
       \varphi^d\in H;
      \]
\item with $A:=\langle G,\varphi\rangle$ and
      $U:=\langle H,\varphi\rangle$, the subgroup $U$ is core-free in $A$.
\end{enumerate}
\end{definition}

For a semilinear $\mathbf P^1$-datum
$\mathcal D=(\kappa/k,C,G,H,\varphi)$ over $k$, with $A$ and $U$ as above,
put $\Omega:=\kappa(C)$. Then $H$ is core-free in $G$.
Indeed, $\operatorname{core}_G(H)$ is normalized by $\varphi$, hence normal
in $A$, and is contained in $U$; therefore the core-freeness of $U$ in $A$
forces $\operatorname{core}_G(H)=\{1\}$. Restriction to the constants gives
\[
 A/G\cong U/H\cong\operatorname{Gal}(\kappa/k)
 \quad\text{and}\quad
 \Omega^{A}\cap \kappa
 =\Omega^{U}\cap \kappa=k.
\]
Let $X_{\mathcal D}$ and $Y_{\mathcal D}$ be the smooth projective
$k$-curves with function fields
\[
 k(X_{\mathcal D})=\Omega^{U},\qquad
 k(Y_{\mathcal D})=\Omega^{A}.
\]
Since
\[
 \Omega^{A}\kappa\subseteq\Omega^{G},\qquad
 \Omega^{U}\kappa\subseteq\Omega^{H},
\]
and the two inclusions have equal degrees on the two sides, namely
$[A:G]=[U:H]=[\kappa:k]$, they are equalities.  Hence
\[
 (X_{\mathcal D})_{\kappa}\cong C/H,\qquad
 (Y_{\mathcal D})_{\kappa}\cong C/G.
\]
The inclusion $\Omega^{A}\subseteq\Omega^{U}$ therefore defines a
$k$-morphism
\[
 \gamma_{\mathcal D}:X_{\mathcal D}\longrightarrow Y_{\mathcal D}.
\]
Under the preceding identifications, its base change to $\kappa$ is the
unique quotient morphism induced by $H\leqslant G$, which we denote by
\[
 \pi_{\mathcal D}:C/H\longrightarrow C/G.
\]
By L\"uroth's theorem, the quotient curves $C/H$ and $C/G$ are
isomorphic over $\kappa$ to $\mathbf P^1_{\kappa}$, so $X_{\mathcal D}$ and
$Y_{\mathcal D}$ have genus zero.  Over the finite field $k$ each is therefore
$k$-isomorphic to $\mathbf P^1$.  Choosing $k$-isomorphisms from source and
target to $\mathbf P^1$ gives a rational function in $k(X)$; changing either
coordinate changes it precisely by pre- or post-composition by an element of
$\operatorname{PGL}_2(k)$.  Thus the resulting $k$-M\"obius class is independent
of the choices.  We denote it by $\mathcal M_{\mathcal D}$.  We call
$\pi_{\mathcal D}$ \emph{tame} if all its ramification indices are prime to
$p$, and \emph{wild} otherwise.  Since $(\gamma_{\mathcal D})_{\kappa}$ is
identified with $\pi_{\mathcal D}$ and $k$-M\"obius transformations preserve
ramification indices, $\pi_{\mathcal D}$ is tame or wild exactly when
$\mathcal M_{\mathcal D}$ is tame or wild, respectively.

For a semilinear $\mathbf P^1$-datum
$\mathcal D=(\kappa/k,C,G,H,\varphi)$, we call a subgroup
$J\leqslant G$ \emph{$\varphi$-stable} if $\varphi J\varphi^{-1}=J$.

\begin{theorem}[Semilinear descent and indecomposability]
Let $f\in k(X)$ be separable and suppose that its Galois closure has genus zero.
Then there is a semilinear $\mathbf P^1$-datum
$\mathcal D=(\kappa/k,C,G,H,\varphi)$ over $k$ with
$f\in\mathcal M_{\mathcal D}$ and $\deg f=[G:H]$.  Moreover, $f$ is
$k$-indecomposable if and only if there is no $\varphi$-stable subgroup $J$
with $H\lneq J\lneq G$.  More generally, every functional decomposition of $f$ is
encoded by a strict chain of $\varphi$-stable subgroups between $H$ and $G$.
\end{theorem}

The $f$-subscripted Galois-closure data and the unsubscripted datum data are
compared precisely in Corollary~\ref{cor:f-datum-compatibility}: when
$f\in\mathcal M_{\mathcal D}$, one has $\kappa_f=\kappa$, and there is a
$\kappa$-isomorphism $C_f\to C$ identifying $G_f,H_f$ and the Frobenius coset
$H_f\varphi_f$ with $G,H$ and $H\varphi$.

The exact equality criterion for two classes $\mathcal M_{\mathcal D}$ is
Theorem~\ref{thm:exact-map-equivalence}; the complete decomposition statement
is Theorem~\ref{thm:complete-decomposition}.  The same datum controls branch
descent and ramification through Frobenius orbits and double cosets.

\medskip
\noindent\textit{Classification summary.}

The complete lists are proved in Sections~\ref{sec:tame-basic}--\ref{sec:wild-novelty}.
\begin{description}[font=\normalfont\itshape]
\item[Tame] The $V_4/A_4$ quartic and the cyclic and dihedral prime-degree
classes (Section~\ref{sec:tame-basic}); the tetrahedral and octahedral quartics
and the icosahedral classes of degrees $5,6,10$
(Section~\ref{sec:tame-polyhedral}).
\item[Wild] The affine $V\rtimes C_n$ classes, the characteristic-$2$
dihedral classes, and the characteristic-$3$ $A_5$ classes
(Section~\ref{sec:wild-master}); the geometrically indecomposable defining-characteristic
  $\operatorname{PSL}_2/\operatorname{PGL}_2$ quotients
(Section~\ref{sec:wild-lie}); and the Lie-type quotients that are $k$-indecomposable but geometrically
  decomposable
(Section~\ref{sec:wild-novelty}).
\end{description}
Those sections also give the complete stabilizer ranges, finite-field forms,
and exact class counts, including the classes that are $k$-indecomposable but geometrically decomposable.

The general permutation--exceptionality equivalence and the exact
extension-degree criterion are proved first
(Section~\ref{sec:no-accidental}).  The permutation extension degrees are then
determined explicitly for every classified family
(Section~\ref{sec:family-support}).  Exact finite-field class counts and
exceptional subcounts are computed (Section~\ref{sec:arithmetic-closure}), and
polynomial representatives are characterized there as well.

Every Frobenius orbit on the geometric branch locus in the indecomposable
classification has length at most $3$ (Section~\ref{sec:wild-novelty}).  The least exceptional
degree is $3$, the least wild exceptional degree is $4$ in characteristic $2$
and $p$ in odd characteristic, and every odd prime occurs as the degree of an
indecomposable exceptional class (Section~\ref{sec:arithmetic-closure}).

\medskip
\noindent\textit{Relation to previous work.}
Over $\mathbb C$, Pakovich \cite{Pakovich2018} gives a complete list of
rational functions whose Galois closure has genus zero and a geometric
description of those with Galois closure of genus one. Related work of
Gow--McGuire \cite{GowMcGuire2021} treats the realization of finite
subgroups of $\operatorname{PGL}_2$ as Galois groups over function fields
and descent to subfields.

For the geometric monodromy groups, we use the classical Klein--Dickson
classification in the tame case, with Beauville \cite{Beauville2010} as a
modern reference for finite subgroups of order prime to the characteristic,
and Faber's classification of finite subgroups whose order is divisible by the
characteristic \cite{Faber2023}.  The exceptional-permutation framework is
taken from Guralnick--M\"uller--Saxl \cite{GMS2003} and
Guralnick--Tucker--Zieve \cite{GTZ2007}.  For the defining-characteristic
$\operatorname{PSL}_2/\operatorname{PGL}_2$ cases, we use Giudici's
maximal-subgroup classification, including the cases that are $k$-indecomposable but geometrically decomposable
\cite{Giudici2007}.  These results constrain the geometric monodromy and
stabilizers, but they do not determine the descended $k$-M\"obius classes,
the equality relation between such classes, $k$-indecomposability, or the exact
extension arithmetic.  Those are the finite-field descent problems resolved here.
Fried \cite{Fried2005} places exceptional covers in a broader Diophantine
framework, including their tower structure over finite fields.

Explicit and low-degree exceptional rational functions have been studied in
several settings \cite{FerragutiMicheli2020,DingZieve2022,DingZieve2023,Valentini2017,Valentini2019,TangFan2026Degree5}.
Section~\ref{sec:historical-recovery} gives the item-by-item comparison and
separates independent recoveries and structural strengthenings from statements
of compatibility with known normal forms.  More specialized fixed-degree and
positive-genus comparisons are deferred to that section and to
Section~\ref{sec:conclusion}.

For orientation, Sections~\ref{sec:semilinear-data}--\ref{sec:primitivity-decomposition}
develop the descent, equality, decomposition, and ramification formalism;
Sections~\ref{sec:tame-basic}--\ref{sec:wild-novelty} prove the classification;
and Sections~\ref{sec:no-accidental}--\ref{sec:arithmetic-closure} prove the
permutation--exceptionality theorem and its exact arithmetic consequences.
Section~\ref{sec:historical-recovery} contains the historical comparisons.

\section{Semilinear data and \texorpdfstring{$k$}{k}-M\"obius classes}\label{sec:semilinear-data}
\subsection{Arithmetic and geometric Galois data}

Whenever objects carry a specified structure over a field $K$, the notation $X\cong_KY$ means an isomorphism over $K$. For curves over $K$ it denotes an isomorphism of $K$-curves, while for $K$-algebras it denotes an isomorphism fixing $K$ pointwise; in particular, this convention applies to function fields. We use the unsubscripted symbol $\cong$ for abstract isomorphisms for which no base field is part of the structure, such as isomorphisms of finite groups, quotient groups, or permutation sets.

\begin{proposition}[The inseparable indecomposable boundary]
Let $k=\mathbb F_q$ have characteristic $p$, and let $f\in k(X)$ have degree greater than $1$. If $f$ is inseparable and $k$-indecomposable, then $f=\mu\circ X^p$ for some $\mu\in\operatorname{PGL}_2(k)$.
\end{proposition}

\begin{proof}
Since $k$ is perfect, inseparability is equivalent to $f'(X)=0$, and hence $f\in k(X^p)$. Thus $f=g\circ X^p$ for some $g\in k(X)$. If $\deg g>1$, this is a nontrivial decomposition over $k$. Therefore $k$-indecomposability forces $\deg g=1$, so $g=\mu\in\operatorname{PGL}_2(k)$.
\end{proof}

The remainder of the paper concerns the separable case.

\begin{theorem}[Semilinear data and $k$-M\"obius classes]
\label{thm:reconstruction}
For every semilinear $\mathbf P^1$-datum
$\mathcal D=(\kappa/k,C,G,H,\varphi)$ over $k$, every representative of
$\mathcal M_{\mathcal D}$ is separable, has Galois closure of genus zero
with full constant field $\kappa$, and has degree $[G:H]$.

Conversely, if $f\in k(X)$ is separable and its Galois closure has genus
zero, then $f\in\mathcal M_{\mathcal D}$ for some semilinear
$\mathbf P^1$-datum $\mathcal D$ over $k$.
\end{theorem}

\begin{proof}
For the forward implication put $\Omega=\kappa(C)$. By construction,
$k(Y_{\mathcal D})=\Omega^{A}$ and $k(X_{\mathcal D})=\Omega^{U}$. The
extension $\Omega/\Omega^{A}$ is finite Galois with group $A$. Since
$\operatorname{core}_A(U)=1$, the arithmetic Galois closure of
$\Omega^{U}/\Omega^{A}$ is $\Omega$. Its full constant field is $\kappa$,
and its smooth projective curve is $C$, which has genus zero. Moreover
$A=GU$ and $G\cap U=H$, so
\[
 \deg \gamma_{\mathcal D}=[A:U]=[G:H].
\]
These assertions are unchanged by choosing a different representative of
$\mathcal M_{\mathcal D}$.

Conversely, put $\mathbf t=f(\mathbf x)$ and let
$\Omega_f/k(\mathbf t)$ be the arithmetic Galois closure of
$k(\mathbf x)/k(\mathbf t)$. Let $\kappa_f$ be the full constant field of
$\Omega_f$ and let $C_f/\kappa_f$ be the smooth projective curve with function
field $\Omega_f$. Put
\[
\begin{alignedat}{2}
 A_f&=\operatorname{Gal}(\Omega_f/k(\mathbf t)),\qquad
 &G_f&=\operatorname{Gal}(\Omega_f/\kappa_f(\mathbf t)),\\
 U_f&=\operatorname{Gal}(\Omega_f/k(\mathbf x)),
 &H_f&=\operatorname{Gal}(\Omega_f/\kappa_f(\mathbf x)).
\end{alignedat}
\]
Hence $\Omega_f^{H_f}=\kappa_f(\mathbf x)$, and base extension preserves the
degree, so
\[
 [G_f:H_f]
 =[\kappa_f(\mathbf x):\kappa_f(\mathbf t)]
 =[k(\mathbf x):k(\mathbf t)]
 =[A_f:U_f].
\]
Thus $[U_f:H_f]=[A_f:G_f]=[\kappa_f:k]$, and the restriction map
$U_f\to A_f/G_f\cong\operatorname{Gal}(\kappa_f/k)$ is surjective with
kernel $H_f$. Choose $\varphi_f\in U_f$ inducing the $q$-Frobenius on
$\kappa_f$. Since $\Omega_f$ is the Galois closure of
$k(\mathbf x)/k(\mathbf t)$, $\operatorname{core}_{A_f}(U_f)=1$.
Restriction to $\kappa_f$ gives
$U_f/H_f\cong A_f/G_f\cong\operatorname{Gal}(\kappa_f/k)$, and hence
\[
 A_f=\langle G_f,\varphi_f\rangle,\qquad
 U_f=\langle H_f,\varphi_f\rangle,\qquad
 \varphi_f^d\in H_f.
\]
The hypothesis on the Galois closure shows that $C_f$ is isomorphic over
$\kappa_f$ to $\mathbf P^1_{\kappa_f}$. Thus
$\mathcal D_f=(\kappa_f/k,\allowbreak C_f,\allowbreak G_f,\allowbreak H_f,\allowbreak\varphi_f)$ is a semilinear
$\mathbf P^1$-datum over $k$. Finally,
$\Omega_f^{U_f}=k(\mathbf x)$ and $\Omega_f^{A_f}=k(\mathbf t)$, so the
descended morphism represents $f$ and therefore
$f\in\mathcal M_{\mathcal D_f}$.
\end{proof}

\begin{remark}
For a semilinear $\mathbf P^1$-datum over $k$, the Frobenius lift $\varphi$ need
not have order $d=[\kappa:k]$; equivalently, the exact sequence
\[
 1\longrightarrow H\longrightarrow U
 \longrightarrow\operatorname{Gal}(\kappa/k)\longrightarrow1
\]
need not split.
\end{remark}

\begin{theorem}[Equality of $k$-M\"obius classes]
\label{thm:exact-map-equivalence}
For $i=1,2$, let $\mathcal D_i=(\kappa_i/k,\allowbreak C_i,\allowbreak G_i,\allowbreak H_i,\allowbreak\varphi_i)$ be
semilinear $\mathbf P^1$-data over $k$, and put $d_i=[\kappa_i:k]$. Then
$\mathcal M_{\mathcal D_1}=\mathcal M_{\mathcal D_2}$ if and only if
$d_1=d_2$ and, writing $\kappa_1=\kappa_2=:\kappa$, there is a
$\kappa$-isomorphism $\alpha:C_1\to C_2$ such that
\[
 \alpha G_1\alpha^{-1}=G_2,\qquad
 \alpha H_1\alpha^{-1}=H_2,\qquad
 \alpha\varphi_1\alpha^{-1}\in H_2\varphi_2.
\]
\end{theorem}

\begin{proof}
Suppose first that $\mathcal M_{\mathcal D_1}=\mathcal M_{\mathcal D_2}$.
Choose representatives and an isomorphism of the corresponding $k$-cover
squares. By Theorem~\ref{thm:reconstruction}, their arithmetic Galois closures
are $\Omega_i=\kappa_i(C_i)$. The cover isomorphism lifts to an isomorphism of
the normal closures. The algebraic closure of $k$ in a normal closure is
intrinsic, so $d_1=d_2$. Since $\kappa_1$ and $\kappa_2$ are finite
subextensions of the fixed field $\bar k/k$ of the same degree, they are
equal; write $\kappa_1=\kappa_2=:\kappa$. The lifted cover isomorphism induces
a $k$-automorphism of $\kappa$, hence a power of Frobenius. Composing the lift
with the corresponding power of $\varphi_2^{-1}$ gives a $\kappa$-isomorphism
$\alpha:C_1\to C_2$ without changing the descended cover square.

The lifted square gives
\[
 \alpha A_1\alpha^{-1}=A_2,\qquad
 \alpha U_1\alpha^{-1}=U_2.
\]
Since $G_i=\ker(A_i\to\operatorname{Gal}(\kappa/k))$ and
$H_i=G_i\cap U_i$, one obtains $\alpha G_1\alpha^{-1}=G_2$ and
$\alpha H_1\alpha^{-1}=H_2$. The source quotient isomorphism is defined over
$k$, so $\alpha\varphi_1\alpha^{-1}$ and $\varphi_2$ induce the same
semilinear action on $C_2/H_2$. Their quotient is therefore an automorphism of the cover
$C_2\to C_2/H_2$, hence an element of $H_2$, proving
$\alpha\varphi_1\alpha^{-1}\in H_2\varphi_2$.

Conversely, suppose the displayed conditions hold. The conjugacy relations
let $\alpha$ descend to a commutative square of the quotient morphisms
\[
 C_1/H_1\longrightarrow C_1/G_1,
 \qquad
 C_2/H_2\longrightarrow C_2/G_2.
\]
Write $\alpha\varphi_1\alpha^{-1}=h\varphi_2$ with $h\in H_2$. Since $h$
acts trivially on $C_2/H_2$, the source quotient isomorphism intertwines the
two Frobenius descents; because $H_2\leqslant G_2$, the same holds on the target.
Hence the quotient square descends to $k$, giving
$\mathcal M_{\mathcal D_1}=\mathcal M_{\mathcal D_2}$.
\end{proof}

\begin{corollary}[Compatibility with Galois-closure data]
\label{cor:f-datum-compatibility}
Let $f\in k(X)$ be separable with Galois closure of genus zero, and let
$\kappa_f$ be its full constant field. Let
$\Omega_f,C_f,A_f,G_f,U_f,H_f$ be the Galois-closure objects attached to
$f$, and choose a Frobenius lift $\varphi_f\in U_f$. If
$\mathcal D=(\kappa/k,C,G,H,\varphi)$ is a semilinear $\mathbf P^1$-datum
over $k$ with $f\in\mathcal M_{\mathcal D}$, then $\kappa_f=\kappa$. There
is a $\kappa$-isomorphism $\alpha:C_f\to C$ such that
\[
 \alpha G_f\alpha^{-1}=G,\qquad
 \alpha H_f\alpha^{-1}=H,\qquad
 \alpha\varphi_f\alpha^{-1}\in H\varphi.
\]
Consequently $\alpha A_f\alpha^{-1}=A$ and $\alpha U_f\alpha^{-1}=U$,
where $A=\langle G,\varphi\rangle$ and $U=\langle H,\varphi\rangle$; the
induced function-field isomorphism identifies $\Omega_f$ with
$\Omega=\kappa(C)$.
\end{corollary}

\begin{proof}
The chosen Frobenius lift defines the semilinear $\mathbf P^1$-datum
$\mathcal D_f=(\kappa_f/k,\allowbreak C_f,\allowbreak G_f,\allowbreak H_f,\allowbreak\varphi_f)$ over $k$, and the converse
direction of Theorem~\ref{thm:reconstruction} gives
$f\in\mathcal M_{\mathcal D_f}$. Hence
$\mathcal M_{\mathcal D_f}=\mathcal M_{\mathcal D}$. Applying
Theorem~\ref{thm:exact-map-equivalence} gives $\kappa_f=\kappa$, the asserted
$\kappa$-isomorphism, and the three conjugacy relations. The identities for
$A_f$ and $U_f$ follow from $A_f=\langle G_f,\varphi_f\rangle$ and
$U_f=\langle H_f,\varphi_f\rangle$.
\end{proof}

\begin{remark}
The subgroup $H_2$ in the Frobenius-coset condition is sharp: replacing
it by $G_2$ would ensure compatibility on the target quotient, but not
on the source quotient.
\end{remark}

\subsection{Cohomological classification for fixed geometric data}

Fix a semilinear $\mathbf P^1$-datum
$\mathcal D_0=(\kappa/k,C,G,H,\varphi_0)$ over $k$ with
$\kappa=\mathbb F_{q^d}$, and put $\Omega=\kappa(C)$.
For a subgroup $J\leqslant L$, write $\mathrm N_L(J)$ for the normalizer of $J$
in $L$. Define the pair normalizer and its twisting quotient by
\[
 \mathscr N
 :=\mathrm N_{\operatorname{Aut}_{\kappa}(C)}(G)
  \cap \mathrm N_{\operatorname{Aut}_{\kappa}(C)}(H),
 \qquad
 \mathscr W:=\mathscr N/H.
\]
Let $\theta$ be the automorphism of $\mathscr W$ induced by conjugation by
$\varphi_0$. Since $\varphi_0^d\in H$, one has $\theta^d=1$. Let $C_d$
denote the cyclic group of order $d$, acting on $\mathscr W$ through $\theta$, and
define
\[
\begin{aligned}
 Z^1_\theta(\mathscr W)
 &:=\{c\in \mathscr W:c\theta(c)\cdots\theta^{d-1}(c)=1\},\\
 c\sim_\theta c'
 &\Longleftrightarrow
 c'=u\,c\,\theta(u)^{-1}
 \quad\text{for some }u\in \mathscr W,\\
 H^1_\theta(C_d,\mathscr W)
 &:=Z^1_\theta(\mathscr W)/{\sim_\theta}.
\end{aligned}
\]

Elements of $Z^1_\theta(\mathscr W)$ are called \emph{$1$-cocycles}.
The relation $\sim_\theta$ is \emph{$\theta$-twisted conjugacy}, and
$H^1_\theta(C_d,\mathscr W)$ is the corresponding
\emph{nonabelian cohomology set}.

For $c\in Z^1_\theta(\mathscr W)$ choose a lift $\widetilde c\in \mathscr N$ and put
\[
 \psi_{\widetilde c}:=\widetilde c\,\varphi_0,
 \qquad
 A_c:=\langle G,\psi_{\widetilde c}\rangle,
 \qquad
 U_c:=\langle H,\psi_{\widetilde c}\rangle.
\]
The groups $A_c$ and $U_c$ do not depend on the chosen lift. Let $e(c)$ be
the order of the class of $\operatorname{Int}(\psi_{\widetilde c})|_G$ in
\[
 \mathrm N_{\operatorname{Aut}(G)}(H)
 \Big/
 \{\operatorname{Int}(h)|_G:h\in H\}.
\]
Here $\mathrm N_{\operatorname{Aut}(G)}(H)$ denotes the subgroup of
$\operatorname{Aut}(G)$ consisting of automorphisms carrying $H$ onto itself.

\begin{proposition}[Arithmetic core and full constant field]
\label{prop:fixed-pair-core-order}
The integer $e(c)$ is independent of the chosen lift and is constant on the
twisted-conjugacy class of $c$. Moreover,
\[
 \operatorname{core}_{A_c}(U_c)=C_{A_c}(G)\cap U_c,
 \qquad
 \bigl|\operatorname{core}_{A_c}(U_c)\bigr|=\frac d{e(c)}.
\]
The candidate descent defined by $\psi_{\widetilde c}$ has arithmetic Galois
closure with full constant field $\mathbb F_{q^{e(c)}}$. Consequently
$\mathcal D(\widetilde c):=(\kappa/k,C,G,H,\psi_{\widetilde c})$ is a
semilinear $\mathbf P^1$-datum over $k$ in the sense of
Definition~\ref{def:semilinear-P1-datum} if and only if $e(c)=d$.
\end{proposition}

\begin{proof}
Changing $\widetilde c$ by an element of $H$ changes
$\operatorname{Int}(\psi_{\widetilde c})|_G$ by an inner automorphism induced
by $H$, so $e(c)$ is lift-independent. Twisted conjugacy conjugates the
induced class in the displayed quotient, and hence preserves its order.

Put $N=\operatorname{core}_{A_c}(U_c)$. Since $N\cap G\triangleleft G$ and
$N\cap G\leqslant H$, and $H$ is core-free in $G$, one has $N\cap G=1$. As $N$
and $G$ are normal in $A_c$, one has $[N,G]\leqslant N\cap G=1$, so
$N\leqslant C_{A_c}(G)\cap U_c$. Conversely, if
$x\in C_{A_c}(G)\cap U_c$, then $A_c=GU_c$, and for $a=gu$ one has
$a^{-1}xa=u^{-1}xu\in U_c$. Thus $x$ belongs to every conjugate of $U_c$,
proving $N=C_{A_c}(G)\cap U_c$.

The homomorphism
\[
 U_c/H
 \longrightarrow
 \mathrm N_{\operatorname{Aut}(G)}(H)
 \Big/
 \{\operatorname{Int}(h)|_G:h\in H\}
\]
induced by conjugation has image generated by the class of
$\psi_{\widetilde c}$ and therefore has image order $e(c)$. A coset $uH$
lies in its kernel exactly when
$\operatorname{Int}(u)|_G=\operatorname{Int}(h)|_G$ for some $h\in H$; then
$h^{-1}u\in C_{A_c}(G)\cap U_c=N$. Conversely every element of $N$ gives
such a kernel coset. Since $N\cap H\leqslant N\cap G=1$, this identifies the
kernel with $N$. Hence $|N|=d/e(c)$. Since $N\cap G=1$, restriction embeds
$N$ into $\operatorname{Gal}(\kappa/k)$. Hence the full constant field of
the arithmetic Galois closure $\Omega^N$ has degree $d/|N|=e(c)$ over $k$.
The final assertion follows from
$\operatorname{core}_{A_c}(U_c)=1\iff e(c)=d$.
\end{proof}

For a twisted-conjugacy class $[c]$, write $e([c])$ for this common value.

\begin{theorem}[Cohomological classification for fixed geometric data]
\label{thm:fixed-pair-cohomology}
For $[c]\in H^1_\theta(C_d,\mathscr W)$ with $e([c])=d$, choose any representative
and any lift $\widetilde c\in \mathscr N$. Then
\[
 [c]\longmapsto\mathcal M_{\mathcal D(\widetilde c)}
\]
is a bijection from $\{[c]\in H^1_\theta(C_d,\mathscr W):e([c])=d\}$ onto the
$k$-M\"obius classes determined by semilinear $\mathbf P^1$-data over $k$
of the form $(\kappa/k,C,G,H,\psi)$ with fixed $C,G,H$. The displayed class
is independent of all choices.
\end{theorem}

\begin{proof}
If the lift $\widetilde c$ is replaced by $h\widetilde c$ with
$h\in H$, then the induced semilinear actions on both $C/H$ and $C/G$
are unchanged, so the resulting $k$-M\"obius class is unchanged.

If $c'$ is twisted-conjugate to $c$, choose a lift of the conjugating
element in $\mathscr N$.  The corresponding semilinear data are conjugate
up to multiplication of the Frobenius lift by an element of $H$.
Theorem~\ref{thm:exact-map-equivalence} therefore gives the same
$k$-M\"obius class.

Conversely, any semilinear datum with the displayed fixed $C,G,H$ has
Frobenius lift $\psi$ with
$\psi\varphi_0^{-1}\in \mathscr N$, hence arises from a $1$-cocycle in $Z^1_\theta(\mathscr W)$; its
arithmetic core-freeness forces $e=d$ by
Proposition~\ref{prop:fixed-pair-core-order}.  Finally, if two resulting
$k$-M\"obius classes are equal, Theorem~\ref{thm:exact-map-equivalence}
gives a $\kappa$-automorphism of $C$ normalizing both $G$ and $H$.
Its class in $\mathscr W$ realizes exactly the twisted-conjugacy relation.
\end{proof}

\begin{corollary}[Burnside count for fixed geometric data]
The number of these $k$-M\"obius classes is
\[
 \frac1{|\mathscr W|}
 \sum_{u\in \mathscr W}
 \left|
 \left\{
 c\in Z^1_\theta(\mathscr W):
 e(c)=d,\;
 u\,c\,\theta(u)^{-1}=c
 \right\}
 \right|.
\]
\end{corollary}

\section{Decomposition, monodromy, and ramification}\label{sec:primitivity-decomposition}

For a finite group $B$ and a proper subgroup $V\lneq B$, we call the natural
action $B\curvearrowright B/V$ \emph{primitive} if $V$ is a maximal proper
subgroup of $B$, equivalently, if there is no subgroup $J$ with
$V\lneq J\lneq B$.

For a semilinear $\mathbf P^1$-datum
$\mathcal D=(\kappa/k,C,G,H,\varphi)$, put
$A=\langle G,\varphi\rangle$ and $U=\langle H,\varphi\rangle$. We call the
natural actions $G\curvearrowright G/H$ and $A\curvearrowright A/U$ the
\emph{geometric monodromy action} and the \emph{arithmetic monodromy action},
respectively.

\begin{theorem}[Arithmetic decomposition criterion]
\label{thm:decomposition}
Let $\mathcal D=(\kappa/k,C,G,H,\varphi)$ be a semilinear
$\mathbf P^1$-datum over $k$, let $f\in\mathcal M_{\mathcal D}$, and put
$A=\langle G,\varphi\rangle$ and $U=\langle H,\varphi\rangle$. The following
are equivalent:
\begin{enumerate}[label=\textup{(\roman*)},ref=\roman*]
\item\label{item:decomposition-1} $f$ is decomposable over $k$;
\item\label{item:decomposition-2} there is a $\varphi$-stable subgroup $J$ with $H\lneq J\lneq G$;
\item\label{item:decomposition-3} $U$ is not maximal in $A$.
\end{enumerate}
Consequently, $f$ is $k$-indecomposable if and only if the arithmetic
monodromy action of $A$ on $A/U$ is primitive.
\end{theorem}

\begin{proof}
Put $\Omega=\kappa(C)$, $K=\Omega^A$, and $L=\Omega^U$. A nontrivial
decomposition of a representative $f\in\mathcal M_{\mathcal D}$ is equivalent
to a strict intermediate field $K\subsetneq M\subsetneq L$. Since $\Omega/K$
is Galois with group $A$, such fields correspond to subgroups $U\lneq V\lneq A$. For
such $V$, put $J=V\cap G$. Since $A=GU$ and $U\leqslant V$, every $v\in V$ can
be written $v=gu$ with $g\in G$ and $u\in U$; then $g=vu^{-1}\in V\cap G=J$.
Hence $V=JU=\langle J,\varphi\rangle$, so $J$ is $\varphi$-stable and
$H\lneq J\lneq G$.

Conversely, if $H\lneq J\lneq G$ is $\varphi$-stable, then
$V=\langle J,\varphi\rangle=JU$ satisfies $U\lneq V\lneq A$. The intermediate field
$\Omega^V$ lies strictly between $K$ and $L$. Since $L$ is a rational
function field over $k$, L\"uroth's theorem gives a functional decomposition
over $k$. This proves
\textup{(\ref{item:decomposition-1})}$\Longleftrightarrow$\textup{(\ref{item:decomposition-2})}, and the same correspondence proves
\textup{(\ref{item:decomposition-2})}$\Longleftrightarrow$\textup{(\ref{item:decomposition-3})}.
\end{proof}

\begin{corollary}[Arithmetic and geometric indecomposability]
\label{cor:geoarith}
Let $\mathcal D$ and $f\in\mathcal M_{\mathcal D}$ be as above.
Then $f$ is geometrically indecomposable if and only if $H$ is maximal
in $G$.  Hence geometric indecomposability implies
$k$-indecomposability.  The converse fails exactly when $H$ is not
maximal in $G$ but $U$ is maximal in $A$.
\end{corollary}

A decomposition $f=f_r\circ\cdots\circ f_1$ over $k$ is understood to have
$\deg f_i>1$ for every $i$.  Two
decompositions $(f_1,\ldots,f_r)$ and $(f_1',\ldots,f_r')$ are
identified if there are $\mu_i\in\operatorname{PGL}_2(k)$ with
$\mu_0=\mu_r=1$ and
\[
 f_i'=\mu_i\circ f_i\circ\mu_{i-1}^{-1}
 \qquad(1\leqslant i\leqslant r).
\]
A decomposition is \emph{complete} if every factor is
$k$-indecomposable.

\begin{theorem}[Functional decompositions and Frobenius-stable chains]
\label{thm:complete-decomposition}
Let $\mathcal D=(\kappa/k,C,G,H,\varphi)$ be a semilinear
$\mathbf P^1$-datum over $k$, let $f\in\mathcal M_{\mathcal D}$, and put
$\Omega=\kappa(C)$. Decompositions $f=f_r\circ\cdots\circ f_1$ over $k$,
modulo $k$-M\"obius changes of the intermediate coordinates, are in bijection
with strict chains
\[
 H=J_0\lneq J_1\lneq\cdots\lneq J_r=G
\]
of $\varphi$-stable subgroups. For the corresponding decomposition,
\[
 \deg f_i=[J_i:J_{i-1}].
\]
The decomposition is complete if and only if, for every $i$, there is no
$\varphi$-stable subgroup strictly between $J_{i-1}$ and $J_i$.

For each $i$ put
\[
\begin{aligned}
 A_i&=\langle J_i,\varphi\rangle,
 &U_i&=\langle J_{i-1},\varphi\rangle,\\
 N_i&=\operatorname{core}_{A_i}(U_i),
 &K_i&=\operatorname{core}_{J_i}(J_{i-1}).
\end{aligned}
\]
Then $N_i\cap J_i=K_i$. The arithmetic Galois closure of the $i$th factor
inside $\Omega$ is $\Omega^{N_i}$. Its full constant field $\kappa_i$ satisfies
\[
 [\kappa_i:k]=\frac{d}{[N_i:K_i]}\mid d,
 \qquad d=[\kappa:k],
\]
and after extension of constants to $\kappa$ its geometric Galois closure has
function field $\Omega^{K_i}$. In particular every factor is separable and has
Galois closure of genus zero.
\end{theorem}

\begin{proof}
Put $A=\langle G,\varphi\rangle$ and $U=\langle H,\varphi\rangle$. For a
strict $\varphi$-stable chain $J_i$, set $V_i=\langle J_i,\varphi\rangle=J_iU$.
Then $U=V_0\lneq V_1\lneq\cdots\lneq V_r=A$. The corresponding intermediate fields
$M_i=\Omega^{V_i}$ form
\[
 k(X_{\mathcal D})=M_0
 \supsetneq M_1\supsetneq\cdots\supsetneq
 M_r=k(Y_{\mathcal D}).
\]
Each $M_i$ is a rational function field by L\"uroth's theorem, so choosing
$k$-coordinates produces a decomposition. Conversely, any intermediate-field
chain gives a subgroup chain $V_i$, with $J_i=V_i\cap G$.
Theorem~\ref{thm:decomposition} gives $\varphi$-stability and the inverse
correspondence. Different choices of intermediate coordinates differ exactly
by $k$-M\"obius transformations.

Since $J_i\cap U=H$, one has $[V_i:V_{i-1}]=[J_i:J_{i-1}]$, giving the degree
formula. Completeness is equivalent, by Theorem~\ref{thm:decomposition}, to
the absence of a $\varphi$-stable subgroup strictly between successive $J_i$.

For the $i$th factor the ambient arithmetic Galois group is $A_i=V_i$ and the
point stabilizer is $U_i=V_{i-1}$, so its arithmetic Galois closure inside
$\Omega$ is $\Omega^{N_i}$. Now $N_i\cap J_i\triangleleft J_i$ and
$N_i\cap J_i\leqslant J_{i-1}$, hence $N_i\cap J_i\leqslant K_i$. Conversely, $K_i$
is $\varphi$-stable, normal in $A_i$, and contained in $U_i$, hence
$K_i\leqslant N_i$. Thus $N_i\cap J_i=K_i$. Since $A_i\cap G=J_i$, one has
$N_i\cap G=N_i\cap J_i=K_i$.

Thus the kernel of the action of $N_i$ on $\kappa$ is $K_i$. If $\kappa_i$
denotes the full constant field of $\Omega^{N_i}$, then
$[\kappa:\kappa_i]=[N_i:K_i]$, and hence
$[\kappa_i:k]=d/[N_i:K_i]$. Because $K_i\leqslant N_i$ and $K_i\leqslant G$ fixes
$\kappa$, one has $\Omega^{N_i}\kappa\subseteq\Omega^{K_i}$. Moreover
\[
 [\Omega^{K_i}:\Omega^{N_i}]
 =[N_i:K_i]
 =[\kappa:\kappa_i]
 =[\Omega^{N_i}\kappa:\Omega^{N_i}],
\]
so the inclusion is an equality, $\Omega^{N_i}\kappa=\Omega^{K_i}$. Since
$\Omega^{K_i}$ is an intermediate field of the rational function field
$\kappa(C)$, L\"uroth's theorem shows that it is a rational function field
over $\kappa$. Genus is unchanged by finite constant extension, so the
arithmetic Galois-closure curve of the factor has genus zero. This proves the
final assertions.
\end{proof}

\subsection{Normal cores, ramification, and the wildness criterion}

Let $\mathcal D=(\kappa/k,C,G,H,\varphi)$ be a semilinear
$\mathbf P^1$-datum over $k$ attached to a separable rational function.
Recall that $\pi_{\mathcal D}:C/H\to C/G$ is the quotient morphism attached
to $\mathcal D$. Let $\pi_H:C\to C/H$ and $\pi_G:C\to C/G$ be the
quotient morphisms. Then $\pi_G=\pi_{\mathcal D}\circ\pi_H$.

For a geometric point $P\in C(\bar k)$, the \emph{inertia group} of
$\pi_G$ at $P$ is the stabilizer
\[
 G_P:=\{g\in G:g(P)=P\}.
\]
Here the geometric residue fields are all $\bar k$, so the stabilizer
acts trivially on the residue field. The inertia groups above any one
branch point are conjugate in $G$.

For a finite separable morphism $\pi:X\to Y$ of smooth projective
geometrically irreducible curves over a field $F$, let $d_P(\pi)$ denote
the \emph{different exponent} at $P\in X(\overline F)$, namely the
valuation at $P$ of the different ideal of the corresponding extension
of discrete valuation rings after base change to $\overline F$.
The Riemann--Hurwitz formula gives
\[
 2g(X)-2
 =\deg(\pi)\bigl(2g(Y)-2\bigr)
  +\sum_{P\in X(\overline F)}d_P(\pi).
\]
Moreover, $d_P(\pi)=e_P(\pi)-1$ at a tamely ramified point,
whereas $d_P(\pi)\geqslant e_P(\pi)$ at a wildly ramified point;
see \cite[Section~53.12, especially Lemma~53.12.4]{Stacks2026}.
In particular, if $X$ and $Y$ both have genus zero, then
\[
 2\deg(\pi)-2=\sum_{P\ \mathrm{ramified}}d_P(\pi).
\]

\begin{theorem}[Structural consequences of the genus-zero datum]\label{thm:structural-consequences}
With the notation above, the following hold.
\begin{enumerate}[label=(\arabic*),ref=\arabic*,font=\itshape]
\item\label{item:structural-consequences-1} For every $P\in C(\bar k)$,
\begin{equation}\label{eq:quotient-ramification-index}
 e_{\pi_H(P)}(\pi_{\mathcal D})=[G_P:H\cap G_P].
\end{equation}
\item\label{item:structural-consequences-2} The morphism $\pi_{\mathcal D}$ is wild if and only if $p\mid |G|$.
\item\label{item:structural-consequences-3} If the arithmetic monodromy action is primitive, then either $H=1$ and $\mathrm N_G(H)=G$, or $H\ne1$ and $\mathrm N_G(H)=H$.
\item\label{item:structural-consequences-4} If the arithmetic monodromy action is primitive and $1\ne N\triangleleft G$ is $\varphi$-stable, then $NH=G$. In particular, $\operatorname{Soc}(G)H=G$, where
$\operatorname{Soc}(G)$, the \emph{socle} of $G$, is the subgroup generated
by the minimal nontrivial normal subgroups of $G$.
\end{enumerate}
\end{theorem}

\begin{proof}
For (\ref{item:structural-consequences-1}), the ramification indices of
$\pi_H$ and $\pi_G$ at $P$ are $|H\cap G_P|$ and $|G_P|$, respectively.
Multiplicativity in the factorization
$\pi_G=\pi_{\mathcal D}\circ\pi_H$ gives the asserted formula.

For (\ref{item:structural-consequences-2}), if $\pi_{\mathcal D}$ is wild, then part~(\ref{item:structural-consequences-1}) gives $p\mid |G_P|$ for some
$P\in C(\bar k)$, and hence $p\mid |G|$. Conversely, suppose $p\mid |G|$.
By the observation after Definition~\ref{def:semilinear-P1-datum},
$H$ is core-free in $G$. Let $S$ be a Sylow
$p$-subgroup of $G$. Some $G$-conjugate of $S$ is not contained in $H$;
otherwise the nontrivial normal subgroup generated by all conjugates of $S$
would lie in $\operatorname{core}_G(H)$. Replace $S$ by such a conjugate.

A nontrivial finite $p$-subgroup of $\operatorname{PGL}_2(\bar k)$ in
characteristic $p$ has a unique common fixed point $P$. Thus $S\leqslant G_P$, and
since $S$ is Sylow in $G$, it is Sylow in $G_P$. After moving $P$ to infinity,
$G_P$ is a finite subgroup of the affine group. The kernel of the homomorphism
$G_P\to\bar k^{\times}$ sending $(z\mapsto az+b)$ to $a$
is its unique Sylow $p$-subgroup, so
this subgroup is $S$. If $p\nmid [G_P:H\cap G_P]$, then $H\cap G_P$
contains $S$, contradicting $S\nleq H$. Hence
$p\mid [G_P:H\cap G_P]$, and part~(\ref{item:structural-consequences-1})
shows that $\pi_{\mathcal D}$ is wild.

For (\ref{item:structural-consequences-3}), $\mathrm N_G(H)$ is $\varphi$-stable, so primitivity gives $H$ or $G$; in the latter case core-freeness forces $H=1$. For (\ref{item:structural-consequences-4}), $NH$ is $\varphi$-stable. Since $N\leqslant H$ would force $N\leqslant\operatorname{core}_G(H)=1$, one has $H\lneq NH$, and primitivity gives $NH=G$.
Since $\operatorname{Soc}(G)$ is characteristic in $G$, it is
$\varphi$-stable; it is nontrivial because $G\ne1$. Taking
$N=\operatorname{Soc}(G)$ proves the final assertion.
\end{proof}

\begin{proposition}[Canonical normal-subgroup factorization]
Let $1\ne N\triangleleft G$ be $\varphi$-stable and put $J=NH$. Then the
descended rational function factors over $k$ as
$f=f_{\mathrm{res},N}\circ f_N$, and
\[
 \Omega^J=\Omega^N\cap\Omega^H,
 \qquad
 \deg f_N=[N:N\cap H],
 \qquad
 \deg f_{\mathrm{res},N}=[G:J].
\]
Both factors have Galois closure of genus zero.
\end{proposition}

\begin{proof}
The fixed field of the subgroup generated by $N$ and $H$ is the intersection of their fixed fields. Since $J$ is $\varphi$-stable, it descends to a $k$-defined intermediate field and hence to a factorization. The degree formulas are immediate from Galois correspondence. The geometric normal closure of $C/H\to C/J$ is $C/\operatorname{core}_J(H)$, while that of $C/J\to C/G$ is $C/\operatorname{core}_G(J)$. Both are quotients of $C$, which is isomorphic over $\kappa$ to
$\mathbf P^1_{\kappa}$, so their function fields are rational function fields
over $\kappa$ by L\"uroth's theorem.
\end{proof}

\subsection{Branch orbits and ramification of the quotient morphism \texorpdfstring{$\pi_{\mathcal D}$}{pi(D)}}

For a curve $X$ over a finite field $F$, we identify a \emph{closed point}
$P_0$ with its Frobenius orbit in $X(\overline F)$. Its degree is the
length of this orbit, equivalently $[\kappa(P_0):F]$, where $\kappa(P_0)$
is its residue field.

Retain the datum $\mathcal D$ and the quotient-morphism notation from the preceding subsection. Put
\[
 \mathcal R_G=\{P\in C(\bar k):G_P\ne1\},\qquad \mathcal B_G=G\backslash\mathcal R_G.
\]
Thus $\mathcal R_G$ is the ramification locus of the Galois quotient
$C\to C/G$, and $\mathcal B_G$ parametrizes its geometric branch points.

\begin{theorem}[Branch orbits and ramification partitions]\label{thm:branch-orbit}
The geometric branch locus of the quotient morphism $\pi_{\mathcal D}:C/H\to C/G$ is exactly the branch locus of $\pi_G:C\to C/G$. Its geometric branch points are therefore parametrized by $\mathcal B_G$. Frobenius acts by $[P]_G\mapsto[\varphi(P)]_G$, and the corresponding closed point of the branch divisor has degree equal to its Frobenius orbit length.

Fix $P\in\mathcal R_G$ and put $I=G_P$. The points of $C/H$ over the corresponding branch point are parametrized by $H\backslash G/I$, and at the point represented by $HgI$ the ramification index is
\begin{equation}\label{eq:double-coset-passport}
 e(HgI)=[I:I\cap g^{-1}Hg].
\end{equation}
The multiset of these indices, as $HgI$ ranges over $H\backslash G/I$,
is the ramification partition of the corresponding branch point.
\end{theorem}

\begin{proof}
Suppose a branch point of $\pi_G$, represented by $P$, were unramified at every point of $C/H$ above it. Equation~\eqref{eq:quotient-ramification-index} would give $I\leqslant g_0^{-1}Hg_0$ for every chosen representative $Hg_0I$. If $g=hg_0i$ with $h\in H$ and $i\in I$, then $g^{-1}Hg=i^{-1}g_0^{-1}Hg_0i$; since $i$ normalizes $I$, the same containment holds for every $g\in G$. Thus $I\leqslant\operatorname{core}_G(H)=1$, a contradiction. The converse is immediate from \eqref{eq:quotient-ramification-index}.

Since $\varphi$ normalizes $G$, it permutes the $G$-orbits in the ramification locus of $\pi_G:C\to C/G$ and induces arithmetic Frobenius on the descended branch divisor. Finally, the $\pi_G$-fiber is the transitive $G$-set $G/I$; passing to $H$-orbits gives $H\backslash G/I$, and \eqref{eq:double-coset-passport} follows from \eqref{eq:quotient-ramification-index}.
\end{proof}

\begin{corollary}[Assembly from indecomposable genus-zero factors]
Every separable rational function of degree greater than $1$ over a finite field whose Galois closure has genus zero is a finite composition of $k$-indecomposable rational functions with Galois closure of genus zero. Consequently, up to $k$-M\"obius equivalence, every factor in any complete decomposition belongs to the tame or wild classification of Theorems~\ref{thm:intro-tame} and~\ref{thm:intro-wild}.
\end{corollary}

\begin{example}[Complete decompositions need not be unique]\label{ex:A4-decomposition}
Let $k=\mathbb F_5$. Since $\operatorname{PGL}_2(5)\cong S_5$, choose a subgroup $G\cong A_4$ of $\operatorname{PGL}_2(k)$. Let $f:\mathbf P^1_k\to \mathbf P^1_k/G\cong_k\mathbf P^1_k$ be the
Galois quotient; the target is a genus-zero curve with a $k$-rational point, hence is $k$-isomorphic to $\mathbf P^1$. Here $H=1$, the cover has degree $12$, and the semilinear Frobenius lift is trivial. Thus every subgroup of $G$ is $\varphi$-stable. The subgroup lattice of $A_4$ contains the two maximal chains
\[
 1\lneq C_3\lneq A_4,
 \qquad
 1\lneq C_2\lneq V_4\lneq A_4.
\]
Theorem~\ref{thm:complete-decomposition} therefore yields complete decompositions whose multisets of factor degrees are, respectively,
\[
 \{3,4\}
 \qquad\text{and}\qquad
 \{2,2,3\}.
\]
All factors are separable, $k$-indecomposable, and have Galois closure of genus zero. Since degree is invariant under $k$-M\"obius equivalence, these two complete decompositions cannot become equivalent after reordering their factors. Thus neither the number of indecomposable factors nor the unordered multiset of their $k$-M\"obius classes is determined by $f$.
\end{example}

\section{Tame classification: reduction and basic cases}\label{sec:tame-basic}
\subsection{Group-theoretic reduction}

Throughout this section, $G$ is a finite subgroup of $\operatorname{PGL}_2(\bar k)$ with $p\nmid|G|$. The Klein--Dickson classification gives $G\cong C_m,D_{2m},A_4,S_4$, or $A_5$; see Beauville \cite{Beauville2010}. For $H\lneq G$, the transitive coset action $G\curvearrowright G/H$ is faithful exactly when $H$ is core-free in $G$, and primitive exactly when $H$ is maximal in $G$. We begin by classifying these faithful primitive pairs under the tame condition $p\nmid|G|$.

\begin{theorem}[Faithful primitive pairs with $p\nmid|G|$]\label{thm:tamepairs}
Let $G\leqslant\operatorname{PGL}_2(\bar k)$ be finite with $p\nmid|G|$, and let $H$ be a maximal core-free subgroup of $G$. Then, up to isomorphism of pairs, exactly the following possibilities occur:
\[
\begin{array}{c|c|c}
G&H&[G:H]\\ \hline
C_\ell&1&\ell\\
D_{2\ell}&C_2&\ell\\
A_4&C_3&4\\
S_4&S_3&4\\
A_5&A_4&5\\
A_5&D_{10}&6\\
A_5&S_3&10
\end{array}
\]
where $\ell$ is prime and, in the dihedral case, $\ell$ is odd.
\end{theorem}

\begin{proof}
If $G=C_m$, every subgroup is normal. Core-freeness forces $H=1$, and maximality then forces $m$ to be prime.

Let $G=D_{2m}=\langle r,s\mid r^m=s^2=1,\ srs=r^{-1}\rangle$. A core-free subgroup meets the normal rotation subgroup $\langle r\rangle$ trivially. A nontrivial core-free maximal subgroup is therefore a reflection subgroup $C_2$, and it is maximal exactly when $m$ is prime. The case $m=2$ gives $V_4$, in which every order-$2$ subgroup is normal, so $m$ must be odd.

For $A_4$, the maximal subgroups are the normal $V_4$ and the four subgroups $C_3$; only $C_3$ is core-free.

For $S_4$, the maximal subgroups have types $A_4$, $D_8$, and $S_3$. The subgroup $A_4$ is normal. Every Sylow-$2$ subgroup $D_8$ contains the normal Klein four subgroup
\(V_4=\{1,(12)(34),(13)(24),(14)(23)\},\)
so
\(V_4\leqslant\operatorname{core}_{S_4}(D_8)\neq1.\)
Thus the only core-free maximal type is $S_3$, of index $4$.

For $A_5$, the maximal subgroups have types $A_4,D_{10},S_3$, of indices $5,6,10$. Since $A_5$ is simple, each is core-free.
\end{proof}

The next theorem identifies the only gap between geometric and arithmetic indecomposability.

\begin{theorem}[Unique tame geometrically decomposable case]\label{thm:arith-only}
Let $f\in k(X)$ be separable and $k$-indecomposable, with Galois closure of genus zero and $p\nmid|G_f|$.  Choose a Frobenius lift $\varphi_f\in U_f$.  If $f$ is geometrically decomposable, then necessarily $G_f\cong V_4$ and $H_f=1$, and conjugation by $\varphi_f$ acts as a $3$-cycle on the three order-$2$ subgroups of $V_4$. Conversely, such a datum has primitive arithmetic monodromy.

In arithmetic/geometric monodromy notation,
\[
 (A_f,G_f,U_f,H_f)\cong(A_4,V_4,C_3,1),
 \qquad
 [\kappa_f:k]=3,
 \qquad
 \deg f=4.
\]
\end{theorem}

\begin{proof}
By Theorem~\ref{thm:decomposition}, the arithmetic monodromy action
$A_f\curvearrowright A_f/U_f$ is primitive. Since $f$ is geometrically
decomposable, Corollary~\ref{cor:geoarith} shows that $H_f$
is not maximal in $G_f$; moreover, $H_f$ is core-free in $G_f$. Because
$p\nmid|G_f|$, the Klein--Dickson classification of finite subgroups of
$\operatorname{PGL}_2(\bar k)$, for which we use Beauville
\cite{Beauville2010} as in Theorem~\ref{thm:tamepairs}, gives
$G_f\cong C_m,D_{2m},A_4,S_4$, or $A_5$. We examine these ambient group
types and determine when a nonmaximal core-free subgroup $H_f$ can
nevertheless give primitive arithmetic monodromy.

For $G_f=C_m$, core-freeness forces $H_f=1$. If $m$ is composite, a nontrivial proper characteristic subgroup of $C_m$ is Frobenius-stable. Hence primitivity of the arithmetic monodromy action forces $m$ prime, and the rational function is already geometrically primitive.

Let $G_f=D_{2m}$ with $m>2$. A core-free $H_f$ has trivial intersection with the rotation subgroup $R=\langle r\rangle$, so $H_f$ is either $1$ or a reflection subgroup. If $H_f=1$, then the characteristic subgroup $R$ is a Frobenius-stable intermediate subgroup. If $H_f=\langle s\rangle$ and $\varphi_f(H_f)=H_f$, then after choosing $s$ to generate $H_f$ one has
\(\varphi_f(r)=r^a,\qquad \varphi_f(s)=s.\)
Every intermediate dihedral subgroup $\langle s,r^d\rangle$ is then $\varphi_f$-stable. Thus primitivity of the arithmetic monodromy action forces $H_f$ maximal. The only excluded endpoint is $m=2$, namely $V_4$.

For $G_f=V_4$, core-freeness forces $H_f=1$. The open interval $1\lneq J\lneq V_4$ consists of the three order-$2$ subgroups. Since
\(\operatorname{Aut}(V_4)\cong S_3,\)
there is no Frobenius-fixed intermediate subgroup exactly when Frobenius acts as a $3$-cycle. This gives the claimed $V_4/A_4$ case.

For $G_f=A_4$, the unique normal $V_4$ is characteristic. If $H_f=1$ or $H_f=C_2$, it supplies a Frobenius-stable intermediate subgroup. The remaining core-free primitive case is $H_f=C_3$, already maximal.

For $G_f=S_4$, every automorphism is inner. Suppose a nonmaximal core-free $H_f$ is stabilized by $\alpha=\operatorname{Inn}(a)$. If
\(H_f\lneq \mathrm N_{G_f}(H_f)\lneq G_f,\)
then $\mathrm N_{G_f}(H_f)$ is an $\alpha$-stable intermediate subgroup. If $\mathrm N_{G_f}(H_f)=H_f$, then $a\in H_f$, and every overgroup of $H_f$ is stabilized by conjugation by $a$; since $H_f$ is nonmaximal, one exists. If $\mathrm N_{G_f}(H_f)=G_f$, then $H_f\triangleleft G_f$, so core-freeness forces $H_f=1$, and a nontrivial proper cyclic subgroup generated by $a$ (or any proper subgroup if $a=1$) is stable. Thus no geometrically decomposable $k$-indecomposable case occurs here.

For $G_f=A_5$, an inner automorphism is excluded by the same normalizer argument. For an outer automorphism, identify $\operatorname{Aut}(A_5)$ with $S_5$. If $\alpha$ is odd and stabilizes $H_f$, put $\widehat U=\langle H_f,\alpha\rangle$. Frobenius-primitivity is equivalent to maximality of $\widehat U$ in $S_5$. The maximal subgroups of $S_5$ not contained in $A_5$ have types $S_4$, $S_3\times S_2$, and $\operatorname{AGL}_1(5)$, and their intersections with $A_5$ are respectively
\(A_4,\qquad S_3,\qquad D_{10},\)
all maximal in $A_5$. Hence the outer action gives no geometrically decomposable $k$-indecomposable case.

It remains to determine the arithmetic group in the $V_4$ case. Since $H_f=1$, the group $U_f\cong\operatorname{Gal}(\kappa_f/k)$ is cyclic. The arithmetic group acts faithfully on four sheets, so $A_f\leqslant S_4$. Conjugation by $U_f$ induces an automorphism of order $3$ on $V_4$, hence $3\mid |U_f|$. Since $S_4$ has no cyclic subgroup of order divisible by $3$ larger than $3$, one has $U_f=C_3$. Therefore
\(A_f=V_4\rtimes C_3\cong A_4,\)
and $[\kappa_f:k]=3$.
\end{proof}

\begin{corollary}\label{cor:tame-group-complete}
Every tame $k$-indecomposable rational function with Galois closure of genus zero is either geometrically primitive with one of the pairs in Theorem~\ref{thm:tamepairs}, or belongs to the unique $V_4/A_4$ case of Theorem~\ref{thm:arith-only}.
\end{corollary}

\subsection{The \texorpdfstring{$V_4/A_4$}{V4/A4} quartic}

We make Theorem~\ref{thm:arith-only} explicit.

\begin{proposition}[The $V_4/A_4$ quartic]\label{prop:V4-quartic}
Assume $q$ is odd. Let $h(X)=X^3+aX^2+bX+c\in k[X]$ be monic and irreducible, and define
\[
 L_h(X)=
 \frac{X^4-2bX^2-8cX+b^2-4ac}
 {4h(X)}.
\]
Then $L_h$ is separable of degree $4$, the automorphism group of its geometric cover is $V_4$, and its three branch points are the roots of $h$. Moreover, $(A,G,U,H)\cong(A_4,V_4,C_3,1)$. Thus $L_h$ is $k$-indecomposable but geometrically decomposable.

Conversely, every tame $k$-indecomposable rational function with Galois closure of genus zero that is geometrically decomposable is $k$-M\"obius equivalent to $L_h$. Consequently, for every odd $q$ there is exactly one such $k$-M\"obius class.
\end{proposition}

\begin{proof}
Consider the elliptic curve $E:y^2=h(x)$.
For $P=(x,y)$, the tangent slope is $h'(x)/(2y)$, and the duplication formula on the generalized short Weierstrass model gives
\(x(2P)=\frac{h'(x)^2-4h(x)(a+2x)}{4h(x)}.\)
Expanding the numerator yields
\(h'(x)^2-4h(x)(a+2x) =x^4-2bx^2-8cx+b^2-4ac,\)
so the induced morphism on $E/\{\pm1\}\cong_k\mathbf P^1_k$ is exactly $L_h$.

For a root $\alpha$ of $h$, let $T_\alpha=(\alpha,0)\in E[2]\setminus\{0\}$. Translation by $T_\alpha$ descends to the M\"obius involution \(\iota_\alpha(X)=\alpha+\frac{h'(\alpha)}{X-\alpha}\). The three involutions $\iota_\alpha$ are the nontrivial elements of a Klein four group. Since $[2](P+T_\alpha)=[2]P$, they are automorphisms of the geometric cover defined by $L_h$, and degree considerations show that they exhaust its automorphism group.

The morphism $[2]:E\to E$ is \'etale in odd characteristic, and
\(L_h\circ x=x\circ[2].\)
Comparison of ramification in this identity shows that the branch points of $L_h$ are precisely $x(T_\alpha)=\alpha$ for the three nonzero points of $E[2]$.

Because $h$ is irreducible cubic over a finite field, its roots form one Frobenius orbit of length $3$. Moreover, $\operatorname{Fr}_q\iota_\alpha\operatorname{Fr}_q^{-1}=\iota_{\alpha^q}$. Thus Frobenius acts as a $3$-cycle on $V_4\setminus\{1\}$, giving
\(A=V_4\rtimes C_3\cong A_4, \qquad U=C_3.\)
The subgroup $C_3$ is maximal in $A_4$, whereas $1$ is not maximal in $V_4$; the arithmetic and geometric decomposition assertions now follow from Theorem~\ref{thm:decomposition} and Corollary~\ref{cor:geoarith}.

For uniqueness, the three branch points of any rational function in this $V_4/A_4$ case form a single closed point of degree $3$ on $\mathbf P^1_k$. The group $\operatorname{PGL}_2(k)$ acts transitively on degree-$3$ closed points: if $\alpha$ and $\beta$ are representatives, the unique M\"obius transformation carrying $(\alpha,\alpha^q,\alpha^{q^2})$ to $(\beta,\beta^q,\beta^{q^2})$ is Frobenius-fixed and hence lies in $\operatorname{PGL}_2(k)$.

After aligning the branch divisors, the two $V_4$-covers are geometrically isomorphic. The remaining descent discrepancy is a $1$-cocycle with values in $V_4$ for the Frobenius action that cyclically permutes $V_4\setminus\{1\}$. Writing $V_4\cong\mathbb F_2^2$, this Frobenius operator has minimal polynomial $T^2+T+1$, so $F-1$ is invertible and
\[
 H^1(\langle F\rangle,V_4)=0.
\]
After composing with an automorphism of the geometric cover, the geometric isomorphism therefore descends. This proves uniqueness of the $k$-M\"obius class.
\end{proof}

\begin{remark}[Explicit geometric decomposition]
Over the splitting field of $h$, for any root $\alpha$ the degree-$2$ invariant
\[
 h_\alpha(X)=X+\iota_\alpha(X)
 =X+\alpha+\frac{h'(\alpha)}{X-\alpha}
\]
generates the fixed field of $\langle\iota_\alpha\rangle$. Thus $L_h$ factors geometrically as two quadratic rational functions. Frobenius cyclically permutes the three possible quadratic intermediate fields, so none descends individually to $k$.
\end{remark}

\subsection{Cyclic cases}

For a finite field $F$ and a Frobenius-stable unordered pair
$\{P_1,P_2\}$ of distinct geometric points on an $F$-curve, we call
the pair \emph{split over $F$} if both points are $F$-rational, and
\emph{nonsplit over $F$} if Frobenius exchanges them; equivalently,
in the latter case they form a single closed point of degree $2$.
We refer to this dichotomy as the \emph{splitting type} of the pair.

\begin{proposition}[Cyclic $k$-forms]\label{prop:cyclic}
Let $f\in k(X)$ be separable, let $n>1$ with $p\nmid n$, and suppose that
$G_f\cong C_n$. For $\delta\in\mathbb F_{q^2}\setminus\mathbb F_q$, put
$\nu_\delta(X)=(X-\delta^q)/(X-\delta)$. Up to $k$-M\"obius equivalence,
exactly one of the following occurs:
\begin{enumerate}[label=(\arabic*),ref=\arabic*,font=\itshape]
\item the two branch points are $k$-rational and $f(X)\sim_k X^n$;
\item the two branch points form one closed point of degree $2$ and
\[
 f(X)\sim_k R_{n,\delta}(X)
 =\nu_\delta^{-1}\!\left(\nu_\delta(X)^n\right)
\]
for any $\delta\in\mathbb F_{q^2}\setminus\mathbb F_q$.
\end{enumerate}
The two cases are inequivalent, and each consists of one $k$-M\"obius class.
In the split and nonsplit cases the Frobenius actions on the automorphism group of the geometric cover are respectively $r\mapsto r^q$ and $r\mapsto r^{-q}$.
\end{proposition}

\begin{proof}
Since $H_f$ is core-free in $G_f$ and $G_f\cong C_n$ is cyclic, one has
$H_f=1$. Hence $\deg f=n$ and $\Omega_f=\kappa_f(\mathbf x)$, so
$C_f\cong_{\kappa_f}\mathbf P^1_{\kappa_f}$; in particular, the Galois closure
has genus zero. We may therefore regard $G_f$ as a cyclic subgroup of
$\operatorname{PGL}_2(\kappa_f)\leqslant\operatorname{PGL}_2(\bar k)$. Let $r$ generate
$G_f$. Since $r$ is nonidentity and has order prime to $p$, it has two distinct
geometric fixed points; after a geometric M\"obius change of coordinate, we may write $r(z)=\zeta z$, where $\zeta$
has order $n$. Every nonidentity element of $G_f$ then fixes exactly $0$ and
$\infty$. Thus these are precisely the two ramification points of the Galois
quotient; both are totally ramified, and their images are the two branch
points. Frobenius preserves this unordered pair, so it either fixes the two
points individually or exchanges them.

In the split case, source and target $k$-M\"obius transformations place the
ramification and branch points at $\{0,\infty\}$. Over $\bar k$ the cover
is $cX^n$. Here $k$-rationality gives $c\in k^\times$, and a target scaling
removes $c$.

In the nonsplit case, source and target $k$-M\"obius transformations identify
their respective quadratic pairs with $\{\delta,\delta^q\}$. The coordinate
$\nu_\delta$ sends this pair to $\{0,\infty\}$ and satisfies
$\nu_\delta^{(q)}(X)=1/\nu_\delta(X)$, where the superscript $(q)$ acts on
coefficients. Thus Frobenius acts on the normalized coordinate by inversion.
Over $\mathbb F_{q^2}$ the cover is $y\mapsto cy^n$. The descent condition is
$c^q=c^{-1}$, so $c$ has norm $1$. Multiplication by $c^{-1}$ in the target
$y$-coordinate commutes with the nonsplit descent, since
$(c^{-1})^q=(c^{-1})^{-1}$; it therefore descends to a $k$-M\"obius
transformation and removes $c$. Hence the unique nonsplit representative is
$R_{n,\delta}$. The splitting type of the branch divisor distinguishes the two
classes.

Conjugating $y\mapsto\zeta y$ by Frobenius gives the stated actions
$\zeta\mapsto\zeta^q$ and $\zeta\mapsto\zeta^{-q}$.
\end{proof}

\begin{remark}
The integer $n$ in Proposition~\ref{prop:cyclic} need not be prime; the only
restriction on its order is $p\nmid n$. No congruence condition on $q$ modulo
$n$ is required for the classification itself; such arithmetic restrictions
arise only after imposing additional conditions, such as exceptionality. When
$n$ is composite, these cyclic forms are decomposable, so only prime $n$ enters
the indecomposable classification above. The cyclic classification in
Proposition~\ref{prop:cyclic} already appears, in slightly different form, in
the author's recent joint work with Tang
\cite[Theorem~4.2]{TangFan2026Degree5}. We include the proof here for
completeness and to keep the present genus-zero classification self-contained.
\end{remark}

\subsection{Dihedral cases}

\begin{proposition}[Dihedral $k$-forms]\label{prop:dihedral}
Let $f\in k(X)$ be separable and suppose that its Galois closure has genus
zero. Let $\ell$ be an odd prime with $p\nmid 2\ell$, and suppose that
$G_f\cong D_{2\ell}$ and $\deg f=\ell$. Then
$f(X)\sim_k D_\ell(X,a)$ for some $a\in k^\times$, where $D_\ell$ is characterized by
\[
 D_\ell\!\left(z+\frac az,a\right)
 =z^\ell+\left(\frac az\right)^\ell.
\]
Moreover,
\[
 D_\ell(X,a)\sim_k D_\ell(X,b)
 \quad\Longleftrightarrow\quad
 a/b\in(k^\times)^2.
\]
Hence there are exactly two classes, represented by $D_\ell(X,1)$ and $D_\ell(X,\epsilon)$, where $\epsilon$ is any nonsquare in $k^\times$.
\end{proposition}

\begin{proof}
Since $[G_f:H_f]=\deg f=\ell$ and $|G_f|=2\ell$, one has $|H_f|=2$.
Thus $H_f$ is a reflection subgroup of $G_f$; all such subgroups are conjugate.
After choosing an identification, write
$G=\langle r,s\mid r^\ell=s^2=1,\ srs=r^{-1}\rangle$ and
$H=\langle s\rangle$. The degree-$\ell$ quotient has a unique branch point
of ramification index $\ell$, with a unique point above it. Both points are
$k$-rational. Sending them to $\infty$ gives a polynomial representative.

Over $\bar k$ choose a coordinate $z$ upstairs with
\(r(z)=\zeta z, \qquad s(z)=\frac az.\)
Then
\(x=z+\frac az\)
generates the fixed field of $H$, while
\(t=z^\ell+\frac{a^\ell}{z^\ell}\)
generates the fixed field of $G$. The defining Dickson identity yields $t=D_\ell(x,a)$.

The two remaining branch points arise from the reflection fixed points $z^2=\zeta^i a$. If $u^2=a$, their unordered pair is $\{2u^\ell,-2u^\ell\}$. Because $\ell$ is odd, this pair is split over $k$ exactly when $a$ is a square, and forms a quadratic Frobenius orbit exactly when $a$ is a nonsquare.

Finally,
\(D_\ell(uX,u^2a)=u^\ell D_\ell(X,a),\)
so parameters in the same square class are $k$-M\"obius equivalent. The branch splitting type proves the converse.
\end{proof}

\section{Tame classification: polyhedral cases and completion}\label{sec:tame-polyhedral}

For a tame Galois quotient $C\to C/G$ with geometric branch points
$b_1,\ldots,b_r$, let $I_i$ be an inertia group above $b_i$.
The groups $I_i$ are cyclic. We call the tuple $(|I_1|,\ldots,|I_r|)$, considered up to permutation,
the \emph{signature} of the quotient.
We call $b_i$ an \emph{order-$m$ branch point} if its inertia group
in $C\to C/G$ has order $m$.
A generator of $I_i$ acts on $G/H$ with cycle lengths equal to the
ramification indices above $b_i$ in $C/H\to C/G$, so its cycle type
records the same data as the ramification partition of that fiber.

\subsection{The tetrahedral pair \texorpdfstring{$(A_4,C_3)$}{(A4,C3)}}

The tetrahedral Galois quotient has signature $(2,3,3)$. In the degree-$4$ action on $A_4/C_3$, generators of the three inertia groups have cycle types $2^2$, $3\,1$, and $3\,1$, respectively. Thus the order-$2$ branch point is $k$-rational, while the two order-$3$ branch points may be split or form one quadratic Frobenius orbit.

\begin{proposition}[The two tetrahedral quartics]\label{prop:tetra}
Assume $p\neq2,3$. The geometrically primitive pair $(A_4,C_3)$ has exactly two $k$-M\"obius classes. They are distinguished by the splitting type of the two order-$3$ branch points.

The split class is represented by
\[
 T_4(X)=\frac{X^3(X+8)}{64(X-1)}.
\]
The nonsplit class is the unique nontrivial \emph{quadratic twist} of this
cover: it is obtained by composing Frobenius on the source and target
with $X\mapsto-8/X$ and $t\mapsto1/t$, respectively, and descending the
resulting compatible actions over the quadratic extension of $k$. The arithmetic monodromy is as follows:
\[
\begin{array}{c|c|c}
q\bmod3&\text{split class}&\text{nonsplit class}\\ \hline
1&A_4&S_4\\
-1&S_4&A_4
\end{array}
\]
\end{proposition}

\begin{proof}
First note the exact identity
\[
 T_4(X)-1=\frac{(X^2+4X-8)^2}{64(X-1)}.
\]
Thus the branch points $0,\infty,1$ have ramification partitions
$\{\!\{3,1\}\!\}$, $\{\!\{3,1\}\!\}$, and $\{\!\{2,2\}\!\}$,
respectively, as required.

The pair automorphism calculation gives
\[
 \operatorname{Aut}(A_4)\cong S_4,
 \qquad
 \mathrm N_{A_4}(C_3)=C_3,
 \qquad
 \mathrm N_{S_4}(C_3)\cong S_3.
\]
Hence the quotient controlling semilinear pair types is
\(\mathrm N_{S_4}(C_3)/C_3\cong C_2.\)
There are therefore at most two $k$-forms, and no further cover-over-target ambiguity because $\mathrm N_{A_4}(C_3)/C_3=1$.

Both forms occur explicitly, since \(T_4(-8/X)=\frac1{T_4(X)}\). The source involution $X\mapsto-8/X$ and the target involution $t\mapsto1/t$ exchange the two order-$3$ branch points. Simultaneously twisting source and target by the unique quadratic extension of $k$ therefore produces the nonsplit class. The descended curves have genus zero over a finite field and are consequently $k$-isomorphic to $\mathbf P^1$.

For the arithmetic group of the split class, the defining quartic
\(X^3(X+8)-64t(X-1)\)
has discriminant \(-2^{24}3^3t^2(t-1)^2\). Thus the split class has arithmetic group $A_4$ exactly when $-3$ is a square in $k$, equivalently $q\equiv1\pmod3$, and otherwise $S_4$. The quadratic twist reverses the inner/outer image, giving the table.
\end{proof}

\subsection{The octahedral pair \texorpdfstring{$(S_4,S_3)$}{(S4,S3)}}

\begin{proposition}[The unique octahedral quartic]\label{prop:octa}
Assume $p\neq2,3$. Every rational function with geometric pair $(S_4,S_3)$ is $k$-M\"obius equivalent to
\[
 O_4(X)=4X^3-3X^4.
\]
There is exactly one $k$-M\"obius class, and its arithmetic and geometric monodromy groups are both $S_4$.
\end{proposition}

\begin{proof}
The only faithful primitive octahedral pair is $(S_4,S_3)$ by Theorem~\ref{thm:tamepairs}. Since the rational function has degree $4$, its arithmetic group satisfies
\(S_4=G\leqslant A\leqslant S_4,\)
so $A=G=S_4$.

The octahedral signature is $(2,3,4)$, and in the natural degree-$4$ action the cycle types are
\(2\,1\,1, \qquad 3\,1, \qquad 4.\)
Because the three inertia orders are distinct, all three branch points are $k$-rational. Normalize the branch points of orders $3,2,4$ to $0,1,\infty$. Put the unique index-$3$ point above $0$ at $X=0$ and the totally ramified point above $\infty$ at $X=\infty$. The rational function is then a polynomial
\(F(X)=cX^3(X-a).\)
Its derivative is
\(F'(X)=cX^2(4X-3a).\)
Scale the source so that the remaining finite critical point is $1$, giving $a=4/3$. The normalization $F(1)=1$ then gives $c=-3$. Hence
\(F(X)=4X^3-3X^4.\)
Conversely,
\(O_4'(X)=12X^2(1-X),\)
so the ramification indices at $0,1,\infty$ are $3,2,4$. Generators of the corresponding tame inertia groups generate $S_4$.
\end{proof}

For the icosahedral cases below, assume $p\notin\{2,3,5\}$. The icosahedral Galois quotient has signature $(2,3,5)$.

\subsection{Degree \texorpdfstring{$5$}{5}: the pair \texorpdfstring{$(A_5,A_4)$}{(A5,A4)}}

\begin{proposition}[The icosahedral quintic]\label{prop:I5}
Every rational function with geometric pair $(A_5,A_4)$ is $k$-M\"obius equivalent to
\[
 I_5(X)=\frac{X^3(X^2+5X+40)}{1728}.
\]
There is exactly one $k$-M\"obius class. Its geometric monodromy is $A_5$; its arithmetic monodromy is $A_5$ when $5\in(k^\times)^2$, and $S_5$ otherwise.
\end{proposition}

\begin{proof}
In the degree-$5$ action on $A_5/A_4$, generators of the three inertia
groups have cycle types $2^2\,1$, $3\,1^2$, and $5$. The branch points have distinct inertia orders and are therefore all $k$-rational. Normalize the order-$3$, order-$2$, and order-$5$ branch points to $0,1,\infty$, and put the unique index-$3$ point over $0$ at $X=0$ and the totally ramified point over $\infty$ at $X=\infty$. The rational function has the form
\(F(X)=cX^3(X^2+sX+t).\)
Then
\(F'(X)=cX^2(5X^2+4sX+3t).\)
The two remaining critical points must be distinct and must have the same critical value. Reducing $X^3(X^2+sX+t)$ modulo $5X^2+4sX+3t$, the coefficient of $X$ is a nonzero scalar multiple of
\(32s^4-140s^2t+75t^2.\)
Thus
\(32s^4-140s^2t+75t^2=0.\)
Writing $r=s^2/t$ gives $r=15/4$ or $5/8$. The first value makes the critical quadratic inseparable, so $r=5/8$. After source scaling one may take $s=5,t=40$. The exact identity
\[
 X^3(X^2+5X+40)-1728=(X-3)(X^2+4X+24)^2
\]
fixes the remaining target scaling and yields $I_5$.

Moreover,
\(I_5'(X)=\frac{5X^2(X^2+4X+24)}{1728},\)
so generators of the corresponding tame inertia groups have cycle types
$3\,1^2$, $2^2\,1$, and $5$. All are even. The transitive degree-$5$ group generated by them contains elements of orders $3$ and $5$, hence is $A_5$.

Finally, the generic equation
\(X^5+5X^4+40X^3-1728t=0\)
has discriminant
\begin{equation}\label{eq:I5-disc}
 2^{24}3^{12}5^5t^2(t-1)^2.
\end{equation}
Since $A_5\leqslant A\leqslant S_5$, the discriminant criterion gives the stated arithmetic monodromy. The normalized solution is unique; hence there is a single $k$-M\"obius class.
\end{proof}

\subsection{Degree \texorpdfstring{$6$}{6}: the pair \texorpdfstring{$(A_5,D_{10})$}{(A5,D10)}}

\begin{proposition}[The icosahedral sextic]\label{prop:I6}
Every rational function with geometric pair $(A_5,D_{10})$ is $k$-M\"obius equivalent to
\[
 I_6(X)=\frac{(X^2+10X+5)^3}{1728X}.
\]
There is exactly one $k$-M\"obius class. Its geometric monodromy is $A_5$, and the same square-class criterion as in Proposition~\ref{prop:I5} determines whether the arithmetic group is $A_5$ or its outer extension $S_5$.
\end{proposition}

\begin{proof}
In the degree-$6$ action on $A_5/D_{10}$, generators of the three inertia
groups have cycle types $2^2\,1^2$, $3^2$, and $5\,1$. An involution fixes two cosets: the six subgroups $D_{10}$ each contain five involutions, whereas $A_5$ has fifteen involutions. An element of order $3$ lies in no $D_{10}$, hence has no fixed point, and a $5$-cycle fixes the unique coset associated with its Sylow-$5$ subgroup.

Normalize the order-$3$, order-$2$, order-$5$ branch points to $0,1,\infty$. Above $\infty$ there is one point of index $5$ and one simple point; normalize them to $\infty$ and $0$. The ramification partition at $0$ is $\{\!\{3,3\}\!\}$, so
\(F(X)=c\frac{(X^2+aX+b)^3}{X}.\)
Writing $P=X^2+aX+b$ gives
\(F'(X)=c\frac{P(X)^2(5X^2+2aX-b)}{X^2}.\)
The two roots of $5X^2+2aX-b$ must have the same critical value. Reduction modulo this quadratic gives the condition
\(a^4-15a^2b-100b^2=0.\)
With $r=a^2/b$, one has $r=20$ or $-5$. The latter makes the critical quadratic have zero discriminant, so $a^2=20b$. Source scaling reduces the quadratic to
\(X^2+10X+5.\)
The identity
\begin{equation}\label{eq:I6-factor}
 (X^2+10X+5)^3-1728X
 =(X^2+4X-1)^2(X^2+22X+125)
\end{equation}
fixes the target scaling. Thus the unique normalized rational function is $I_6$.

Its derivative is
\[
 I_6'(X)=
 \frac{5(X^2+10X+5)^2(X^2+4X-1)}{1728X^2},
\]
and \eqref{eq:I6-factor} gives the required ramification partitions.
Choose generators $\sigma_2,\sigma_3,\sigma_5$ of the three tame inertia
groups that generate the geometric monodromy group and satisfy
$\sigma_2\sigma_3\sigma_5=1$; the tame-cover existence theorem permits
this choice as in \cite[Section~2.8]{GMS2003}. Their orders are $2,3,5$,
respectively, so that group is a quotient of
\[
 \langle x_2,x_3,x_5\mid
 x_2^2=x_3^3=x_5^5=x_2x_3x_5=1\rangle\cong A_5.
\]
The quotient is nontrivial because it contains an element of order $5$.
Simplicity of $A_5$ therefore gives the geometric group $A_5$.

The generic equation
\((X^2+10X+5)^3-1728tX=0\)
has discriminant
\begin{equation}\label{eq:I6-disc}
 2^{36}3^{18}5^5t^4(t-1)^2.
\end{equation}
Its square class is $5$, giving the stated arithmetic criterion. Uniqueness follows from the normalization just performed.
\end{proof}

\subsection{Degree \texorpdfstring{$10$}{10}: the pair \texorpdfstring{$(A_5,S_3)$}{(A5,S3)}}

\begin{lemma}[Cycle types in the degree-$10$ action]\label{lem:I10-passport}
In the action of $A_5$ on $A_5/S_3$, the elements of orders $2,3,5$ have cycle types
\[
 2^4\,1^2,
 \qquad
 3^3\,1,
 \qquad
 5^2,
\]
respectively.
\end{lemma}

\begin{proof}
For $g\in A_5$, the number of fixed cosets is
\[
 |\operatorname{Fix}_{A_5/S_3}(g)|
 =\frac{|C_{A_5}(g)|\,|g^{A_5}\cap S_3|}{|S_3|}.
\]
An involution has centralizer of order $4$ and $S_3$ contains three involutions, so it fixes $4\cdot3/6=2$ cosets. An element of order $3$ has centralizer of order $3$ and $S_3$ contains two such elements, so it fixes one coset. Finally $S_3$ contains no element of order $5$, so a $5$-cycle has no fixed point. These fixed-point counts give the stated cycle types.
\end{proof}

\begin{proposition}[The icosahedral degree-$10$ map]\label{prop:I10}
Every rational function with geometric pair $(A_5,S_3)$ is $k$-M\"obius equivalent to
\[
 I_{10}(X)=
 \frac{125(X+1)(2X+1)^3(2X^2-3X+3)^3}
 {1728(X^2+X-1)^5}.
\]
There is exactly one $k$-M\"obius class. Its geometric monodromy is $A_5$, and its arithmetic monodromy obeys the same square-class criterion as Propositions~\ref{prop:I5} and~\ref{prop:I6}.
\end{proposition}

\begin{proof}
Let
\begin{align*}
 N(X,Z)&=125(X+Z)(2X+Z)^3(2X^2-3XZ+3Z^2)^3,\\
 D(X,Z)&=(X^2+XZ-Z^2)^5.
\end{align*}
Then $I_{10}=N/(1728D)$. The ramification partitions at $0$ and $\infty$ are
$\{\!\{1,3,3,3\}\!\}$ and $\{\!\{5,5\}\!\}$, respectively. The order-$2$ fiber follows from the exact identity
\begin{equation*}
\begin{split}
 N(X,Z)-1728D(X,Z)
 ={}&(4X^2-6XZ-9Z^2)^2\\
 &\cdot(7X^2+2XZ+3Z^2)^2\\
 &\cdot(8X^2-12XZ+7Z^2).
\end{split}
\end{equation*}
Outside characteristics $2,3,5$, the displayed factors are separable and pairwise coprime. Hence the ramification partition at $1$ is $\{\!\{2,2,2,2,1,1\}\!\}$.
The associated cycle types are those in Lemma~\ref{lem:I10-passport}. The total ramification contribution equals
\(3(3-1)+2(5-1)+4(2-1)=18=2\cdot10-2,\)
so there is no further ramification.

As in the proof of Proposition~\ref{prop:I6}, choose generators
$\sigma_2,\sigma_3,\sigma_5$ of the tame inertia groups that generate the
geometric monodromy group, have orders $2,3,5$, and satisfy
$\sigma_2\sigma_3\sigma_5=1$. They give a nontrivial quotient of the
presentation displayed there; hence the geometric group is $A_5$. The point stabilizer has order $60/10=6$ and is $S_3$, so this is the
natural action on $A_5/S_3$.

Geometrically, there is a single $\operatorname{PGL}_2(\bar k)$-conjugacy class realizing the pair $(A_5,S_3)$. Moreover,
\(\mathrm N_{A_5}(S_3)=S_3,\)
so the cover has no nontrivial automorphism over its target. In $\operatorname{Aut}(A_5)=S_5$, the pair-normalizer quotient has order $2$. The two $A_5$-classes of $5$-cycles are interchanged by every outer automorphism; taking the $q$th power preserves either class for $q\equiv\pm1\pmod5$ and interchanges them for $q\equiv\pm2\pmod5$. Since the order-$5$ branch point is $k$-rational, the tame inertia relation forces the arithmetic Frobenius to induce exactly this $q$-power action. Hence, for fixed $q$, exactly one of the inner or outer possibilities is admissible, proving uniqueness of the $k$-form.

As an arithmetic check, the generic equation $N(X,1)-1728tD(X,1)=0$ has discriminant
\begin{equation}\label{eq:I10-disc}
 2^{60}3^{30}5^{75}t^6(t-1)^4.
\end{equation}
Again the square class is $5$, giving the claimed arithmetic monodromy.
\end{proof}

\begin{corollary}[Uniform icosahedral arithmetic criterion]\label{cor:A5criterion}
For each of the degrees $5,6,10$ and every $q$ prime to $30$,
\[
 A=A_5\iff q\equiv\pm1\pmod5,
 \qquad
 A\cong S_5\iff q\equiv\pm2\pmod5.
\]
\end{corollary}

\begin{proof}
In all three cases the generic discriminant has square class $5$ by \eqref{eq:I5-disc}, \eqref{eq:I6-disc}, and \eqref{eq:I10-disc}. Over $\mathbb F_q$, the element $5$ is a square exactly when $q$ is a square modulo $5$.
\end{proof}

\subsection{Completion of the tame classification}

\begin{theorem}[Complete tame classification]\label{thm:intro-tame}
Let $f\in k(X)$ be separable and $k$-indecomposable, assume that its Galois closure has genus zero, and suppose $p\nmid |G_f|$. Then, up to $k$-M\"obius equivalence, $f$ falls into exactly one of the cases below. Within each case, the classes listed are pairwise distinct and exhaustive, and all of them occur.
\begin{enumerate}[label=(\arabic*),ref=\arabic*,font=\itshape]
\item $G_f=C_\ell$ and $H_f=1$, where $\ell\ne p$ is prime, and $f$ belongs to one of the two $k$-M\"obius classes: the split power class and the nonsplit R\'edei class.
\item $G_f=D_{2\ell}$ and $H_f=C_2$, where $\ell$ is an odd prime, and $f$ belongs to one of the two Dickson $k$-M\"obius classes.
\item $f$ belongs to the unique $V_4/A_4$ quartic $k$-M\"obius class; its arithmetic monodromy group is $A_4$, its geometric monodromy group is $V_4$, its arithmetic point stabilizer is $C_3$, and its geometric point stabilizer is trivial.
\item $(G_f,H_f)=(A_4,C_3)$, and $f$ belongs to one of the two tetrahedral quartic $k$-M\"obius classes, split and nonsplit.
\item $(G_f,H_f)=(S_4,S_3)$, and $f$ belongs to the unique octahedral quartic $k$-M\"obius class.
\item $(G_f,H_f)$ is one of $(A_5,A_4)$, $(A_5,D_{10})$, or $(A_5,S_3)$, of respective degrees $5,6,10$, and $f$ belongs to the unique $k$-M\"obius class associated with that pair.
\end{enumerate}
Moreover, $f$ is geometrically decomposable if and only if it belongs to the $V_4/A_4$ quartic class. Consequently, $\deg f$ is either prime or belongs to $\{4,6,10\}$.
\end{theorem}

\begin{proof}[Proof of Theorem~\ref{thm:intro-tame}]
By Theorem~\ref{thm:decomposition}, $k$-indecomposability is equivalent to primitivity of the arithmetic monodromy action. Corollary~\ref{cor:tame-group-complete} says that, in the tame case, either the rational function is geometrically primitive with one of the seven pairs in Theorem~\ref{thm:tamepairs}, or it is the unique $V_4/A_4$ case.

The cyclic and dihedral pairs are classified over $k$ by Propositions~\ref{prop:cyclic} and~\ref{prop:dihedral}. The $V_4/A_4$ case is Proposition~\ref{prop:V4-quartic}. The tetrahedral and octahedral pairs are Propositions~\ref{prop:tetra} and~\ref{prop:octa}; the three icosahedral pairs are Propositions~\ref{prop:I5},~\ref{prop:I6}, and~\ref{prop:I10}. These cases are mutually exclusive at the level of geometric monodromy pairs. Different cases may share the same degree, but this does not affect the classification.

The only geometrically decomposable case is the $V_4/A_4$ quartic by Theorem~\ref{thm:arith-only}. The degree statement follows from the indices in Theorem~\ref{thm:tamepairs} together with the geometrically decomposable degree $4$: prime degrees occur in the cyclic and dihedral cases, while the remaining degrees are $4,6,10$.
\end{proof}

\begin{remark}[Arithmetic monodromy of the rigid polyhedral forms]
The split and nonsplit tetrahedral arithmetic groups are determined in Proposition~\ref{prop:tetra}. For each of the three icosahedral forms, Corollary~\ref{cor:A5criterion} gives arithmetic group $A_5$ when $5$ is a square in $k$ and the nontrivial outer extension when $5$ is a nonsquare.
\end{remark}

For reference, the tame classification is summarized in the following table.

\begin{center}
\renewcommand{\arraystretch}{1.25}
\begin{tabular}{@{}>{\raggedright\arraybackslash}p{0.23\textwidth}>{\raggedright\arraybackslash}p{0.20\textwidth}>{\raggedright\arraybackslash}p{0.37\textwidth}c@{}}
\toprule
Geometric pair & Degree & $k$-representatives / description & Number \\
\midrule
$(C_\ell,1)$ & $\ell$ prime & $X^\ell$; $R_{\ell,\delta}$ & $2$\\
$(D_{2\ell},C_2)$ & $\ell$ odd prime & $D_\ell(X,1)$; $D_\ell(X,\epsilon)$ & $2$\\
$(V_4,1)$ geometrically decomposable & $4$ & $L_h$, $h$ irreducible cubic & $1$\\
$(A_4,C_3)$ & $4$ & split $T_4$; unique nonsplit twist & $2$\\
$(S_4,S_3)$ & $4$ & $4X^3-3X^4$ & $1$\\
$(A_5,A_4)$ & $5$ & $I_5$ & $1$\\
$(A_5,D_{10})$ & $6$ & $I_6$ & $1$\\
$(A_5,S_3)$ & $10$ & $I_{10}$ & $1$\\
\bottomrule
\end{tabular}
\end{center}

\section{Wild classification: reduction, affine, and small-characteristic cases}\label{sec:wild-master}
\subsection{Geometric groups and reduction to families}

We turn to the wild part of the classification, where $p=\operatorname{char}k$
divides the order of the geometric monodromy group. For the groups
$\operatorname{PSL}_2(Q)$ and $\operatorname{PGL}_2(Q)$, the phrase
\emph{defining characteristic} means that $Q=p^a$ for some $a\geqslant1$,
with the same characteristic $p$ as the base field. Following Faber, a finite group is \emph{$p$-semi-elementary} if it has a unique Sylow $p$-subgroup $P$ of exponent $p$ and its quotient by $P$ is cyclic. Faber's classification of finite subgroups of $\operatorname{PGL}_2$ whose order is divisible by $p$ reduces the possible geometric monodromy groups, after the tame overlaps are removed, to defining-characteristic $\operatorname{PSL}_2/\operatorname{PGL}_2$, $p$-semi-elementary groups, dihedral groups and $A_5$ \cite[Theorem~B]{Faber2023}. This determines only the possible geometric monodromy groups. For the prescribed finite-field problem, we must still determine the core-free point stabilizers yielding primitive arithmetic monodromy, classify the admissible semilinear Frobenius actions, and identify the resulting $k$-M\"obius classes.

These requirements lead to five cases: affine $p$-semi-elementary,
characteristic-$2$ dihedral, characteristic-$3$ icosahedral,
geometrically indecomposable defining-characteristic Lie quotients, and
Lie-type quotients that are $k$-indecomposable but geometrically decomposable. The following theorem gives the resulting wild classification.

In the natural action on $\mathbf P^1(\mathbb F_Q)$, a point stabilizer
is called a \emph{Borel subgroup}.
The pointwise stabilizer in $\operatorname{PGL}_2(Q)$ of two distinct
$\mathbb F_Q$-rational points is a \emph{split Cartan subgroup}; the
pointwise stabilizer of two geometric points exchanged by $Q$-Frobenius
is a \emph{nonsplit Cartan subgroup}. The corresponding Cartan subgroups
of $\operatorname{PSL}_2(Q)$ are their intersections with that group.
These two Cartan types refer to the defining field $\mathbb F_Q$,
not necessarily to the base field $k$ of a descended cover.
If $Q=Q_0^r$ with $r>1$, a \emph{subfield subgroup} is, up to conjugacy,
a subgroup arising from the inclusion $\mathbb F_{Q_0}\subseteq\mathbb F_Q$,
of type $\operatorname{PSL}_2(Q_0)$ or $\operatorname{PGL}_2(Q_0)$ as
appropriate to the ambient group.

\begin{theorem}[Exhaustion of the wild cases]\label{thm:wild-master}
Let $f\in k(X)$ be separable and $k$-indecomposable, with Galois closure of genus zero and $p\mid |G_f|$. Then, after the small-group identifications recorded below, precisely one of the following occurs.
\begin{enumerate}[label=(\arabic*),ref=\arabic*,font=\itshape]
\item $G_f$ is $p$-semi-elementary; the classes are exactly those of Theorem~\ref{thm:affine-wild}.
\item $p=2$ and $(G_f,H_f)=(D_{2\ell},C_2)$ with $\ell$ an odd prime; Proposition~\ref{prop:char2-dihedral} applies.
\item $p=3$ and $G_f=A_5$ with $H_f=A_4,D_{10}$ or $S_3$; Proposition~\ref{prop:char3-A5} applies.
\item $G_f=\operatorname{PSL}_2(Q)$ or $\operatorname{PGL}_2(Q)$ and $H_f$ is one of the core-free maximal Dickson--Giudici cases recorded below, including the $Q=3$ Borel endpoints; the $k$-forms are classified in Propositions~\ref{prop:borel-wild},~\ref{prop:split-cartan},~\ref{prop:nonsplit-cartan},~\ref{prop:subfield-wild}, and~\ref{prop:exceptional-wild}.
\item $G_f=\operatorname{PSL}_2(Q)$ and $H_f$ is nonmaximal but yields a primitive arithmetic action after adjoining Frobenius; these are exactly the geometrically decomposable cases in Theorem~\ref{thm:wild-novelty}.
\end{enumerate}
No other wild $k$-indecomposable rational function with Galois closure of genus zero occurs.
\end{theorem}

\medskip
\noindent\textit{Group-theoretic inputs and parameter ranges.}
To make the classification cases precise we record the small-group conventions and the exact Dickson--Giudici ranges used below. Faber's additional characteristic-two possibility, in which the common fixed point of the group is not rational, requires a nonsquare in the ground field \cite[Theorem~A(3)]{Faber2023}. It does not occur over a finite field, since squaring is bijective.

\medskip
\noindent\textit{Dihedral convention.}
Throughout this paper $D_m$ denotes a dihedral group of order $m$. This differs from Faber's convention $D_n$ for a dihedral group of order $2n$.
The accidental isomorphisms
\[
\begin{gathered}
 A_4\cong \mathbb F_4\rtimes C_3,\qquad
 A_5\cong\operatorname{PGL}_2(4),\qquad
 S_3\cong\operatorname{PGL}_2(2),\\
 A_4\cong\operatorname{PSL}_2(3),\qquad
 S_4\cong\operatorname{PGL}_2(3),\qquad
 A_5\cong\operatorname{PSL}_2(5)
\end{gathered}
\]
are counted only once. More precisely, $S_3\cong\operatorname{PGL}_2(2)$ is included in the characteristic-$2$ dihedral case, $A_4\cong\mathbb F_4\rtimes C_3$ in the affine family, $A_5\cong\operatorname{PGL}_2(4)$ in the $Q=4$ Lie-type cases, $A_4\cong\operatorname{PSL}_2(3)$ and $S_4\cong\operatorname{PGL}_2(3)$ as the $Q=3$ Borel endpoints, and $A_5\cong\operatorname{PSL}_2(5)$ in the $Q=5$ Lie-type cases.

For the Lie-type families, we separate the proof-critical external inputs by role. Faber supplies the finite ambient subgroups whose order is divisible by $p$. The ordinary $\operatorname{PSL}_2/\operatorname{PGL}_2$ maximal-subgroup types are the classical Dickson list; see also the tabulation in \cite{Bray2013}. We use Giudici \cite[Theorems~2.1, 2.2 and~3.5]{Giudici2007} for the exact maximal-subgroup ranges, \cite[Lemma~2.3]{Giudici2007} for the action of $\operatorname{PGL}_2(Q)$ on $\operatorname{PSL}_2(Q)$-conjugacy classes of subgroups, and \cite[Theorem~1.1, Corollary~1.2 and Table~1]{Giudici2007} for maximality of $U$ in $A$ when $U\cap G$ is not maximal in $G$; Giudici's Proposition~3.1 shows that the subfield cases do not yield this phenomenon. The two $Q=3$ endpoints are handled directly in Proposition~\ref{prop:borel-wild}. For even $Q\geqslant4$, $\operatorname{PSL}_2(Q)=\operatorname{PGL}_2(Q)$ and the primitive cases are Borel, the two Cartan normalizers, and subfield groups $\operatorname{PGL}_2(Q_0)$ with
\[
 Q=Q_0^r,\qquad r\text{ prime},\qquad Q_0\ne2.
\]
The apparent subfield realization at $Q=4,Q_0=2$ is the same $S_3$ conjugacy class as the split-Cartan normalizer and is counted there. For odd $Q\geqslant5$, the $\operatorname{PSL}_2(Q)$ cases are Borel,
\[
 D_{Q-1}\ (Q\geqslant13),\qquad
 D_{Q+1}\ (Q\ne7,9),
\]
subfield groups
\[
 \operatorname{PGL}_2(Q_0)\quad(Q=Q_0^2),
 \qquad
 \operatorname{PSL}_2(Q_0)\quad(Q=Q_0^r,\text{ $r$ odd prime}),
\]
and the cases with stabilizer $A_4,S_4$, or $A_5$
\begin{align*}
 A_4&:\quad Q=p\equiv\pm3\pmod8,\quad Q\not\equiv\pm1\pmod{10},\\
 S_4&:\quad Q=p\equiv\pm1\pmod8,\\
 A_5&:\quad Q\equiv\pm1\pmod{10},
 \quad Q=p\text{ or }Q=p^2\text{ with }p\equiv\pm3\pmod{10}.
\end{align*}
For odd $Q>3$, the core-free maximal cases in $\operatorname{PGL}_2(Q)$ are Borel,
\[
 D_{2(Q-1)}\ (Q\ne5),\qquad D_{2(Q+1)},
\]
$\operatorname{PGL}_2(Q_0)$ for $Q=Q_0^r$ with $r$ an odd prime, and the $S_4$-stabilizer case $Q=p\equiv\pm3\pmod8$. The normal subgroup $\operatorname{PSL}_2(Q)$ is maximal but not core-free and hence does not yield a case.

\begin{lemma}[Outer images of Frobenius and uniqueness of $k$-forms]\label{lem:outer-rigidity}
Let $Q=p^a>3$ be odd, let $G=\operatorname{PSL}_2(Q)$ in its standard action on $\mathbf P^1$, and let $H\leqslant G$ be \emph{self-normalizing} in $G$, meaning
$\mathrm N_G(H)=H$. Fix $k=\mathbb F_{p^s}$ and the geometric $G$-conjugacy class of $H$. Modulo $G$, a semilinear $q$-Frobenius lift normalizing $G$ has one of the two outer images
\[
 \bar\phi^{\,s},\qquad \bar\delta\bar\phi^{\,s}
 \quad\text{in}\quad
 \operatorname{Out}(G)=\langle\bar\delta\rangle\times\langle\bar\phi\rangle.
\]
Here $\bar\delta$ is the \emph{diagonal outer involution}, the nontrivial
outer class induced by conjugation by any element of
$\operatorname{PGL}_2(Q)\setminus\operatorname{PSL}_2(Q)$, and
$\bar\phi$ is the outer class of the \emph{field automorphism}
induced by the $p$-power Frobenius on matrix entries. An outer image occurs for the pair $(G,H)$ exactly when it stabilizes the $G$-conjugacy class of $H$, and each occurring outer image gives at most one $k$-form. If both outer images occur, the resulting $k$-covers are not $k$-M\"obius equivalent.
\end{lemma}

\begin{proof}
The description of $\operatorname{Out}(G)$ and the action of outer automorphisms on subgroup conjugacy classes are given by Giudici \cite[Lemma~2.3]{Giudici2007}. In the standard projective realization, $\mathrm N_{\operatorname{PGL}_2(\bar k)}(G)=\operatorname{PGL}_2(Q)$; hence its quotient by $G$ consists of the trivial class and the diagonal outer class. Suppose two lifts $\psi_1$ and $\psi_2$ of the same outer image normalize the same representative $H$. Since their outer images agree, the quotient $\psi_2\psi_1^{-1}$ lies in $G$; it also normalizes $H$, hence belongs to $\mathrm N_G(H)=H$. Therefore $\langle H,\psi_1\rangle=\langle H,\psi_2\rangle$, so the two descents have the same fixed field. If both outer images occur, the corresponding arithmetic groups have distinct images
\[
 \langle\bar\phi^{\,s}\rangle,
 \qquad
 \langle\bar\delta\bar\phi^{\,s}\rangle
\]
in the abelian group $\operatorname{Out}(G)$. A $k$-M\"obius equivalence would identify the arithmetic and geometric Galois data and hence these outer images, which is impossible.
\end{proof}

\subsection{\texorpdfstring{$p$}{p}-semi-elementary groups and Ore skew polynomials}

Let $G=V\rtimes C_n\leqslant\operatorname{PGL}_2(\bar k)$, where $(n,p)=1$,
$V\ne0$ is an additive $p$-group, and $C_n$ acts by scalar multiplication.
For $n>1$, put $e:=\operatorname{ord}_n(p)$ and $Q:=p^e$; for $n=1$, put
$Q:=p$. Then $V$ is naturally an $\mathbb F_Q$-vector space.
A polynomial of the form $\sum_{i=0}^r a_iZ^{Q^i}$ is called
\emph{$Q$-linearized}. For a finite-dimensional
$\mathbb F_Q$-subspace $W\subseteq\bar k$, its \emph{subspace polynomial} is
\[
 L_W(Z):=\prod_{w\in W}(Z-w).
\]
It is the unique monic separable $Q$-linearized polynomial whose set of
roots is $W$. If $W$ is stable under $q$-Frobenius, then $L_W\in k[Z]$.
Put $\rho(a):=a^Q$ on $k$, and write $k[\tau;\rho]$ for the
Ore skew-polynomial ring with convention $\tau a=\rho(a)\tau$ for $a\in k$.
For a nonzero skew polynomial $h=\sum_{i=0}^r a_i\tau^i$ with $a_r\ne0$,
put $\deg_\tau h:=r$.
Under the correspondence
\[
 \Psi\!\left(\sum_i a_i\tau^i\right)=\sum_i a_iZ^{Q^i},
 \qquad
 \Psi(uv)=\Psi(u)\circ\Psi(v),
\]
Ore multiplication corresponds to composition of $Q$-linearized polynomials;
if $\deg_\tau h=r$, the ordinary degree of $\Psi(h)$ is $Q^r$.
In $k[\tau;\rho]$, we call $h$ a \emph{right divisor}, or
\emph{right factor}, of $g$ if $g=ah$ for some $a\in k[\tau;\rho]$.
A nonzero nonunit skew polynomial is \emph{irreducible} if it cannot
be written as a product of two nonunits of positive $\tau$-degree.

\begin{lemma}[Subspace polynomials and Ore right factors]\label{lem:affine-ore-dictionary}
Let $W_1\subseteq W_2$ be finite-dimensional $q$-Frobenius-stable $\mathbb F_Q$-subspaces of $\bar k$, and let $\lambda_{W_i}\in k[\tau;\rho]$ be the unique monic skew polynomials satisfying $\Psi(\lambda_{W_i})=L_{W_i}$. Then $\lambda_{W_1}$ is a right divisor of $\lambda_{W_2}$. Conversely, every monic right divisor $h$ of $\lambda_{W_2}$ with nonzero constant coefficient has $q$-Frobenius-stable root space $\ker\Psi(h)\subseteq W_2$. Consequently, $\lambda_W$ is irreducible if and only if $W$ has no nonzero proper $q$-Frobenius-stable $\mathbb F_Q$-subspace.
\end{lemma}

\begin{proof}
Perform right division of $\lambda_{W_2}$ by $\lambda_{W_1}$ in
$k[\tau;\rho]$. If the remainder $r_0$ were nonzero, then
$\deg_\tau r_0<\dim_{\mathbb F_Q}W_1$, so $\Psi(r_0)$ would have ordinary
degree less than $|W_1|$ while vanishing on $W_1$. Thus the remainder is
zero, proving right divisibility. Conversely, if $\lambda_{W_2}=gh$, then $\Psi(\lambda_{W_2})=\Psi(g)\circ\Psi(h)$, so $\ker\Psi(h)\subseteq W_2$. Since $h$ has coefficients in $k$, this kernel is stable under $q$-Frobenius; the nonzero constant coefficient ensures that $\Psi(h)$ is separable, so its root-space dimension equals $\deg_\tau h$. The final assertion follows by taking a nontrivial right factor in either direction.
\end{proof}

\begin{theorem}[Complete affine wild classification]\label{thm:affine-wild}
Let $k=\mathbb F_q$, and let
$\mathcal D=(\kappa/k,C,G,H,\varphi)$ be a semilinear $\mathbf P^1$-datum
over $k$ with $p$-semi-elementary geometric group $G=V\rtimes C_n$. Let
$f\in\mathcal M_{\mathcal D}$. If $f$ is $k$-indecomposable, then
\[
 H\cap V=1,\qquad HV=G;
\]
in other words, $H$ is a \emph{complement} to $V$ in $G$.
After an affine change of coordinate, $H$ may be taken to be the scalar
subgroup $\mu_n\cong C_n$.

In this case, put $r=\dim_{\mathbb F_Q}V$. After an affine descent
normalization, there is a normalized translation space $V'\subseteq\bar k$,
obtained from $V$ by the corresponding scalar change of coordinate, such that
$V'$ is stable under $q$-Frobenius. Write
\[
\begin{aligned}
 L_{V'}(Z)&=Z^{Q^r}+c_{r-1}Z^{Q^{r-1}}+\cdots+c_0Z\in k[Z],
 \qquad c_0\ne0,\\
 \ell(\tau)&=\tau^r+c_{r-1}\tau^{r-1}+\cdots+c_0\in k[\tau;\rho].
\end{aligned}
\]
Then
\[
 f\text{ is $k$-indecomposable}
 \quad\Longleftrightarrow\quad
 \ell\text{ is irreducible in }k[\tau;\rho].
\]
A representative is
\begin{equation}\label{eq:affine-wild}
 F_{\ell,n}(X)=
 X\left(\sum_{i=0}^r c_iX^{(Q^i-1)/n}\right)^n,
 \qquad c_r=1.
\end{equation}
It has degree $Q^r$ and satisfies
\begin{equation}\label{eq:affine-derivative}
 F_{\ell,n}'(X)=
 c_0\left(\sum_{i=0}^r c_iX^{(Q^i-1)/n}\right)^{n-1}.
\end{equation}
It is geometrically indecomposable exactly when $r=1$. Thus the geometrically
decomposable affine cases are precisely the irreducible cases $r>1$.

For fixed $(n,r)$ two monic irreducible data $\ell,\ell'$ give
$k$-M\"obius equivalent maps if and only if there exists $u\in k^\times$ such
that, with $d_i:=(Q^i-1)/n$ and $D:=(Q^r-1)/n$, one has
\[
 c_i'=c_i u^{d_i-D}\qquad(0\leqslant i<r).
\]
\end{theorem}

\begin{proof}
Assume that $H$ is core-free and that the arithmetic monodromy action is primitive. The unique Sylow-$p$ subgroup $V$ is characteristic, so $HV$ is Frobenius-stable. Arithmetic primitivity and core-freeness therefore force
\(HV=G, \qquad H\cap V=1,\)
so $H$ is a complement. For the remaining assertions, assume only that $H$ is a complement.

Assume first that $n>1$. The translation group $V$ has the unique common fixed point $\infty$. After a $\kappa$-affine change of coordinate we may therefore take this point to be $\infty$, the second fixed point of $H$ to be $0$, and
\(H=\mu_n, \qquad \varphi(Z)=aZ\)
semilinearly, with $a\in \kappa^\times$. If $d=[\kappa:k]$, the condition $\varphi^d\in H$ gives
\(N_{\kappa/k}(a)^n=1.\)
Hence $N_{\kappa/k}(a^n)=1$, and Hilbert's Theorem~90 gives $u_0\in \kappa^\times$ such that
\(u_0^{q-1}=a^{-n}.\)
Choose, only as an auxiliary geometric scalar, $c\in\bar k^\times$ with $c^n=u_0$, and put
\(b=a c^{q-1}, \qquad V'=cV.\)
Then $b^n=1$, so $b\in\mu_n\leqslant\mathbb F_Q^\times$. Since conjugation by $\varphi$ sends the translation parameter $v$ to $a^{-1}v^q$, the equality $\varphi V\varphi^{-1}=V$ is equivalent to $V^q=aV$. Consequently
\((V')^q=c^qV^q=bV'=V'.\)
Indeed every $H\leqslant J\leqslant G$ is $J=(J\cap V)\rtimes H$, and $J\cap V$ is an $\mathbb F_Q$-subspace; moreover $J$ is $\varphi$-stable exactly when $(c(J\cap V))^q=c(J\cap V)$. Thus the intermediate $k$-subgroups are encoded exactly by the Frobenius-stable $\mathbb F_Q$-subspaces of $V'$.

By construction, $L_{V'}$ is defined over $k$. Lemma~\ref{lem:affine-ore-dictionary} identifies proper $q$-Frobenius-stable $\mathbb F_Q$-subspaces of $V'$ with proper right factors of $\ell$, proving the irreducibility criterion.

Although the auxiliary coordinate $Z'=cZ$ need not be defined over $\kappa$, its quotient coordinates are: indeed
\[
 X=(Z')^n=u_0Z^n\in \kappa(Z),
 \qquad
 T=L_{V'}(Z')^n=u_0^{Q^r}L_V(Z)^n\in \kappa(Z).
\]
Moreover $\varphi(Z')=bZ'$ geometrically and $b^n=1$, so $X$ is Frobenius-fixed. Since the coefficients of $L_{V'}$ lie in $k$ and $b\in\mathbb F_Q^\times$, one has $L_{V'}(bZ')=bL_{V'}(Z')$; hence $T$ is Frobenius-fixed as well. Thus $X$ and $T$ are genuine $k$-coordinates on the two quotient lines. From
\[
 L_{V'}(Z')=Z'\sum_{i=0}^rc_i\bigl((Z')^n\bigr)^{(Q^i-1)/n}
\]
we obtain \eqref{eq:affine-wild}. Differentiating and using $1+nd_i=Q^i$ in characteristic $p$ gives \eqref{eq:affine-derivative}. A different Hilbert--90 solution replaces $u_0$ by $\lambda u_0$ with $\lambda\in k^\times$; this rescales $X$ by $\lambda$ and $T$ by $\lambda^{Q^r}$, so after restoring the monic leading term it induces exactly the weighted $k^\times$-action displayed in the theorem. Hence the resulting weighted orbit is independent of the auxiliary descent normalization.

For $n>1$, write the parenthesized factor in \eqref{eq:affine-wild} as $P(X)$. Since $L_{V'}(Z)=ZP(Z^n)$ is separable, $P$ is squarefree; hence the zero fiber has one simple point and all other points of index $n$, while $\infty$ is totally ramified. These distinct ramification partitions force any $k$-M\"obius equivalence to fix $0,\infty$ on both source and target, so it is given by scalings. Comparison of the monic highest terms gives precisely
\(c_i'=c_i u^{d_i-D}\qquad(0\leqslant i<r).\)

It remains to treat $n=1$. Then $H=1$ and $\varphi^d=1$. Write $\varphi(Z)=aZ+b$ semilinearly. The relation $\varphi^d=1$ gives $N_{\kappa/k}(a)=1$, so multiplicative Hilbert~90 allows a $\kappa$-scaling after which $a=1$. In that coordinate the same relation says that the trace of the translation term is zero; additive Hilbert~90 then removes it by a $\kappa$-translation. Thus $\varphi$ becomes the standard coefficient Frobenius, so $V^q=V$, and the same subspace-polynomial and Ore-factorization argument applies. Here
\(F_{\ell,1}(X)=\sum_{i=0}^r c_iX^{Q^i}\)
is a separable additive polynomial, so every finite fiber has $Q^r$ distinct geometric points. Hence a $k$-M\"obius equivalence between polynomial representatives must fix $\infty$ on the target, and then also on the source; both transformations are affine. A source translation contributes only the constant $F_{\ell,1}(v)$ and is absorbed by a target translation. After this reduction, comparison of the monic leading term under the source scaling $X\mapsto uX$ gives
\(c_i'=c_i u^{Q^i-Q^r}=c_i u^{d_i-D},\)
which is again the asserted weighted action.

Over $\bar k$, intermediate subgroups between $H$ and $V\rtimes H$ correspond to arbitrary $\mathbb F_Q$-subspaces of $V$, without a Frobenius-stability condition; this also holds for $n=1$, when $Q=p$. Thus $H$ is maximal in $G$ exactly when $r=1$, proving the geometric indecomposability assertion. Finally, $n=1$ is characterized by the absence of finite ramification; for $n>1$ the finite ramification index recovers $n$, and the degree $Q^r$ then recovers $r$.
\end{proof}

\begin{corollary}[No independent affine shift class]
The finite-field affine form in Theorem~\ref{thm:affine-wild} has no additional translation or shift invariant beyond the normalized Frobenius-stable translation module. The remaining arithmetic invariant is exactly the weighted $k^\times$-orbit of the Ore polynomial $\ell$.
\end{corollary}

\begin{proof}
For $n>1$, the normalized semilinear Frobenius preserves the unique translation fixed point and the second fixed point of the complement, so its additive term vanishes; multiplicative Hilbert~90 absorbs the remaining scalar descent into the Frobenius-stable translation space. For $n=1$, the multiplicative and additive Hilbert--90 normalizations in the proof of Theorem~\ref{thm:affine-wild} remove the scalar and translation parts. What remains is precisely the normalized module encoded by $\ell$.
\end{proof}

For $\ell(\tau)=\tau^r+c_{r-1}\tau^{r-1}+\cdots+c_0$, put
\[
 \mathcal K_\ell(Z)=Z^{Q^r}+c_{r-1}Z^{Q^{r-1}}+\cdots+c_0Z.
\]

\begin{proposition}[Affine kernel identity and reconstruction]
The representative in \eqref{eq:affine-wild} satisfies
\begin{equation}\label{eq:affine-kernel-identity}
 \mathcal K_\ell(Z)^n=F_{\ell,n}(Z^n).
\end{equation}
Moreover, $V'=\ker\mathcal K_\ell$. For $n>1$, after the branch normalization in the proof of Theorem~\ref{thm:affine-wild}, the unique monic $n$th root of $F(Z^n)$ is $\mathcal K_\ell(Z)$; hence the normalized polynomial $F$ recovers the translation space and the Ore polynomial.
\end{proposition}

\begin{proof}
Since $Q^i\equiv1\pmod n$,
\[
 \mathcal K_\ell(Z)=Z\sum_{i=0}^r c_i(Z^n)^{(Q^i-1)/n},
\]
which gives \eqref{eq:affine-kernel-identity}. The root space of the separable $Q$-linearized polynomial is the normalized translation module. The finite branch fiber has divisor $[0]+nD$, so $F(X)=XP(X)^n$ with a unique monic $P$. Frobenius fixes this monic root, hence it lies in $k[X]$, and substituting $X=Z^n$ recovers $\mathcal K_\ell$.
\end{proof}

Although complete decompositions need not satisfy a Jordan--H\"older principle in general, the affine family has a module-theoretic Jordan--H\"older property: its Frobenius-stable intermediate subgroups are exactly the submodules of the associated skew module.

This module also controls decompositions within the affine category.

Retain the normalized translation space $V'$ from
Theorem~\ref{thm:affine-wild}, and let $\Phi:V'\to V'$ be $q$-Frobenius,
$\Phi(v)=v^q$. Put $R_\sigma:=\mathbb F_Q[T,T^{-1};\sigma]$, where
$\sigma(\lambda)=\lambda^q$ and $T\cdot v=\Phi(v)$.
A nonzero module over a ring is \emph{simple} if it has no nonzero proper
submodule. A \emph{composition series} of $V'$ as an $R_\sigma$-module is
a chain
\[
 0=W_0\lneq W_1\lneq\cdots\lneq W_s=V'
\]
whose successive quotients are simple. A module admitting such a series
has \emph{finite length}. Since $V'$ is finite-dimensional over
$\mathbb F_Q$, it has finite length as an $R_\sigma$-module.

\begin{proposition}[Jordan--H\"older invariance for affine decompositions]\label{prop:affine-jh}
With $V'$ and $R_\sigma$ as above, complete $k$-decompositions of the affine
family correspond to composition series of the $R_\sigma$-module $V'$.
Hence all complete affine decompositions have the same length and the same
multiset of factor degrees. If $0=W_0\lneq W_1\lneq\cdots\lneq W_s=V'$ is the corresponding
composition series, then
\[
 \deg f_i=Q^{\dim_{\mathbb F_Q}(W_i/W_{i-1})}.
\]
\end{proposition}

\begin{proof}
After the normalization of Theorem~\ref{thm:affine-wild}, every subgroup between $H=\mu_n$ and $G=V'\rtimes H$ is uniquely $W\rtimes H$ with $W=J\cap V'$. Stability under $H$ makes $W$ an $\mathbb F_Q$-subspace, and arithmetic stability is equivalent to $\Phi(W)=W$. These are exactly the $R_\sigma$-submodules. Theorem~\ref{thm:complete-decomposition} now identifies complete decompositions with composition series, and Jordan--H\"older gives the assertions. This is a module-theoretic statement; it does not assert literal uniqueness of an ordered factorization in the noncommutative Ore ring.
\end{proof}

\subsection{Characteristic \texorpdfstring{$2$}{2} wild dihedral cases}

\begin{proposition}[Characteristic-$2$ dihedral cases]\label{prop:char2-dihedral}
Let $k$ have characteristic $2$. A $k$-indecomposable rational function with
Galois closure of genus zero and wild dihedral geometric group has
$(G,H)=(D_{2\ell},C_2)$, where $\ell$ is an odd prime. It is geometrically
indecomposable. There is exactly one $k$-M\"obius class, represented by
$D_\ell(X,1)$. There is no geometrically decomposable characteristic-$2$
dihedral case.
\end{proposition}

\begin{proof}
The rotation subgroup $C_m$ of $D_{2m}$ is characteristic. Core-freeness gives $H\cap C_m=1$. The subgroup $H$ contains at most one reflection, since the product of two distinct reflections is a nontrivial rotation. Thus $H=1$ or $H\cong C_2$. The first possibility leaves $C_m$ as a proper Frobenius-stable intermediate subgroup, contradicting arithmetic primitivity; hence $H=C_2$. If $m$ is composite, a proper characteristic subgroup of $C_m$ produces a Frobenius-stable intermediate dihedral subgroup; hence $m=\ell$ is prime. Conversely $C_2$ is maximal in $D_{2\ell}$.

The degree-$\ell$ cover has two branch points distinguished by their ramification: one is totally ramified of index $\ell$, while the other is wild and
has ramification partition $\{\!\{1,\underbrace{2,\ldots,2}_{(\ell-1)/2}\}\!\}$. Frobenius preserves these distinct ramification types, so both branch points are $k$-rational. The unique point in the totally ramified fiber and the unique unramified point in the wild fiber are therefore $k$-rational. After normalization the geometric cover is the standard Dickson quotient. Since
\(\mathrm N_{D_{2\ell}}(C_2)/C_2=1,\)
the normalized geometric isomorphism is unique and therefore Frobenius-invariant. It descends to $k$. In characteristic $2$ all nonzero parameters are square-equivalent, giving $D_\ell(X,1)$.
\end{proof}

\subsection{Characteristic \texorpdfstring{$3$}{3} icosahedral cases}

\begin{proposition}[Characteristic-$3$ icosahedral cases]\label{prop:char3-A5}
Let $k=\mathbb F_{3^s}$, and let $f\in k(X)$ be separable and geometrically
indecomposable. Assume that its Galois closure has genus zero and that its
geometric monodromy group is $A_5$. Then $\deg f\in\{5,6,10\}$, and $f$ is
$k$-M\"obius equivalent to the corresponding map in
\begin{align}
 W_5(X)&=X^3(X+1)^2,\label{eq:W5}\\
 W_6(X)&=\frac{(X^2+X-1)^3}{X},\notag\\
 W_{10}(X)&=\frac{X^6(X+1)(1-X)^3}{(X^2+X-1)^5}.\label{eq:W10}
\end{align}
The corresponding geometric point stabilizers are $A_4,D_{10},S_3$.
In each degree there is exactly one $k$-M\"obius class. The arithmetic
monodromy group is $A_5$ when $s$ is even and $S_5$ when $s$ is odd.
Among rational functions with Galois closure of genus zero, no
$k$-indecomposable characteristic-$3$ $A_5$ class is geometrically decomposable.
\end{proposition}

\begin{proof}
Geometric indecomposability makes the geometric point stabilizer a core-free
maximal subgroup of $A_5$. These subgroups have types $A_4,D_{10},S_3$, of
indices $5,6,10$, respectively. Faber's classification gives a single
$\operatorname{PGL}_2(\bar k)$-conjugacy class of characteristic-$3$ subgroups
isomorphic to $A_5$ \cite[Theorem~B]{Faber2023}. Each of the three
maximal-subgroup types also forms a single $A_5$-conjugacy class. Since the
Galois-closure curve is a projective line, there is therefore at most one
$\bar k$-M\"obius class for each pair. We show that the displayed maps realize
these pairs and then determine their forms over $k$.

For $d=5,6,10$, let $F_d=1728I_d=N_d/D_d$ be the unscaled icosahedral map
from Propositions~\ref{prop:I5},~\ref{prop:I6}, and~\ref{prop:I10}, using
the displayed integral numerators and denominators with the factor $1728$
removed. Put $P_d(X,T)=N_d(X)-TD_d(X)$. Substituting $t=T/1728$ into
\eqref{eq:I5-disc}, \eqref{eq:I6-disc}, and~\eqref{eq:I10-disc} gives
\[
\begin{aligned}
 \operatorname{disc}_X(P_5)&=5^5T^2(T-1728)^2,\\
 \operatorname{disc}_X(P_6)&=5^5T^4(T-1728)^2,\\
 \operatorname{disc}_X(P_{10})&=5^{75}T^6(T-1728)^4.
\end{aligned}
\]
Direct coefficient reduction modulo $3$ sends $F_5,F_6,F_{10}$ to
$W_5,W_6,W_{10}$ in \eqref{eq:W5}--\eqref{eq:W10}, respectively, without
cancellation of numerator and denominator.

For specialization, let $K=\mathbf C_3$ be the completion of an algebraic
closure of $\mathbb Q_3$, with $3$-adic valuation $v$. Equip $K(T)$ with the
\emph{Gauss valuation}, defined for nonzero polynomials by
\[
 v_G\!\left(\sum_i a_iT^i\right):=\min_i v(a_i),
 \qquad v_G(P/Q):=v_G(P)-v_G(Q).
\]
Its residue field is $\overline{\mathbb F}_3(T)$.
The characteristic-zero monodromy calculations in
Propositions~\ref{prop:I5}, \ref{prop:I6}, and~\ref{prop:I10}
give generic splitting group $A_5$.
The leading coefficients of $P_5,P_6,P_{10}$ are $1,1,1000-T$;
these and the displayed discriminants all have Gauss valuation $0$.
Let $K(T)^h$ denote the henselization of $(K(T),v_G)$.
After division by the leading coefficient, each polynomial is monic,
$v_G$-integral, and has separable reduction.
Hensel's lemma shows that its splitting extension over $K(T)^h$ has
trivial inertia at $v_G$. The residue-field splitting group therefore
identifies with a subgroup of the generic Galois group.
Since the residue field is $\overline{\mathbb F}_3(T)$, this is the
geometric monodromy group $\bar G_d$ of $W_d$, so $\bar G_d\leqslant A_5$.

The ramification partitions at $0$ are $\{\!\{3,2\}\!\}$,
$\{\!\{3,3\}\!\}$, and $\{\!\{6,3,1\}\!\}$; those at $\infty$ are
$\{\!\{5\}\!\}$, $\{\!\{5,1\}\!\}$, and $\{\!\{5,5\}\!\}$,
respectively.
These ramification indices force $3$ and $5$ to divide $|\bar G_d|$. No
proper subgroup of $A_5$ has order divisible by both primes, so
$\bar G_d=A_5$. The point stabilizer has order $60/d$ and is therefore
$A_4,D_{10}$, or $S_3$, respectively.

We verify the genus of each special Galois closure directly. Put
$B(X)=X^2+X-1$. In characteristic $3$,
\[
\begin{aligned}
 W_5'(X)&=-X^3(X+1),\\
 W_6'(X)&=-\frac{B(X)^3}{X^2},\\
 W_{10}'(X)&=-\frac{X^7(X-1)^3}{B(X)^6}.
\end{aligned}
\]
The different exponents at the points above $0$ are, in the order of the entries in the ramification partitions just given, $(3,1)$, $(3,3)$, and $(7,3,0)$.
The tame fibers above $\infty$ contribute $4,4,8$ in total. The total
different degrees are therefore $8,10,18$, respectively, equal to $2d-2$;
there are no other branch points.

For a geometric point $P$ of a Galois cover, let $I$ be its inertia
group and let $z$ be a local parameter at $P$. We write
\[
 I_i:=\{\sigma\in I:
 \operatorname{ord}_P(\sigma(z)-z)\geqslant i+1\}
 \qquad(i\geqslant0)
\]
for the \emph{lower ramification groups}; these groups are independent
of the choice of $z$. Thus $I_0=I$, and $I_1$ is the
\emph{wild inertia subgroup}. Hilbert's different formula gives
\[
 d_P=\sum_{i\geqslant0}(|I_i|-1),
\]
where $d_P$ is the different exponent of this Galois cover at $P$.

Let $I_0$ be inertia in the $A_5$-Galois closure above $0$. Its wild inertia subgroup $I_1$ is a nontrivial $3$-subgroup of $A_5$, hence isomorphic to $C_3$.
Since it is normal in $I_0$, one has
$I_0\leqslant\mathrm N_{A_5}(I_1)\cong S_3$, so $I_0$ is $C_3$ or $S_3$.
In degrees $5$ and $10$, the even ramification indices rule out $C_3$.
In degree $6$, if $I_0=C_3$, an index-$3$ point would give a cyclic cubic
local extension. Its first two lower ramification groups would both have
order $3$, so its different exponent would be at least $4$, contradicting
the exponent $3$ above. Thus $I_0=S_3$ in all three degrees.

Write the lower ramification filtration as $I_0=S_3$,
$I_1=\cdots=I_j=C_3$, and $I_{j+1}=1$, with $j\geqslant1$.
Its different exponent is $5+2j$. In degrees $5$ and $6$, an index-$3$
point has different exponent $3$, and the remaining quadratic step to the
local Galois closure is tame. Transitivity of the different gives
$5+2j=1+2\cdot3=7$. In degree $10$, the index-$6$ point is already the full
local extension and has different exponent $7$. Hence $j=1$ in every case:
$I_0=S_3$, $I_1=C_3$, and $I_2=1$. Above $\infty$ the inertia is tame cyclic
of order $5$. Riemann--Hurwitz for the degree-$60$ Galois closure shows that
twice the genus minus two equals
\[
 -120+\frac{60}{6}\cdot7+\frac{60}{5}\cdot4=-2.
\]
Thus the Galois-closure curve has genus zero, and the $W_d$ realize the
three geometric classes.

The integral discriminants reduce to $-T^4,-T^6,-T^{10}$, respectively.
For $G=A_5$ acting on $G/H$, its permutation centralizer is
$\mathrm N_G(H)/H=1$. Hence conjugation embeds the arithmetic group into
$\operatorname{Aut}(A_5)=S_5$, with geometric subgroup $A_5$.
A transposition in $S_5$ acts on the five letters, the six Sylow-$5$
subgroups, and the ten Sylow-$3$ subgroups with cycle types
$2\,1^3$, $2^3$, and $2^3\,1^4$, respectively. For the latter two actions,
it normalizes no Sylow-$5$ subgroup and fixes precisely the four
three-element supports containing both or neither of its moved letters.
Thus the outer coset is odd in all three actions. The discriminant
criterion gives arithmetic group $A_5$ exactly when $-1$ is a square in
$\mathbb F_{3^s}$, equivalently when $s$ is even; otherwise the group is
$S_5$. The full constant-field degree is correspondingly $1$ or $2$.

Finally, in characteristic $3$,
\[
 \mathrm N_{\operatorname{PGL}_2(\bar k)}(A_5)=A_5.
\]
Indeed, an element centralizing $A_5$ fixes the distinct fixed points of
its Sylow-$3$ subgroups, so this centralizer is trivial. The normalizer
therefore embeds into $\operatorname{Aut}(A_5)=S_5$. A larger normalizer
would be $S_5$ of order $120$, which is excluded by Faber's classification
\cite[Theorem~B]{Faber2023}. After conjugating two geometric pairs to the
same pair $(A_5,H)$, their semilinear $q$-Frobenius lifts differ by an element of $\mathrm N_{\operatorname{PGL}_2(\bar k)}(G)$. It also normalizes $H$, so it belongs
to $\mathrm N_{A_5}(H)=H$. The two lifts induce the same descent on both
quotient curves. Hence there is at most one $k$-M\"obius class in each
degree; the maps $W_d\in\mathbb F_3(X)$ give existence over every
$\mathbb F_{3^s}$.

The $A_5$ argument in the proof of Theorem~\ref{thm:arith-only} uses only
$\operatorname{Aut}(A_5)=S_5$ and the maximal subgroups of $S_5$, not
tameness. Applied to a genus-zero Galois closure in characteristic $3$,
it excludes every nonmaximal geometric stabilizer that could yield
primitive arithmetic monodromy. This proves the final assertion.
\end{proof}

\section{Geometrically primitive Lie-type wild quotients}\label{sec:wild-lie}
We treat the geometrically primitive defining-characteristic Lie-type cases: Borel, split- and nonsplit-Cartan, subfield, and cases with stabilizer $A_4,S_4$, or $A_5$. The Cartan formulas are stated uniformly; their small nonmaximal specializations at $Q=7,9,11$ are assigned in Section~\ref{sec:wild-novelty} to the $k$-indecomposable but geometrically decomposable cases.
\subsection{Dickson invariants and fixed-field reconstruction}

For $Q=p^a$ define the classical Dickson invariant
\begin{equation}\label{eq:DQ}
 D_Q(Z)=
 \frac{(Z^{Q^2}-Z)^{Q+1}}{(Z^Q-Z)^{Q^2+1}}.
\end{equation}
When $Q$ is odd define its $\operatorname{PSL}_2(Q)$ square root
\begin{equation}\label{eq:JQ}
 J_Q(Z)=
 \frac{(Z^{Q^2}-Z)^{(Q+1)/2}}
 {(Z^Q-Z)^{(Q^2+1)/2}}.
\end{equation}
Then $\bar k(Z)^{\operatorname{PGL}_2(Q)}=\bar k(D_Q)$, $\bar k(Z)^{\operatorname{PSL}_2(Q)}=\bar k(J_Q)$, and for $g\in\operatorname{PGL}_2(Q)$,
\[
 J_Q(gZ)=\chi_2(\det g)J_Q(Z),
\]
where $\chi_2$ is the quadratic character. These invariant-field constructions are closely related to the finite-field invariant theory developed by Gow--McGuire \cite{GowMcGuire2022}.

\begin{proposition}[Finite-field elimination for quotient maps]
\label{prop:exact-finite-elimination}
Let $F$ be a finite field, and let $h=N_h/D_h$ and $j=N_j/D_j$ be nonconstant
rational functions in $F(Z)$ such that $F(j)\subseteq F(h)$. Put
$d:=[F(h):F(j)]$. Then there is a unique rational function $R=P/Q\in F(X)$
of degree $d$ such that $j=R(h)$. Writing
$P(X)=\sum_{i=0}^d p_iX^i$ and $Q(X)=\sum_{i=0}^d q_iX^i$, the coefficients
of $P$ and $Q$, up to a common nonzero scalar, are recovered from the
homogeneous linear system obtained by equating coefficients in
\begin{equation}\label{eq:exact-finite-elimination}
 D_j(Z)\sum_{i=0}^d p_iN_h(Z)^iD_h(Z)^{d-i}
 -
 N_j(Z)\sum_{i=0}^d q_iN_h(Z)^iD_h(Z)^{d-i}
 =0.
\end{equation}
\end{proposition}

\begin{proof}
Since $F(j)\subseteq F(h)$ and $F(h)$ is a rational function field in the
transcendental element $h$, there is a unique $R\in F(X)$ with $j=R(h)$.
Moreover,
\(\deg R=[F(h):F(R(h))]=[F(h):F(j)]=d.\)
Writing $R=P/Q$ with $\max\{\deg P,\deg Q\}=d$ and clearing denominators gives
\eqref{eq:exact-finite-elimination}. Conversely, any nonzero solution
representing a reduced degree-$d$ pair $(P,Q)$ gives the required identity,
and uniqueness of $R$ determines the solution projectively.
\end{proof}

\begin{proposition}[Fixed-field generators and quotient reconstruction]
Let $F$ be a finite field containing the matrix coefficients of a chosen embedded
pair $H\leqslant G\leqslant\operatorname{PGL}_2(F)$. Define
\[
 \Psi_H(T,Z)=\prod_{h\in H}(T-h(Z))
 =T^{|H|}+a_1(Z)T^{|H|-1}+\cdots+a_{|H|}(Z).
\]
Then
\[
 F(Z)^H=F(a_1,\ldots,a_{|H|}).
\]
A single generator $h_H$ of $F(Z)^H$ can be found by enumerating the finite
set of scalar classes of coprime pairs $P,S\in F[Z]$ with
$\max(\deg P,\deg S)=|H|$, and testing $P/S$ for $H$-invariance. Likewise
choose a generator $j_G$ of $F(Z)^G$. The quotient function $R$ satisfying
$j_G(Z)=R(h_H(Z))$ is recovered by
Proposition~\ref{prop:exact-finite-elimination}.
\end{proposition}

\begin{proof}
Every coefficient $a_i(Z)$ is $H$-invariant. Put
\(F_0=F(a_1,\ldots,a_{|H|}).\)
Since $Z$ satisfies $\Psi_H(T,Z)\in F_0[T]$, one has
\([F(Z):F_0]\leqslant |H|.\)
But
\(F_0\subseteq F(Z)^H,\qquad [F(Z):F(Z)^H]=|H|,\)
so equality holds and $F_0=F(Z)^H$. L\"uroth's theorem gives a generator of
$F(Z)^H$, necessarily of degree $|H|$; because $F$ is finite, the stated
finite search reaches one. Finally
\([F(h_H):F(j_G)]=[G:H],\)
so Proposition~\ref{prop:exact-finite-elimination} recovers the unique
degree-$[G:H]$ quotient function.
\end{proof}

\subsection{Borel quotients}

\begin{proposition}[Primitive Borel quotients]\label{prop:borel-wild}
Let $Q=p^a\geqslant3$ and $k=\mathbb F_{p^s}$, and let $B$ denote a point stabilizer in the standard projective action, that is, a Borel subgroup.
\begin{enumerate}[label=(\arabic*),ref=\arabic*,font=\itshape]
\item For the geometric pair $(\operatorname{PGL}_2(Q),B)$ there is exactly one $k$-M\"obius class, represented by
\begin{equation}\label{eq:Bplus}
 \mathcal B_Q^+(X)=\frac{(X+1)^{Q+1}}{X^Q}.
\end{equation}
The same statement applies to $\operatorname{PSL}_2(Q)$ when $Q$ is even.
\item If $Q$ is odd, the pair $(\operatorname{PSL}_2(Q),B)$ has exactly two $k$-M\"obius classes,
\begin{equation}\label{eq:Bminus}
 \mathcal B_{Q,\eta}^-(X)=
 \frac{(X^2+\eta)^{(Q+1)/2}}{X^Q},
 \qquad \eta\in k^\times/(k^\times)^2.
\end{equation}
Both are geometrically primitive. For $Q=3$, these are precisely the two characteristic-$3$ tetrahedral classes with pair $(A_4,C_3)$; the PGL form in part~\emph{(1)} is the unique characteristic-$3$ octahedral class with pair $(S_4,S_3)$.
\end{enumerate}
\end{proposition}

\begin{proof}
Put $L(Z)=Z^Q-Z$. For the PGL Borel, $X=L(Z)^{Q-1}$ generates the fixed field; using $Z^{Q^2}-Z=L(L^{Q-1}+1)$ in \eqref{eq:DQ} gives \eqref{eq:Bplus}. For odd $Q$, $X=L(Z)^{(Q-1)/2}$ generates the PSL Borel fixed field and \eqref{eq:JQ} becomes \eqref{eq:Bminus} with $\eta=1$. The diagonal outer involution acts by $X\mapsto-X$ and $J_Q\mapsto-J_Q$, producing the nonsquare quadratic twist. The two square classes are inequivalent because the two tame points over the zero branch point have different splitting types over $k$. For the PGL pair, $\mathrm N_{\operatorname{PGL}_2(\bar k)}(G)=G$, and $\mathrm N_G(B)=B$. Two Frobenius lifts with the prescribed field component therefore differ by an element of $B$, so they define the same source and target fixed fields. This proves uniqueness, including the even case and $Q=3$. For odd $Q>3$, Lemma~\ref{lem:outer-rigidity} proves exhaustiveness of the two PSL forms. At $Q=3$, one has $\mathrm N_{A_4}(C_3)=C_3$ and $\mathrm N_{S_4}(C_3)/C_3\cong C_2$; the same lift argument gives at most the two forms already distinguished by their splitting types.
\end{proof}

\subsection{Split Cartan normalizers}

Let $W=Z^{Q-1}$ and define $S_Q\in\mathbb F_p(X)$ by
\begin{equation}\label{eq:SQ}
 S_Q(W+W^{-1})=
 \frac{(W^{Q+1}-1)^{Q+1}}
 {W^Q(W-1)^{Q^2+1}}.
\end{equation}
Then $\deg S_Q=Q(Q+1)/2$ and $D_Q(Z)=S_Q(W+W^{-1})$.

\begin{proposition}[Split-Cartan quotients]\label{prop:split-cartan}
For $Q>3$, the $\operatorname{PGL}_2(Q)$ split-Cartan normalizer has exactly one $k$-form, represented by $S_Q$; it is geometrically primitive exactly when $Q\ne5$. At $Q=3$ the PGL split-Cartan pair has two $k$-forms, one represented by $S_3$. If $Q$ is odd, put $\epsilon=(-1)^{(Q-1)/2}$. Put $W_0=Z^{(Q-1)/2}$ and $Y=W_0+\epsilon W_0^{-1}$. There is a unique rational function $T_Q\in\mathbb F_p(X)$ characterized by
\begin{equation}\label{eq:TQ-def}
 J_Q(Z)=T_Q(Y).
\end{equation}
It is odd and satisfies
\begin{equation}\label{eq:TQsquare}
 T_Q(X)^2=S_Q(X^2-2\epsilon).
\end{equation}
For $\eta\in k^\times$, choose $\theta\in\bar k$ with $\theta^2=\eta$ and put $T_{Q,\eta}(X)=\theta T_Q(X/\theta)$. Since $T_Q$ is odd, $T_{Q,\eta}\in k(X)$, and its $k$-M\"obius class depends only on the square class of $\eta$. For odd $Q\geqslant7$, as $\eta$ ranges over the two square classes of $k^\times$, these give exactly the two $k$-forms of the $\operatorname{PSL}_2(Q)$ split-Cartan pair. For odd $Q\geqslant13$ both are geometrically primitive. At $Q=7,9,11$, Theorem~\ref{thm:wild-novelty} identifies exactly the descent with diagonal component as the geometrically decomposable $k$-indecomposable class, of degree $28,45,66$, respectively; the other descent is $k$-decomposable.

At $Q=3$ and $5$, the PSL split-Cartan pair has four and three $k$-forms, respectively, rather than just the displayed square-class forms. All split-Cartan forms at $Q=3,5$, for either geometric group, are $k$-decomposable.
\end{proposition}

\begin{proof}
The fixed field of the PGL split normalizer is generated by $W+W^{-1}$, which gives \eqref{eq:SQ}. For PSL, the fixed-field inclusion gives the unique factor $T_Q$ in \eqref{eq:TQ-def}. Since $W=W_0^2$,
\(W+W^{-1}=Y^2-2\epsilon,\)
and $J_Q^2=D_Q$, which yields \eqref{eq:TQsquare}. A diagonal nonsquare scalar sends both $Y$ and $J_Q$ to their negatives, so $T_Q$ is odd and the two diagonal square classes both occur. For odd $Q\geqslant7$, the PSL stabilizer is self-normalizing in $\operatorname{PSL}_2(Q)$, so Lemma~\ref{lem:outer-rigidity} shows that the two forms are inequivalent and exhaustive. For $Q>3$, the PGL split-Cartan normalizer is self-normalizing in $\operatorname{PGL}_2(Q)$: its characteristic rotation subgroup determines its unordered eigenline pair. Since $\mathrm N_{\operatorname{PGL}_2(\bar k)}(G)=G$ for the PGL geometric group, two lifts with the same prescribed field component differ by an element of $H$. This proves the asserted PGL uniqueness. The geometric maximality assertions are precisely \cite[Theorem~2.2]{Giudici2007} and, for PGL, \cite[Theorem~3.5]{Giudici2007}.

For completeness, the small-field form counts follow directly from the exact equality criterion in Theorem~\ref{thm:exact-map-equivalence}. The automorphism group of the geometric quotient $C/H\to C/G$ is
\[
 \{a\in \mathrm N_{\operatorname{PGL}_2(\bar k)}(G):aHa^{-1}=H\}/H.
\]
For the split pairs with $(Q,G)=(3,\operatorname{PGL}_2(3))$, $(3,\operatorname{PSL}_2(3))$, and $(5,\operatorname{PSL}_2(5))$, these groups are respectively $C_2$, $C_2\times C_2$, and $S_3$. Indeed, the corresponding full pair normalizers are $D_8,D_8,S_4$, with $H$ of order $4,2,4$. All are defined over the prime field. Frobenius consequently acts trivially on these quotient groups, so their finite-field descents are indexed by conjugacy classes, numbering $2,4,3$. Every such descent exists: send Frobenius to a representative of the desired conjugacy class, and descend the two quotient curves; over a finite field both descended genus-zero curves are projective lines.

At $Q=3$ on the PSL side, $H=C_2\lneq \mathrm N_G(H)=V_4\lneq G=A_4$; on the PGL side, the split $V_4$ has normalizer $D_8\lneq S_4$ (it is not the normal Klein four subgroup). At $Q=5$ on the PSL side, $H=V_4\lneq \mathrm N_G(H)=A_4\lneq A_5$. A lift normalizing $H$ normalizes $\mathrm N_G(H)$, so each of these chains is Frobenius-stable. For the unique PGL form at $Q=5$, Frobenius acts on $G$ by conjugation by an element of $H$, and $D_8\lneq S_4\lneq S_5$ gives a stable intermediate subgroup. Theorem~\ref{thm:decomposition} proves decomposability in all these cases.
\end{proof}

\subsection{Nonsplit Cartan normalizers}

Choose $\alpha\in\mathbb F_{Q^2}\setminus\mathbb F_Q$ and put
\[
 T=\frac{Z-\alpha}{Z-\alpha^Q},
 \qquad W=T^{Q+1}.
\]
Define $N_Q\in\mathbb F_p(X)$ by
\begin{equation}\label{eq:NQ}
 N_Q(W+W^{-1})=
 -\frac{W(W^{Q-1}-1)^{Q+1}}{(W-1)^{Q^2+1}}.
\end{equation}
Then $\deg N_Q=Q(Q-1)/2$.

\begin{proposition}[Nonsplit-Cartan quotients]\label{prop:nonsplit-cartan}
For $Q>3$, the $\operatorname{PGL}_2(Q)$ nonsplit-Cartan normalizer has exactly one $k$-form, represented by $N_Q$, and it is geometrically primitive. If $Q$ is odd and $\epsilon=(-1)^{(Q-1)/2}$, there is a unique normalized odd $M_Q\in\mathbb F_p(X)$ satisfying
\begin{equation}\label{eq:MQsquare}
 M_Q(X)^2=N_Q(-X^2-2\epsilon),
 \qquad
 M_Q(X)=X^{-1}+O(X^{-2})\text{ at }\infty.
\end{equation}
For $\eta\in k^\times$, choose $\theta\in\bar k$ with $\theta^2=\eta$ and put $M_{Q,\eta}(X)=\theta M_Q(X/\theta)$. Since $M_Q$ is odd, $M_{Q,\eta}\in k(X)$, and its $k$-M\"obius class depends only on the square class of $\eta$. For odd $Q\geqslant5$, as $\eta$ ranges over the two square classes of $k^\times$, these give exactly the two $k$-forms of the $\operatorname{PSL}_2(Q)$ nonsplit-Cartan pair. They are geometrically primitive exactly when $Q\ne7,9$. At $Q=7,9$ exactly one of the two descents is the geometrically decomposable $k$-indecomposable class of degree $21,36$, respectively. At $Q=3$ the nonsplit stabilizer is $V_4\triangleleft A_4\cong\operatorname{PSL}_2(3)$, so it is not core-free.
\end{proposition}

\begin{proof}
In the coordinate $T=(Z-\alpha)/(Z-\alpha^Q)$, the nonsplit Cartan subgroup
of $\operatorname{PGL}_2(Q)$ acts by $T\mapsto\zeta T$, where
$\zeta^{Q+1}=1$, and its normalizer is generated by these scalings and
$T\mapsto T^{-1}$. Writing $\Delta=\alpha-\alpha^Q$ and $Z=(\alpha^QT-\alpha)/(T-1)$, direct substitution into \eqref{eq:DQ}, with $W=T^{Q+1}$ and $\Delta^Q=-\Delta$, gives \eqref{eq:NQ}, including its sign and independence of $\alpha$. For $Q>3$, the nonsplit-Cartan normalizer is self-normalizing in $\operatorname{PGL}_2(Q)$: its characteristic rotation subgroup determines the conjugate eigenline pair. As in the PGL Borel case, the prescribed field component and $\mathrm N_G(H)=H$ make two normalizing lifts differ by $H$, proving uniqueness of the descended quotient. For odd $Q\geqslant5$, the PSL fixed coordinate has square $-X^2-2\epsilon$, yielding \eqref{eq:MQsquare}; the normalization at infinity is Frobenius-stable and changes sign under $X\mapsto-X$, so the chosen root lies in $\mathbb F_p(X)$ and is odd. The diagonal outer image gives the second square-class form, and Lemma~\ref{lem:outer-rigidity} shows that the two forms are inequivalent and exhaustive. The maximality exceptions $Q=7,9$ are those of \cite[Theorem~2.2]{Giudici2007}.
\end{proof}

\subsection{Subfield quotients}

\begin{proposition}[Primitive subfield quotients]\label{prop:subfield-wild}
Suppose $Q=Q_0^\ell$ with $\ell$ prime and the corresponding subfield subgroup is maximal. The exact number of $k$-M\"obius classes is
\[
\begin{array}{c|c|c|c}
G&H&\text{condition}&\#\text{ classes}\\ \hline
\operatorname{PGL}_2(Q)=\operatorname{PSL}_2(Q)&\operatorname{PGL}_2(Q_0)&Q\text{ even},\ Q_0\ne2&1\\
\operatorname{PSL}_2(Q)&\operatorname{PSL}_2(Q_0)&Q\text{ odd},\ \ell\text{ odd}&2\\
\operatorname{PSL}_2(Q)&\operatorname{PGL}_2(Q_0)&Q\text{ odd},\ \ell=2&1\\
\operatorname{PGL}_2(Q)&\operatorname{PGL}_2(Q_0)&Q\text{ odd},\ \ell\text{ odd}&1.
\end{array}
\]
Exact invariant-field representatives are determined by
\[
\begin{aligned}
 D_Q&=R^+(D_{Q_0}) &&\text{for either PGL case},\\
 J_Q&=R^-(J_{Q_0}) &&\text{for the odd-PSL odd-exponent case},\\
 J_Q&=R^{\mathrm{sq}}(D_{Q_0}) &&\text{for the odd square-subfield PSL case}.
\end{aligned}
\] In the odd-PSL odd-exponent case, $R^-$ is odd; its two $k$-M\"obius classes are the square and nonsquare quadratic twists. The small-field realization $Q=4,Q_0=2$ coincides with the split-Cartan $S_3$ case of Proposition~\ref{prop:split-cartan} and is counted there.
\end{proposition}

\begin{proof}
Existence and uniqueness of the rational factors follow from the fixed-field inclusions and subgroup indices, and finite elimination makes them effective. All the invariant functions involved lie in $\mathbb F_p(Z)$, so uniqueness gives $R^+,R^-,R^{\mathrm{sq}}\in\mathbb F_p(X)$. In the odd-PSL odd-exponent case, a diagonal element over $\mathbb F_{Q_0}$ remains outer over $\mathbb F_Q$ because $\ell$ is odd; it negates both $J_{Q_0}$ and $J_Q$, forcing $R^-$ to be odd and giving the second quadratic twist. Lemma~\ref{lem:outer-rigidity} separates the two forms. In the odd square-subfield case, the diagonal outer automorphism exchanges the two $G$-classes of $\operatorname{PGL}_2(Q_0)$ while field automorphisms preserve each \cite[Lemma~2.3]{Giudici2007}; hence only the field outer image stabilizes a fixed class. Conjugating the entire semilinear datum by a diagonal M\"obius transformation transports the unique descent from one class to the other and gives equality of the resulting $k$-M\"obius classes, so the two geometric classes yield one $k$-M\"obius class even if that transformation is not $k$-rational. The PGL cases have no further diagonal freedom, and $\mathrm N_G(H)=H$ excludes additional twists.
\end{proof}

\subsection{Cases with stabilizer \texorpdfstring{$A_4,S_4$, or $A_5$}{A4, S4, or A5}}

For $Q=p^a$, write
\[
 \operatorname{P\Gamma L}_2(Q)
 :=\operatorname{PGL}_2(Q)\rtimes
 \operatorname{Gal}(\mathbb F_Q/\mathbb F_p)
\]
for the \emph{projective semilinear group}, where field automorphisms act
on matrix entries. We write $\operatorname{P\Sigma L}_2(Q)$ for its
subgroup generated by $\operatorname{PSL}_2(Q)$ and these field
automorphisms.

\begin{proposition}[Primitive quotients with the stabilizers $A_4,S_4,A_5$]\label{prop:exceptional-wild}
For the four prime-field cases below, assume $p\geqslant5$; the square-field $A_5$ case also allows $p=3$. The primitive Lie-type cases with $H\cong A_4,S_4$, or $A_5$ have the following finite-field class counts:
\[
\begin{array}{c|c|c|c}
G&H&\text{parameter range}&\#\text{ classes}\\ \hline
\operatorname{PSL}_2(p)&A_4&p\equiv\pm3\ (8),\ p\not\equiv\pm1\ (10)&2\\
\operatorname{PSL}_2(p)&S_4&p\equiv\pm1\ (8)&1\\
\operatorname{PSL}_2(p)&A_5&p\equiv\pm1\ (10)&1\\
\operatorname{PSL}_2(p^2)&A_5&p\equiv\pm3\ (10)&1\\
\operatorname{PGL}_2(p)&S_4&p\equiv\pm3\ (8)&1.
\end{array}
\]
All are geometrically primitive. In the first case the two arithmetic groups are $\operatorname{PSL}_2(p)$ and $\operatorname{PGL}_2(p)$. In the square-field $A_5$ case, over $k=\mathbb F_{p^s}$ the arithmetic group is $\operatorname{PSL}_2(p^2)$ for $s$ even and $\operatorname{P\Sigma L}_2(p^2)$, with point stabilizer $S_5$, for $s$ odd. Exact invariant-field representatives are obtained by factoring $J_Q$ or $D_Q$ through tetrahedral, octahedral, or icosahedral fixed-field generators.
\end{proposition}

\begin{proof}
The required statements about subgroup conjugacy classes and the action of outer automorphisms are given by \cite[Theorem~2.2 and Lemma~2.3]{Giudici2007}. In the maximal $A_4$ case there is a single $\operatorname{PSL}_2(p)$-conjugacy class of $A_4$ subgroups, and
\[
 \mathrm N_{\operatorname{PGL}_2(p)}(A_4)/A_4\cong C_2,
\]
so a genuine diagonal quadratic twist remains. The induced involution on $\mathbf P^1/A_4$ is split: the degree-$2$ cover $\mathbf P^1/A_4\to\mathbf P^1/S_4$ ramifies over the distinct $S_4$ branch points with inertia orders $2$ and $4$, hence its two fixed points are $\mathbb F_p$-rational. Thus it is conjugate to $X\mapsto-X$ and produces exactly two square-class forms; Lemma~\ref{lem:outer-rigidity} shows that they are inequivalent.

For maximal $S_4$ and prime-field $A_5$, there are two $\operatorname{PSL}_2(p)$-conjugacy classes of such subgroups, and the diagonal outer automorphism interchanges these two $\operatorname{PSL}_2(p)$-conjugacy classes; consequently it cannot normalize a fixed stabilizer. The two geometric subgroup classes are nevertheless identified by the diagonal M\"obius transformation and a target sign, leaving one $k$-M\"obius class. For $Q=p^2$ and $H=A_5$, the diagonal outer automorphism interchanges the two $\operatorname{PSL}_2(p^2)$-conjugacy classes of subgroups while the field automorphism preserves each; \cite[Proposition~3.1]{Giudici2007} gives the maximal $S_5$ normalizer in the field extension. Conjugating the entire semilinear datum by a diagonal element transports the unique descent with field-automorphism outer image between the two classes, and the exact equality criterion identifies the descended $k$-M\"obius classes. The invariant-field factor is first constructed over $\mathbb F_{p^2}$; when $k$ does not contain that field, the $k$-representative is obtained by semilinear descent. For $p=3$ one uses the characteristic-$3$ $A_5$ fixed field from the earlier characteristic-$3$ case, not a tame icosahedral specialization.
\end{proof}

\section{Wild classification: Lie-type novelties and completion}\label{sec:wild-novelty}

The remaining defining-characteristic Lie cases are the counterpart of the geometrically decomposable tame quartic and the higher-rank affine phenomenon: the geometric stabilizer is not maximal, but adjoining Frobenius makes the arithmetic point stabilizer maximal. Thus $k$-primitivity can persist even when geometric primitivity fails.

Let $G=\operatorname{PSL}_2(Q)\triangleleft A\leqslant\operatorname{P\Gamma L}_2(Q)$ and let $U$ be the arithmetic point stabilizer. Then $H=U\cap G$.
A rational function is $k$-indecomposable but geometrically decomposable precisely when $U$ is maximal in $A$ while $H$ is not maximal in $G$. This is exactly a \emph{novelty maximal subgroup} in Giudici's terminology. Giudici's Theorem~1.1 and Table~1 classify all such novelties \cite{Giudici2007}. Since over a finite field $A/G\cong\operatorname{Gal}(\kappa/k)$ is cyclic, the two $\operatorname{P\Gamma L}_2(9)$ novelty cases with quotient $C_2\times C_2$ are excluded.

For $Q=9$ we follow Giudici's notation and write $M_{10}:=M(1,9)$ for the extension of $\operatorname{PSL}_2(9)$ generated by $\operatorname{PSL}_2(9)$ and the product of the diagonal and field involutions.

\begin{theorem}[Complete Lie-type novelty list]\label{thm:wild-novelty}
Every $k$-indecomposable but geometrically decomposable Lie-type rational function with Galois closure of genus zero has one of the following group data; conversely each case occurs over the indicated finite fields and gives exactly one $k$-M\"obius class.
\[
\begin{array}{c|c|c|c|c}
G&A&H&U&\deg\\ \hline
\operatorname{PSL}_2(7)&\operatorname{PGL}_2(7)&D_6&D_{12}&28\\
\operatorname{PSL}_2(7)&\operatorname{PGL}_2(7)&D_8&D_{16}&21\\
\operatorname{PSL}_2(9)&\operatorname{PGL}_2(9)&D_{10}&D_{20}&36\\
\operatorname{PSL}_2(9)&\operatorname{PGL}_2(9)&D_8&D_{16}&45\\
\operatorname{PSL}_2(9)&M_{10}&D_{10}&\operatorname{AGL}_1(5)&36\\
\operatorname{PSL}_2(9)&M_{10}&D_8&\operatorname{SD}_{16}&45\\
\operatorname{PSL}_2(11)&\operatorname{PGL}_2(11)&D_{10}&D_{20}&66\\
\operatorname{PSL}_2(p)&\operatorname{PGL}_2(p)&A_4&S_4& p(p^2-1)/24
\end{array}
\]
Here $\operatorname{SD}_{16}$ is the semidihedral group of order $16$, with presentation $\langle a,b\mid a^8=b^2=1,\ bab^{-1}=a^3\rangle$. In the degree-$36$ row, $\operatorname{AGL}_1(5)=\mathbb F_5\rtimes\mathbb F_5^\times$ uses the faithful multiplication action. The last case occurs for primes $p\equiv\pm11,\pm19\pmod{40}$.
For $Q=7$ and $11$ these classes occur over every $\mathbb F_{Q^s}$; the last case occurs over every $\mathbb F_{p^s}$ in the displayed prime range. For $k=\mathbb F_{3^s}$, the $Q=9$ novelty group is $\operatorname{PGL}_2(9)$ when $s$ is even and $M_{10}$ when $s$ is odd; hence for fixed $k$ there is only one novelty class in each of degrees $36,45$. No such geometrically decomposable $k$-indecomposable Lie-type rational function has geometric group $\operatorname{PGL}_2(Q)$.
\end{theorem}

\begin{proof}
Apply Giudici's classification with his socle $T$ equal to $G$, his maximal subgroup $M$ equal to $U$, and $M\cap T=H$. Completeness of the group data follows from \cite[Theorem~1.1, Corollary~1.2 and Table~1]{Giudici2007}, while \cite[Proposition~3.1]{Giudici2007} excludes subfield novelties. The degree column is $[G:H]$. In the two $M_{10}$ rows, the actions on the cyclic normal subgroups are $a\mapsto a^2$ on $C_5$ and $a\mapsto a^3$ on $C_8$, respectively, giving the stated groups. The cyclicity of $A/G$ gives the finite-field filter above. For $Q=9$ the field component of Frobenius is forced by the parity of $s$, leaving only the diagonal choice, which selects $\operatorname{PGL}_2(9)$ for even $s$ and the field-diagonal group $M_{10}$ for odd $s$.

Existence and uniqueness are supplied by exact invariant-field constructions. For $Q=7,9,11$ the two Cartan $k$-forms from Section~\ref{sec:wild-lie} specialize to the displayed degrees; exactly the form with diagonal component has arithmetic point stabilizer $\mathrm N_A(H)$ equal to Giudici's maximal novelty subgroup, while the other is not novelty. At $Q=9$ the field component, hence whether $A$ is $\operatorname{PGL}_2(9)$ or $M_{10}$, is forced by the parity of $s$. Their cover-over-target automorphism groups are trivial, so each gives one $k$-M\"obius class. For the infinite $A_4$ family, specialize the invariant $J_Q$ of \eqref{eq:JQ} to $Q=p$, and write it as $J_p$. Choose $H=A_4\leqslant\operatorname{PSL}_2(p)$, a generator $h$ of $\mathbb F_p(Z)^H$, and write
\(J_p(Z)=R_p(h(Z)).\)
Then $\deg R_p=p(p^2-1)/24$. An element of $\mathrm N_{\operatorname{PGL}_2(p)}(H)\setminus H$ induces an involution $\mu$ on $h$ and satisfies
\(R_p\circ\mu=-R_p.\)
The untwisted form has arithmetic point stabilizer $A_4$ and is not primitive; the unique nontrivial quadratic twist has arithmetic group $\operatorname{PGL}_2(p)$ and point stabilizer $\mathrm N_{\operatorname{PGL}_2(p)}(A_4)=S_4$, which is maximal in the novelty range. Since $\mathrm N_{\operatorname{PSL}_2(p)}(A_4)=A_4$, this gives exactly one novelty class.
\end{proof}

For the remainder of this subsection, fix the infinite $A_4$ case with $H=A_4\leqslant\operatorname{PSL}_2(p)$ and choose a generator $h$ of $\mathbb F_p(Z)^H$. Since $\mathbb F_p(J_p)\subseteq\mathbb F_p(h)$, let $R_p\in\mathbb F_p(X)$ be the unique rational function satisfying $J_p=R_p\circ h$.

\begin{proposition}[Factor structure of the infinite $A_4$ novelty quotient]\label{prop:Rp-factor}
Assume $p\equiv\pm11,\pm19\pmod{40}$, and use the notation
$J_p=R_p\circ h$ fixed above. Put $b:=p(p-1)/2$, $c:=(p+1)/2$, and
$d:=p(p^2-1)/24$, and define
\[
 \varepsilon_2=
 \begin{cases}
 1,&p\equiv1\pmod4,\\
 0,&p\equiv3\pmod4,
 \end{cases}
 \qquad
 \varepsilon_3=
 \begin{cases}
 1,&p\equiv1\pmod3,\\
 0,&p\equiv-1\pmod3.
 \end{cases}
\]
There is an $\mathbb F_p$-coordinate $[X:Y]$ on $\mathbf P^1/A_4$ in which
the unique order-$2$ tetrahedral branch point is $X=0$ and the order-$3$
branch divisor is $X^2+3Y^2=0$. In this coordinate there are squarefree
homogeneous forms $P_+(X,Y),P_-(X,Y)\in\mathbb F_p[X,Y]$, unique up to
nonzero scalars, with
\[
\begin{aligned}
 \deg P_+&=\frac{p+1-6\varepsilon_2-8\varepsilon_3}{12},\\
 \deg P_-&=\frac{p(p-1)-6(1-\varepsilon_2)-8(1-\varepsilon_3)}{12},
\end{aligned}
\]
such that, after a target M\"obius normalization, $R_p=[F_-:F_+]$, where
\[
\begin{aligned}
 F_+&=P_+^bX^{\varepsilon_2b/2}
 (X^2+3Y^2)^{\varepsilon_3b/3},\\
 F_-&=P_-^cX^{(1-\varepsilon_2)c/2}
 (X^2+3Y^2)^{(1-\varepsilon_3)c/3}.
\end{aligned}
\]
Both $F_+$ and $F_-$ have degree $d$. Moreover, the outer involution in
$\mathrm N_{\operatorname{PGL}_2(p)}(A_4)/\allowbreak A_4\cong C_2$ acts in this normalization
as $X\mapsto-X$, and the free-orbit factors satisfy
\[
 P_+(X,Y)=Y^{1-\varepsilon_2}A_{+,p}(X^2,Y^2),
 \qquad
 P_-(X,Y)=Y^{\varepsilon_2}A_{-,p}(X^2,Y^2)
\]
for homogeneous forms $A_{+,p},A_{-,p}$. Beyond the explicit exponents and
fixed branch factors above, the remaining $p$-dependent coefficient data are
concentrated in these two polynomial families in this normalization.
\end{proposition}

\begin{proof}
The ramification locus of the Galois quotient $C\to C/G$ has two $G=\operatorname{PSL}_2(p)$-orbits:
\[
 \mathcal R_G=\mathcal R_+\sqcup\mathcal R_-,\qquad
 \mathcal R_+=\mathbf P^1(\mathbb F_p),\qquad
 \mathcal R_-=\mathbf P^1(\mathbb F_{p^2})\setminus\mathbf P^1(\mathbb F_p).
\]
Their point stabilizers in $G$ have orders $b$ and $c$, respectively. Inside $H=A_4$, every point stabilizer is cyclic and therefore has order $1$, $2$, or $3$.

An involution of $A_4$ is split precisely when $p\equiv1\pmod4$. Since $A_4$ has three involutions, the corresponding fixed points form one $H$-orbit of size $6$, lying in $\mathcal R_+$ when $\varepsilon_2=1$ and in $\mathcal R_-$ otherwise. Similarly, the four subgroups $C_3\leqslant A_4$ contribute two $H$-orbits of size $4$; these lie in $\mathcal R_+$ precisely when $p\equiv1\pmod3$. Removing these special orbits gives the stated numbers of free $H$-orbits.

The unique order-$2$ point of $\mathbf P^1/A_4$ is $\mathbb F_p$-rational. The two order-$3$ points are individually rational precisely when $p\equiv1\pmod3$, equivalently when $\left(\frac{-3}{p}\right)=1$. Uniqueness of a M\"obius transformation with prescribed images of three distinct points then gives the stated $\mathbb F_p$-normalization.

If $P$ is a point upstairs and $y$ its image on $\mathbf P^1/A_4$, then ramification in the intermediate quotient gives
\(e_y=\frac{|G_P|}{|H_P|}.\)
Consequently a free, order-$2$, or order-$3$ $H$-orbit in $\mathcal R_+$ contributes multiplicity $b$, $b/2$, or $b/3$, and the corresponding multiplicities in $\mathcal R_-$ are $c$, $c/2$, or $c/3$. These multiplicities give exactly the two displayed fibers. Their degrees are $d$, so no further factors occur.

Finally, the nontrivial element of $\mathrm N_{\operatorname{PGL}_2(p)}(A_4)/A_4$ fixes the unique order-$2$ tetrahedral branch point and exchanges the two order-$3$ points. In the chosen coordinate this is $X\mapsto-X$. To identify its other fixed point, use $A_4\triangleleft S_4$. In the degree-two quotient $\mathbf P^1/A_4\to\mathbf P^1/S_4$, the order-$4$ inertia intersects $A_4$ in $C_2$, whereas the other ramifying inertia is generated by an involution in $S_4\setminus A_4$ and meets $A_4$ trivially. The second fixed point, $Y=0$, therefore represents a free $A_4$-orbit.

Such an outer involution has nonsquare determinant class in $\operatorname{PGL}_2(p)$. A trace-zero representative has rational eigenlines exactly when minus its determinant is a square, which here is equivalent to $p\equiv3\pmod4$. Thus the free orbit at $Y=0$ lies in $\mathcal R_+$ exactly when $\varepsilon_2=0$. All other free points occur in pairs exchanged by $X\mapsto-X$; their defining factors are polynomials in $X^2,Y^2$. This proves the displayed parity formulas for $P_+$ and $P_-$.
\end{proof}

\begin{corollary}[The unique polynomial degeneration]
The quotient $R_p$ in Proposition~\ref{prop:Rp-factor} is $k$-M\"obius equivalent to a polynomial if and only if $p=11$.
For $p=11$ one may take
\[
 X^3(X^2+3)^2
 \bigl(X^8-X^6+2X^4-X^2-3\bigr)^6
\]
in characteristic $11$.
\end{corollary}

\begin{proof}
A rational function is $k$-M\"obius equivalent to a polynomial exactly when one fiber consists of a single totally ramified point. Proposition~\ref{prop:Rp-factor} shows that every ramification index is one of $b,b/2,b/3,c,c/2,c/3$.
Since $d/b=(p+1)/12$ and $d/c=p(p-1)/12$, and $p\geqslant11$ in the novelty range, one has $d\geqslant b>c$. Thus none of the smaller indices can equal $d$, and $d=b$ holds exactly for $p=11$, when $d=b=55$.

For $p=11$ the proposition gives $\deg P_+=1$ and $\deg P_-=8$; sending the zero of $P_+$ to infinity produces a polynomial of the form
\(X^3(X^2+3)^2P_-(X)^6.\)
The displayed octic is the degree-$55$ representative in Guralnick--Zieve's Table~B \cite[Table~B]{GuralnickZieve2010}.
\end{proof}

For reference, several small novelty representatives admit especially compact formulas. If $\eta$ is a nonsquare in the relevant base field, the $Q=7$ representatives may be taken as
\begin{align*}
 F_{21,\eta}(X)&=
 X(X^2+3\eta)^4(X^2-\eta)^2(X^2+\eta)^4,\\
 F_{28,\eta}(X)&=
 \frac{(X^2-\eta)^4(X^2-2\eta)^4(X^2-3\eta)^2(X^2+2\eta)^4}{X^{21}},
\end{align*}
and the $Q=11$ representative as
\[
 F_{66,\eta}(X)=
 \frac{(X^2-3\eta)^6(X^2+2\eta)^6(X^2-4\eta)^6
 (X^2+\eta)^6(X^2+3\eta)^6(X^2+4\eta)^3}{X^{55}}.
\]
In characteristic $3$ the degree-$36$ and degree-$45$ boundary maps may be taken as the formulas obtained from Propositions~\ref{prop:nonsplit-cartan} and~\ref{prop:split-cartan}; one compact pair is
\[
 \frac{X^5(X^2+1)^5(X^4-X^2-1)^5}{(X^2-1)^{18}},
 \qquad
 \frac{X^5(X^4+X^2-1)^5(X^4+1)^5}{(X^2-1)^{18}},
\]
with the required diagonal twist chosen according to the base-field parity.

\subsection{Completion of the wild classification}

\begin{proof}[Proof of Theorem~\ref{thm:wild-master}]
By Theorems~\ref{thm:reconstruction} and~\ref{thm:decomposition}, every wild $k$-indecomposable rational function with Galois closure of genus zero arises from a finite subgroup of $\operatorname{PGL}_2(\bar k)$ of order divisible by $p$, with a core-free stabilizer and primitive arithmetic monodromy. Faber's classification \cite[Theorem~B]{Faber2023} gives the four ambient group types listed at the beginning of Section~\ref{sec:wild-master}. Splitting the Lie-type case according to geometric primitivity gives the five cases of the theorem.

The $p$-semi-elementary family is classified in Theorem~\ref{thm:affine-wild}. The characteristic-$2$ dihedral and characteristic-$3$ $A_5$ families are treated in Propositions~\ref{prop:char2-dihedral} and~\ref{prop:char3-A5}. Modulo the displayed accidental isomorphisms, these exhaust the non-Lie wild families.

For the Lie-type groups, consider first $Q=3$. In $A_4$, the case $H=1$ is excluded by the characteristic subgroup $V_4$, while $H=C_2$ is excluded by the Frobenius-stable overgroup $\mathrm N_{A_4}(H)=V_4$; the remaining core-free primitive case is $H=C_3$. In $S_4$, the case $H=1$ is again excluded by the characteristic $V_4$. Every automorphism of $S_4$ is inner; if $H\ne1$ and arithmetic primitivity holds, Theorem~\ref{thm:structural-consequences} gives $\mathrm N_G(H)=H$, so the induced inner automorphism is represented by an element of $H$. Every overgroup of $H$ is then Frobenius-stable, and $H$ must be maximal. The only core-free maximal possibility is $S_3$. Thus exactly the pairs $(A_4,C_3)$ and $(S_4,S_3)$ remain, as in Proposition~\ref{prop:borel-wild}. The Dickson--Giudici list then gives exactly the Borel, Cartan, subfield and cases with stabilizer $A_4,S_4$, or $A_5$, whose $k$-forms are Propositions~\ref{prop:borel-wild}--\ref{prop:exceptional-wild}.

It remains to treat nonmaximal $H_f$ for which the arithmetic monodromy action remains primitive. This is precisely the situation in Giudici's novelty classification: the arithmetic point stabilizer is maximal, whereas its intersection with the geometric group is not maximal in that group. Giudici's novelty theorem, together with the cyclicity of the finite-field constant quotient, gives exactly Theorem~\ref{thm:wild-novelty}. Corollary~1.2 of \cite{Giudici2007} excludes geometric $\operatorname{PGL}_2(Q)$ novelties. Thus the list is exhaustive.
\end{proof}

\begin{theorem}[Complete wild classification]\label{thm:intro-wild}
Let $f\in k(X)$ be separable and $k$-indecomposable, assume that its Galois closure has genus zero, and suppose $p\mid |G_f|$. Then, after the explicit small-group identifications recorded in Section~\ref{sec:wild-master}, $f$ belongs to exactly one of the following three structural types:
\begin{enumerate}[label=(\arabic*),ref=\arabic*,font=\itshape]
\item the affine classes with $G_f=V\rtimes C_n$, where $V\ne0$ is an elementary abelian $p$-group and $C_n$ acts by scalars; these are classified in Theorem~\ref{thm:affine-wild};
\item the small-characteristic cases: the characteristic-$2$ dihedral classes of Proposition~\ref{prop:char2-dihedral} and the characteristic-$3$ $A_5$ classes of Proposition~\ref{prop:char3-A5};
\item the defining-characteristic $\operatorname{PSL}_2(Q)$ and $\operatorname{PGL}_2(Q)$ cases of Sections~\ref{sec:wild-lie}--\ref{sec:wild-novelty}, including both the geometrically primitive and the geometrically decomposable $k$-indecomposable cases.
\end{enumerate}
Sections~\ref{sec:wild-master}--\ref{sec:wild-novelty} give the complete stabilizer ranges, finite-field classes, and exact class counts for these three structural types.
\end{theorem}

\begin{proof}[Proof of Theorem~\ref{thm:intro-wild}]
The statement follows from Theorem~\ref{thm:wild-master}, the affine classification in Theorem~\ref{thm:affine-wild}, the small-characteristic cases in Propositions~\ref{prop:char2-dihedral} and~\ref{prop:char3-A5}, the primitive Lie-type cases in Propositions~\ref{prop:borel-wild}--\ref{prop:exceptional-wild}, and the complete novelty list in Theorem~\ref{thm:wild-novelty}.
\end{proof}

The complete tame and wild classifications also imply an immediate global restriction on branch arithmetic.

\begin{corollary}[Closed-point degrees in the branch divisor]\label{cor:branch-degree-three}
For every separable $k$-indecomposable rational function in the classification, every closed point of the branch divisor has degree at most $3$ over $k$. Degree $3$ occurs only for the $V_4/A_4$ quartic. Degree $2$ occurs only for the nonsplit cyclic class, for the pair of order-$2$ branch points of the nonsquare tame dihedral class, and for the pair of order-$3$ branch points of the nonsplit tetrahedral quartic.
\end{corollary}

\begin{proof}
Apply Theorem~\ref{thm:branch-orbit}, so it is enough to bound the Frobenius orbit lengths on the geometric branch locus of $C\to C/G$. We check each classification case.

For a tame cyclic Galois quotient $C\to C/G$ there are two branch points,
so Frobenius either fixes them or transposes them; the latter is exactly
the nonsplit cyclic class. For a tame dihedral Galois quotient the
order-$\ell$ branch point is distinguished, while the two order-$2$
branch points are fixed or transposed; the transposition is exactly
the nonsquare class. For the tetrahedral Galois quotient the order-$2$
branch point is distinguished and the two order-$3$ branch points are
fixed or transposed. In the $V_4/A_4$ quartic the three order-$2$ branch
points form one Frobenius $3$-cycle. For the octahedral and icosahedral
tame Galois quotients the three inertia orders are pairwise distinct,
hence every branch point is Frobenius-fixed.

For the affine family in Theorem~\ref{thm:affine-wild}, the branch points are $0,\infty$ when $n>1$; when $n=1$, the representative is separable additive and its only branch point is $\infty$. Proposition~\ref{prop:char2-dihedral} gives two rational branch points in the characteristic-$2$ dihedral case. In characteristic $3$, direct differentiation of the representatives in Proposition~\ref{prop:char3-A5} shows that all three maps $W_5,W_6,W_{10}$ have branch points $\{0,\infty\}$; for $W_{10}$, the value $W_{10}(\infty)=-1$ is unramified.

Finally, for a defining-characteristic Lie case, write $G=\operatorname{PSL}_2(Q_0)$ or $\operatorname{PGL}_2(Q_0)$. Including the subfield, cases with $H\cong A_4,S_4$, or $A_5$, and Lie-type novelty cases, the ramification locus of $C\to C/G$ depends only on this geometric group and is
\[
 \mathbf P^1(\mathbb F_{Q_0^2})
 =\mathbf P^1(\mathbb F_{Q_0})\sqcup
 \bigl(\mathbf P^1(\mathbb F_{Q_0^2})\setminus\mathbf P^1(\mathbb F_{Q_0})\bigr).
\]
Indeed, every ramification point is fixed by a nonidentity element of $G\leqslant\operatorname{PGL}_2(\mathbb F_{Q_0})$, hence lies in the displayed quadratic set; conversely its rational part has $p$-divisible Borel stabilizer and its quadratic complement has prime-to-$p$ nonsplit-torus stabilizer, giving precisely the wild and tame $G$-orbits. Since $\varphi$ conjugates point stabilizers, it preserves divisibility of their orders by $p$ and therefore preserves these two orbits separately. Thus both branch points of $C\to C/G$ are Frobenius-fixed, and the only orbit lengths greater than one are precisely the three transpositions and the Klein-four $3$-cycle listed in the statement.
\end{proof}

\section{Permutation--exceptionality equivalence and extension degrees}\label{sec:no-accidental}

Throughout Sections~\ref{sec:no-accidental}--\ref{sec:family-support}, when the ground field is a finite field $K$ we write $Q=|K|$. The same notation is used after finite extensions of the ground field. The symbol $Q_0$ below is reserved for the defining-field parameter of a geometric monodromy group $\operatorname{PSL}_2(Q_0)$ or $\operatorname{PGL}_2(Q_0)$.

\Needspace{9\baselineskip}
\begin{definition}[Exceptionality]\label{def:exceptionality}
For later use we record the geometric and group-theoretic formulations of the rational-function notion used above.
\begin{enumerate}[label=(\arabic*),ref=\arabic*]
\item\label{item:exceptionality-1} Let $K$ be a finite field and let $f\in K(X)$. We call $f$ \emph{exceptional over $K$} if $f:\mathbf P^1(L)\to\mathbf P^1(L)$ is bijective for infinitely many finite extensions $L/K$.
\item Let $h:X\to Y$ be a finite separable morphism of smooth projective geometrically irreducible curves over a finite field $K$. We call $h$ an \emph{exceptional cover over $K$} if the diagonal is the only geometrically irreducible component of $X\times_YX$ that is defined over $K$.
\item Let $A$ and $G$ be transitive permutation groups on a finite set $\Xi$, with $G\triangleleft A$. We call $(A,G,\Xi)$ an \emph{exceptional permutation triple} if the diagonal in $\Xi\times\Xi$ is the only orbit that is simultaneously an $A$-orbit and a $G$-orbit.
\item Suppose in addition that $A/G$ is cyclic and that the image of $G\tau$ generates $A/G$. We call $G\tau$ an \emph{exceptional generating coset on $\Xi$} if $|\operatorname{Fix}_{\Xi}(a)|=1$ for every $a\in G\tau$. Otherwise $G\tau$ is called \emph{nonexceptional}.
\end{enumerate}
\end{definition}

We call a $K$-M\"obius class exceptional if it contains an exceptional representative. This does not depend on the representative, since pre- and post-composition by $K$-M\"obius transformations preserves bijectivity over every finite extension.

\begin{lemma}[Equivalent formulations of exceptionality]\label{lem:exceptionality-dictionary}
Let $K$ be a finite field, let $f:\mathbf P^1_K\to\mathbf P^1_K$ be a separable rational function, and let $A_f$ and $G_f$ be its arithmetic and geometric monodromy groups in their common action on the sheet set $\Xi$. Then $G_f\triangleleft A_f$, both groups are transitive on $\Xi$, and $A_f/G_f$ is cyclic. For every coset $G_f\tau$ whose image generates $A_f/G_f$, the four notions in Definition~\ref{def:exceptionality} are equivalent as follows:
\begin{enumerate}[label=\textup{(\roman*)},ref=\roman*]
\item\label{item:exceptionality-dictionary-1} $f$ is exceptional over $K$;
\item\label{item:exceptionality-dictionary-2} the cover $f:\mathbf P^1_K\to\mathbf P^1_K$ is exceptional;
\item\label{item:exceptionality-dictionary-3} the permutation triple $(A_f,G_f,\Xi)$ is exceptional;
\item\label{item:exceptionality-dictionary-4} $G_f\tau$ is an exceptional generating coset on $\Xi$.
\end{enumerate}
Consequently, if one generating coset is exceptional, then every generating coset is exceptional.
\end{lemma}

\begin{proof}
For \textup{(\ref{item:exceptionality-dictionary-1})} and \textup{(\ref{item:exceptionality-dictionary-2})}, Guralnick--Tucker--Zieve record that an exceptional cover remains exceptional over infinitely many finite extensions; their rational-point theorem then gives bijectivity over each such extension, while bijectivity over infinitely many finite extensions forces exceptionality \cite[Theorem~1 and Proposition~5.6]{GTZ2007}.

For \textup{(\ref{item:exceptionality-dictionary-2})} and \textup{(\ref{item:exceptionality-dictionary-3})}, the geometrically irreducible components of the fiber square are encoded by the geometric monodromy orbits on $\Xi\times\Xi$, while the components defined over $K$ are encoded by the arithmetic action. Hence the diagonal is the only $K$-defined geometrically irreducible component exactly when it is the only common $A_f$- and $G_f$-orbit \cite[Lemma~4.1]{GTZ2007}.

Finally, since $A_f/G_f$ is cyclic and $G_f$ is transitive on $\Xi$, the generating-coset criterion gives
\[
 (A_f,G_f,\Xi)\text{ exceptional}
 \Longleftrightarrow
 |\operatorname{Fix}_{\Xi}(a)|=1
 \quad\text{for every }a\in G_f\tau
\]
for every generating coset $G_f\tau$ \cite[Lemma~3.3]{GMS2003}. This proves \textup{(\ref{item:exceptionality-dictionary-3})}$\Longleftrightarrow$\textup{(\ref{item:exceptionality-dictionary-4})} and the final assertion.
\end{proof}

The preceding definition and lemma place the rational-function notion of exceptionality in its geometric and permutation-group context. The no-accidental-permutations theorem establishes the converse phenomenon for the class with Galois closure of genus zero: bijectivity over one finite field already forces these equivalent exceptional conditions.

Let $K=\mathbb F_Q$ be a finite field and let $f\in K(X)$ be separable
with Galois closure of genus zero. Choose a semilinear $\mathbf P^1$-datum
$\mathcal D=(\kappa/K,C,G,H,\tau)$ over $K$ with
$f\in\mathcal M_{\mathcal D}$. Put $A=\langle G,\tau\rangle$,
$U=\langle H,\tau\rangle$, $\Xi=G/H\cong A/U$, and let $\Gamma=G\tau$ be
the generating Frobenius coset. For $a\in\Gamma$ write
$\chi(a):=|\operatorname{Fix}_{\Xi}(a)|$ and $v(a):=\chi(a)-1$. For the
point-side geometry, let $\mathcal F_a:C(\bar K)\to C(\bar K)$ denote the
$Q$-semilinear map corresponding to $a$, with the convention
\begin{equation}\label{eq:sector-semilinear-convention}
 \mathcal F_{ia}=i\circ\mathcal F_a
 \qquad(i\in G).
\end{equation}
Thus in a projective coordinate one may write $\mathcal F_a=u_a\circ\Phi_Q$
for some $u_a\in\operatorname{PGL}_2(\bar K)$, but no normalization making
the chosen Frobenius lift purely coefficientwise is assumed.

The proof is uniform rather than family-by-family. The local fixed-point proposition converts a permutation hypothesis into fixed-point constraints on the generating coset, yielding the counting and field-degree bounds. For defining-characteristic $\operatorname{PSL}_2/\operatorname{PGL}_2$, the two ramification orbits and Shintani descent reduce every class with $v\ne0$ to identity or involution images, and the second moment forces $v=0$. The remaining primitive cases are structural; complete decomposition then globalizes the result, and constant-field arithmetic determines all extension degrees.

\subsection{Fixed-point moments and local averaging}

\begin{proposition}[First and second fixed-point moments on a generating coset]\label{prop:coset-package}
With the notation above:
\begin{enumerate}[label=(\arabic*),ref=\arabic*,font=\itshape]
\item
\[
 \frac1{|G|}\sum_{a\in G\tau}\chi(a)=1.
\]
\item\label{item:coset-package-2} The cover
$f:\mathbf P^1_K\to\mathbf P^1_K$ is exceptional if and only if
$\chi(a)=1$ for every $a\in G\tau$.
\item If $M_2:=|G|^{-1}\sum_{a\in G\tau}\chi(a)^2$, then
\begin{equation}\label{eq:twisted-variance}
 M_2-1
 =\frac1{|G|}\sum_{a\in G\tau}v(a)^2
 \in\mathbb Z_{\geqslant0}.
\end{equation}
Consequently, if the generating coset is nonexceptional, then it contains a
\emph{derangement}, that is, an element with no fixed point on $\Xi$,
and $M_2\geqslant2$.
\end{enumerate}
\end{proposition}

\begin{proof}
The first assertion is the coset form of Burnside's lemma \cite[Lemma~3.2]{GMS2003}; part~(\ref{item:coset-package-2}) is Lemma~\ref{lem:exceptionality-dictionary}. Apply the same coset-Burnside formula to the diagonal action on $\Xi^2$. Since
\(|\operatorname{Fix}_{\Xi^2}(a)|=\chi(a)^2,\)
the integer $M_2$ counts the $G$-orbits on $\Xi^2$ stabilized by the generator of $A/G$. The diagonal orbit shows that $M_2\geqslant1$. Expanding $\chi=1+v$ and using the first assertion gives \eqref{eq:twisted-variance}.

If the coset is nonexceptional, then $\chi$ is not identically $1$. Since its average is $1$ and all values are nonnegative integers, some value must be $0$; this is a derangement. Finally, \eqref{eq:twisted-variance} is positive and integral, so $M_2\geqslant2$.
\end{proof}

\begin{proposition}[Local fixed-point averaging and the field-degree bound]\label{prop:local-sector}
Assume that $f:\mathbf P^1(K)\to\mathbf P^1(K)$ is bijective. Let $a\in G\tau$, let $P\in\operatorname{Fix}(\mathcal F_a)$, put $I_P=G_P$, and let $\bar y\in Y_{\mathcal D}(K)$ be the descended point corresponding to the $G$-orbit of $P$. Then:
\begin{enumerate}[label=(\arabic*),ref=\arabic*,font=\itshape]
\item\label{item:local-sector-1}
\begin{equation}\label{eq:local-sector-mass}
 |\gamma_{\mathcal D}^{-1}(\bar y)(K)|
 =\frac1{|I_P|}\sum_{i\in I_P}\chi(ia).
\end{equation}
In particular, if $\chi(a)\ne1$, then every point of $\operatorname{Fix}(\mathcal F_a)$ is ramified in $C\to C/G$.
\item Every $\mathcal F_a$ has exactly $Q+1$ fixed points. If the generating coset is nonexceptional and $\mathcal S:=\{a\in G\tau:\chi(a)\ne1\}$ with $M:=|\mathcal S|$, then
\begin{equation}\label{eq:support-lower}
 M\geqslant2\left\lceil\frac{Q+1}{2}\right\rceil.
\end{equation}
If the ramification locus of $C\to C/G$ has $s$ geometric $G$-orbits, then
\begin{equation}\label{eq:ramification-capacity}
 M(Q+1)\leqslant s|G|.
\end{equation}
\item\label{item:local-sector-3} Suppose $Q=p^e$ and, in the chosen projective realization, $G\leqslant\operatorname{PGL}_2(\mathbb F_{p^\alpha})$. If an element $a\in G\tau$ with $\chi(a)\ne1$ exists, then
\begin{equation}\label{eq:quadratic-confinement}
 e\mid2\alpha.
\end{equation}
\end{enumerate}
\end{proposition}

\begin{proof}
Since $\mathcal F_a(P)=P$, the element $a$ normalizes $I_P$. Moreover, if one element $a\in G\tau$ fixes $P$, then the set of labels in $G\tau$ fixing $P$ is exactly $I_Pa$: if $b$ also fixes $P$, then $ba^{-1}\in I_P$, and the converse follows from \eqref{eq:sector-semilinear-convention}. Since $a\in G\tau$, the equality $\mathcal F_a(P)=P$ makes the $G$-orbit of $P$ fixed by the descended Frobenius. Under $(Y_{\mathcal D})_\kappa\cong C/G$, it therefore defines $\bar y\in Y_{\mathcal D}(K)$. After passing to $\bar K$, the fiber of $\gamma_{\mathcal D}$ above $\bar y$ is the set of $I_P$-orbits on $\Xi=G/H$, and $a$ induces Frobenius on that fiber. On $V=\mathbb C[\Xi]$, the orbit sums form a basis of $V^{I_P}$, so
\[
 |\gamma_{\mathcal D}^{-1}(\bar y)(K)|
 =\operatorname{Tr}(a\mid V^{I_P})
 =\operatorname{Tr}\!\left(\frac1{|I_P|}\sum_{i\in I_P}ia\,\middle|\,V\right),
\]
which is \eqref{eq:local-sector-mass}. The $K$-isomorphisms identifying $\gamma_{\mathcal D}$ with a representative of $\mathcal M_{\mathcal D}$ preserve rational fibers. Since $f$ is a permutation, the left-hand side equals $1$. Thus
\begin{equation}\label{eq:local-v-mass}
 \sum_{i\in I_P}v(ia)=0.
\end{equation}
At an unramified point $I_P=1$, this gives $v(a)=0$, proving the final assertion of part~(\ref{item:local-sector-1}).

Write $\mathcal F_a=u_a\Phi_Q$. Lang--Steinberg \cite[Theorem~21.7]{MalleTesterman2011} makes it projectively conjugate to standard $Q$-Frobenius, so $|\operatorname{Fix}(\mathcal F_a)|=Q+1$. If the coset is nonexceptional, choose a derangement $a$ by Proposition~\ref{prop:coset-package}. For each fixed point $P$, \eqref{eq:local-v-mass} gives an $i_P\in I_P\setminus\{1\}$ with $v(i_Pa)>0$. If $i_Pa=i_{P'}a$, the same nonidentity M\"obius transformation fixes $P,P'$, so each positive-$v$ element accounts for at most two fixed points. Hence at least $\lceil(Q+1)/2\rceil$ such elements occur. Since $\sum_{a\in G\tau}v(a)=0$ and the only negative value is $-1$, there are at least as many derangements, proving \eqref{eq:support-lower}.

Every $a\in\mathcal S$ has $Q+1$ fixed points, all ramified by part~(\ref{item:local-sector-1}). For a fixed ramification point $P$, the elements of $G\tau$ fixing $P$ form either the empty set or the coset $I_Pa$, hence there are at most $|I_P|$ of them. Summing incidences gives
\(M(Q+1)\leqslant\sum_{P\in\mathcal R}|I_P|,\)
where $\mathcal R$ is the ramification locus. Each geometric ramification orbit contributes exactly $|G|$ by orbit--stabilizer, proving \eqref{eq:ramification-capacity}.

For a finite subfield $E\subseteq\bar K$, a \emph{projective $E$-subline}
of $\mathbf P^1(\bar K)$ is the image of $\mathbf P^1(E)$ under an element
of $\operatorname{PGL}_2(\bar K)$.
For part~(\ref{item:local-sector-3}), every ramification point is fixed by a nonidentity element of $\operatorname{PGL}_2(\mathbb F_{p^\alpha})$, hence lies in $\mathbf P^1(\mathbb F_{p^{2\alpha}})$. Lang--Steinberg identifies $\operatorname{Fix}(\mathcal F_a)$ with a projective $\mathbb F_{p^e}$-subline; if $\chi(a)\ne1$, part~(\ref{item:local-sector-1}) puts this subline inside $\mathbf P^1(\mathbb F_{p^{2\alpha}})$. A projective $E$-subline contained in $\mathbf P^1(L)$ forces $E\subseteq L$: send three of its points to $0,1,\infty$ and use uniqueness of the corresponding projectivity. Thus $\mathbb F_{p^e}\subseteq\mathbb F_{p^{2\alpha}}$, proving \eqref{eq:quadratic-confinement}.
\end{proof}

\subsection{The indecomposable permutation--exceptionality theorem}

The following lemma records the Shintani facts used for
$\operatorname{PSL}_2$ and $\operatorname{PGL}_2$.
Shintani descent for the indicated commuting Frobenius endomorphisms
gives a bijection between the relevant sets of conjugacy classes,
called the \emph{Shintani map}.

\begin{lemma}[Shintani descent for $\operatorname{PSL}_2$ and $\operatorname{PGL}_2$]\label{lem:rank-one-shintani-carriers}
Let $r$ be a power of $p$, let $F$ denote the $r$-Frobenius on
$\operatorname{PGL}_2(\bar k)$, and put $Q_0=r^n$.
\begin{enumerate}[label=(\arabic*),ref=\arabic*,font=\itshape]
\item\label{item:rank-one-shintani-carriers-1}
For $a=xF\in\operatorname{PGL}_2(Q_0)F$, let
$y\in\operatorname{PGL}_2(r)$ represent the image of the
$\operatorname{PGL}_2(Q_0)$-conjugacy class of $a$ under the Shintani map
for $F^n$ and $F$. Then $S_a=\operatorname{Fix}(a)$ is a projective
$\mathbb F_r$-subline. It may be identified with $\mathbf P^1(\mathbb F_r)$
so that the action induced by $F^n$ is $y$. Consequently
\[
 |S_a\cap\mathbf P^1(\mathbb F_{Q_0})|
 =|\operatorname{Fix}_{\mathbf P^1(\mathbb F_r)}(y)|.
\]
If $\mathcal C_a$ is the $\operatorname{PGL}_2(Q_0)$-conjugacy class of
$a$ in the source coset, then
\[
 \frac{|\mathcal C_a|}{|\operatorname{PGL}_2(Q_0)|}
 =\frac1{|C_{\operatorname{PGL}_2(r)}(y)|}.
\]
\item\label{item:rank-one-shintani-carriers-2}
Suppose $r$ is odd. The Shintani map carries the
$\operatorname{PGL}_2(Q_0)$-conjugacy classes contained in
$\operatorname{PSL}_2(Q_0)F$ to the
$\operatorname{PGL}_2(r)$-conjugacy classes contained in
$\operatorname{PSL}_2(r)$, and the classes in the complementary cosets
correspond as well.
A $\operatorname{PGL}_2(Q_0)$-conjugacy class $\mathcal C_a$ meeting a
specified $\operatorname{PSL}_2(Q_0)$-coset in
$\operatorname{PGL}_2(Q_0)F$ intersects it in either one or two
$\operatorname{PSL}_2(Q_0)$-conjugacy classes.
In the latter case, writing these classes as $\mathcal C_1,\mathcal C_2$, one has
\[
 \frac{|\mathcal C_1|}{|\operatorname{PSL}_2(Q_0)|}
 =\frac{|\mathcal C_2|}{|\operatorname{PSL}_2(Q_0)|}
 =\frac{|\mathcal C_a|}{|\operatorname{PGL}_2(Q_0)|}.
\]
\item\label{item:rank-one-shintani-carriers-3}
Suppose $n$ is odd and use Shintani descent for $F^n$ and $F^2$.
The resulting Shintani map takes
$\operatorname{PGL}_2(Q_0)$-conjugacy classes in
$\operatorname{PGL}_2(Q_0)F^2$ to
$\operatorname{PGL}_2(r^2)$-conjugacy classes in
$\operatorname{PGL}_2(r^2)F$.
Every involution $yF$ in this latter coset is
$\operatorname{PGL}_2(r^2)$-conjugate to the standard field involution
$F$, whose centralizer is $\operatorname{PGL}_2(r)$.
Hence the corresponding source class $\mathcal C_a$ satisfies
\[
 \frac{|\mathcal C_a|}{|\operatorname{PGL}_2(Q_0)|}
 =\frac1{|\operatorname{PGL}_2(r)|}=\frac1{r(r^2-1)}.
\]
For odd $r$, if this class meets a specified
$\operatorname{PSL}_2(Q_0)$-coset in $\operatorname{PGL}_2(Q_0)F^2$,
its intersection again consists of one class or two equal-size classes
under $\operatorname{PSL}_2(Q_0)$-conjugacy.
\end{enumerate}
\end{lemma}

\begin{proof}
Harper's Shintani theorem gives the class bijections and preserves finite
centralizers \cite[Theorem~2.1]{Harper2021}.
For part~(\ref{item:rank-one-shintani-carriers-1}), choose the usual
Lang--Steinberg conjugating element for $a=xF$
\cite[Theorem~21.7]{MalleTesterman2011}. It carries $S_a$ to
$\mathbf P^1(\mathbb F_r)$, and the same Shintani construction identifies
the induced $F^n$-action with the representative $y$ of the corresponding
class. The intersection formula follows. Centralizer preservation gives
the displayed ratio of class size to group order.

For part~(\ref{item:rank-one-shintani-carriers-2}), use the natural
homomorphism $\operatorname{SL}_2\to\operatorname{PGL}_2$.
Its kernel is central, and its image on the finite fixed-point groups is
$\operatorname{PSL}_2$. Harper's compatibility with this homomorphism
\cite[Corollary~2.14]{Harper2021} gives the restriction to classes in
$\operatorname{PSL}_2(Q_0)F$ and $\operatorname{PSL}_2(r)$; bijectivity
gives the assertion for the complementary cosets.
Since $\operatorname{PSL}_2(Q_0)$ has index two in
$\operatorname{PGL}_2(Q_0)$, a conjugacy class under the latter group
splits into at most two classes under the former. If it splits into
$\mathcal C_1,\mathcal C_2$, conjugation by an element of
$\operatorname{PGL}_2(Q_0)\setminus\operatorname{PSL}_2(Q_0)$ exchanges
them. Thus $|\mathcal C_i|=|\mathcal C_a|/2$ and
$|\operatorname{PSL}_2(Q_0)|=|\operatorname{PGL}_2(Q_0)|/2$, proving
the asserted ratios.

For part~(\ref{item:rank-one-shintani-carriers-3}), Lang--Steinberg
conjugates every semilinear involution in $\operatorname{PGL}_2(r^2)F$
to $F$. Its centralizer consists of the $F$-fixed projectivities,
namely $\operatorname{PGL}_2(r)$. Centralizer preservation gives the
class-size ratio, and the final assertion follows by the same
index-two argument as in part~(\ref{item:rank-one-shintani-carriers-2}).
\end{proof}

\begin{proposition}[Permutation implies exceptionality for defining-characteristic $\operatorname{PSL}_2/\operatorname{PGL}_2$ monodromy]\label{prop:rank-one-nap}
Let $K$ be a finite field of characteristic $p$, and assume that the arithmetic action $A\curvearrowright A/U$ is primitive and that its geometric monodromy group is the standard defining-characteristic group
\[
 G=\operatorname{PSL}_2(Q_0)
 \quad\text{or}\quad
 G=\operatorname{PGL}_2(Q_0),
 \qquad
 Q_0=p^\alpha.
\]
If the corresponding rational function permutes $\mathbf P^1(K)$, then its generating Frobenius coset is exceptional.
\end{proposition}

\begin{proof}
The argument is independent of the point stabilizer $H$: it uses only the two $G$-orbits in the ramification locus of $C\to C/G$, field-degree restriction, and Shintani descent.

Write $K=\mathbb F_{p^e}$. In the standard projective realization, the ramification locus of $C\to C/G$ is
\begin{equation}\label{eq:rank-one-ramification}
 \mathcal R=\mathbf P^1(\mathbb F_{Q_0^2})
 =\mathcal R_{\rm w}\sqcup\mathcal R_{\rm t},
\end{equation}
where
\[
 \mathcal R_{\rm w}=\mathbf P^1(\mathbb F_{Q_0}),
 \qquad
 \mathcal R_{\rm t}=\mathbf P^1(\mathbb F_{Q_0^2})\setminus\mathbf P^1(\mathbb F_{Q_0})
\]
are the two $G$-orbits. Indeed, every ramified point is fixed by a nonidentity M\"obius transformation over $\mathbb F_{Q_0}$ and hence has degree at most two; conversely rational points have nontrivial Borel stabilizer and quadratic points have nontrivial nonsplit-torus stabilizer. Orbit--stabilizer gives the asserted two orbits for both $\operatorname{PSL}_2$ and $\operatorname{PGL}_2$.

\medskip
\noindent\emph{Step 1: confinement and reduction to $t=1,2$.}
Assume for contradiction that the generating coset is nonexceptional. Proposition~\ref{prop:local-sector} supplies an element $a\in G\tau$ with $\chi(a)\ne1$ and gives $e\mid2\alpha$. Put $c:=(\alpha,e)$, $r:=p^c$, $n:=\alpha/c$, and $t:=e/c$. Then $(n,t)=1$ and $t\mid2$, so $t\in\{1,2\}$.
After conjugating to the standard semilinear normalizer, write $F$ for the $r$-Frobenius; the associated semilinear maps are represented by $xF^t$ with $x\in\operatorname{PGL}_2(Q_0)$.

\medskip
\noindent\emph{Step 2: the $t=1$ calculation for $\operatorname{PGL}_2(Q_0)$-conjugacy classes.}
Suppose first that $t=1$, so $K=\mathbb F_r$ and $Q_0=r^n$. By Lemma~\ref{lem:rank-one-shintani-carriers}~(\ref{item:rank-one-shintani-carriers-1}), the Shintani map for $F^n$ and $F$ preserves the centralizers of the
corresponding $\operatorname{PGL}_2(Q_0)$- and
$\operatorname{PGL}_2(r)$-conjugacy classes. If $a=xF$, $S_a=\operatorname{Fix}(a)$, and $y\in\operatorname{PGL}_2(r)$ represents the corresponding conjugacy
class under the Shintani map, then $|S_a\cap\mathcal R_{\rm w}|=|\operatorname{Fix}_{\mathbf P^1(\mathbb F_r)}(y)|=:w(y)$.
If $v(a)\ne0$, then $S_a\subseteq\mathcal R$ by Proposition~\ref{prop:local-sector}. Since every point of \eqref{eq:rank-one-ramification} is fixed by $F^{2n}$, the induced element $y^2$ fixes all $r+1$ points of $\mathbf P^1(\mathbb F_r)$; hence $y^2=1$. Thus $v$ can be nonzero only on source classes corresponding under the
Shintani map to the identity or an involution class.

Let $\mathcal C$ run over the $G$-classes in the generating coset $G\tau$, put $d_{\mathcal C}:=|\mathcal C|/|G|$ and $v_{\mathcal C}:=\chi(\mathcal C)-1$, and let $y_{\mathcal C}\in\operatorname{PGL}_2(r)$ represent the image,
under the Shintani map, of the $\operatorname{PGL}_2(Q_0)$-conjugacy class
containing $\mathcal C$. For any $G$-stable ramification orbit $\mathcal O$, summing \eqref{eq:local-v-mass} over $P\in\mathcal O$ and exchanging the finite sums gives $\sum_{a\in G\tau}v(a)|\operatorname{Fix}(\mathcal F_a)\cap\mathcal O|=0$. For $v(a)\ne0$, the preceding proposition puts the whole fixed set in the two $G$-orbits $\mathcal R_{\rm w}$ and $\mathcal R_{\rm t}$, while classes with $v=0$ contribute nothing. Applying this to the wild and tame orbits gives
\begin{align}
 \sum_{\mathcal C}d_{\mathcal C}v_{\mathcal C}w(y_{\mathcal C})&=0,\label{eq:rank-one-mass-w}\\
 \sum_{\mathcal C}d_{\mathcal C}v_{\mathcal C}
 \bigl(r+1-w(y_{\mathcal C})\bigr)&=0.\label{eq:rank-one-mass-t}
\end{align}

If $r$ is even, then $\operatorname{PSL}_2(r)=\operatorname{PGL}_2(r)$ and there is a single involution class; its elements have one fixed point on $\mathbf P^1(\mathbb F_r)$. Equation \eqref{eq:rank-one-mass-t} forces its excess to vanish, and \eqref{eq:rank-one-mass-w} then forces the identity excess to vanish.

Assume now that $r$ is odd and $G=\operatorname{PGL}_2(r^n)$. For the source classes corresponding under the Shintani map to the
identity, a split involution, and a nonsplit involution, the ratios
$d_{\mathcal C}=|\mathcal C|/|G|$ are, respectively,
\[
 d_0=\frac1{r(r^2-1)},
 \qquad
 d_s=\frac1{2(r-1)},
 \qquad
 d_n=\frac1{2(r+1)},
\]
and
\(w(1)=r+1, \qquad w(s)=2, \qquad w(n)=0.\)
Equations \eqref{eq:rank-one-mass-w}--\eqref{eq:rank-one-mass-t} give
\(v_0+rv_s=0, \qquad v_s+v_n=0.\)
Since every $v$ is an integer at least $-1$, the only nonzero possibility is
\[
 (v_0,v_s,v_n)=(r,-1,1).
\]
Then Proposition~\ref{prop:coset-package} gives
\[
 M_2
 =1+\frac{2r}{r^2-1},
\]
which lies strictly between $1$ and $2$, a contradiction.

\medskip
\noindent\emph{Step 3: restriction to $\operatorname{PSL}_2$.}
It remains, in the case $t=1$, to consider $r$ odd and
$G=\operatorname{PSL}_2(r^n)$. Work first with the Shintani map on
$\operatorname{PGL}_2(Q_0)$-conjugacy classes.
Lemma~\ref{lem:rank-one-shintani-carriers}~(\ref{item:rank-one-shintani-carriers-2})
identifies the classes in $\operatorname{PSL}_2(Q_0)F$ with those in
$\operatorname{PSL}_2(r)$ and also identifies the classes in the
complementary cosets. If the generating coset is
$\operatorname{PSL}_2(Q_0)F$, the possible source classes correspond to
the identity and the unique involution class contained in
$\operatorname{PSL}_2(r)$. For the complementary generating coset, only
the source class corresponding to the involution class in
$\operatorname{PGL}_2(r)\setminus\operatorname{PSL}_2(r)$ can occur.

By the same lemma, a $\operatorname{PGL}_2(Q_0)$-conjugacy class $D$
meeting $G\tau$ intersects it in one $G$-conjugacy class or two of equal
size. In the latter case, each satisfies
$d_{\mathcal C}=|D|/|\operatorname{PGL}_2(Q_0)|$.
Put $A_D:=\sum_{\mathcal C\subseteq D}d_{\mathcal C}v_{\mathcal C}$.
Equations \eqref{eq:rank-one-mass-w}--\eqref{eq:rank-one-mass-t} give
$A_D=0$ for every possible source class. A single class therefore has
$v=0$; a split pair has $v_1+v_2=0$, so any nonzero pattern is
$\{-1,1\}$. By centralizer preservation \cite[Theorem~2.1]{Harper2021},
\[
 d_{\rm id}=\frac1{r(r^2-1)},
 \qquad
 d_{\rm inv}\in
 \left\{\frac1{2(r-1)},\frac1{2(r+1)}\right\}.
\]
Hence a nonzero split pattern would give
\[
 0<M_2-1
 \leqslant\frac{2}{r(r^2-1)}+\frac1{r-1}<1,
\]
contrary to \eqref{eq:twisted-variance}.

\medskip
\noindent\emph{Step 4: the quadratic case $t=2$.}
Here $K=\mathbb F_{r^2}$, $Q_0=r^n$, and $n$ is odd.
The Shintani map for $F^n,F^2$ takes values in the conjugacy classes of
$\operatorname{PGL}_2(r^2)F$ under $\operatorname{PGL}_2(r^2)$.
A source class with $v\ne0$ corresponds to an involution class,
hence by Lemma~\ref{lem:rank-one-shintani-carriers}~(\ref{item:rank-one-shintani-carriers-3})
to the single class represented by the standard field involution $F$.
For $G=\operatorname{PGL}_2(Q_0)$ (or even characteristic) the first moment
forces $v=0$. For odd $r$ and $G=\operatorname{PSL}_2(Q_0)$, if the
corresponding $\operatorname{PGL}_2(Q_0)$-conjugacy class does not meet
$G\tau$, no class with $v\ne0$ is possible. Otherwise its intersection
is a single $G$-conjugacy class or splits into two equal-size classes;
in the latter case each has ratio $d=|\mathcal C|/|G|$, where
\[
 d=\frac1{|C_{\operatorname{PGL}_2(r^2)}(F)|}
 =\frac1{|\operatorname{PGL}_2(r)|}
 =\frac1{r(r^2-1)}.
\]
A nonzero aggregate-zero pair is $\{-1,1\}$, giving
\[
 0<M_2-1=\frac{2}{r(r^2-1)}<1,
\]
contrary to \eqref{eq:twisted-variance}.
\end{proof}

We use the following elementary observation twice. If $b\in\mathbf P^1(K)$ and the geometric fiber $f^{-1}(b)$ consists of exactly two distinct points, then $f$ cannot permute $\mathbf P^1(K)$: Frobenius acts on that two-point set either trivially or by transposition, so the fiber contains either two or no $K$-rational points, never exactly one.

\begin{theorem}[Permutation implies exceptionality in the indecomposable case]\label{thm:primitive-nap}
Let $K$ be a finite field, and let $f\in K(X)$ be separable, $K$-indecomposable, of degree greater than $1$, with Galois closure of genus zero. If $f$ permutes $\mathbf P^1(K)$, then $f$ is exceptional over $K$.
\end{theorem}

\begin{proof}
Write $K=\mathbb F_Q$. By Theorems~\ref{thm:intro-tame} and~\ref{thm:wild-master}, together with Theorem~\ref{thm:arith-only}, it is enough to treat four cases.

\medskip
\noindent\emph{Tame case.}
First suppose the geometric monodromy is tame. Tame Riemann--Hurwitz for the Galois quotient $C\to C/G$ gives either two branch points or one of the signatures $(2,2,n)$, $(2,3,3)$, $(2,3,4)$, and $(2,3,5)$.
With two branch points the ramification locus of $C\to C/G$ has only two points, whereas any element $a\in G\tau$ with $\chi(a)\ne1$ in a permutation cover would have $Q+1\geqslant3$ ramified fixed points by Proposition~\ref{prop:local-sector}; hence the generating coset is exceptional.

For the signature $(2,2,n)$, the case $n=2$ is the $V_4/A_4$ quartic. Indecomposability over $K$ forces Frobenius to cycle the three order-two subgroups of $V_4$. In the induced degree-four arithmetic action the generating coset consists of $3$-cycles, each with one fixed sheet, so it is exceptional. For $n>2$, the tame classification and primitivity of $A\curvearrowright A/U$ give $G=D_{2\ell}$ and $H=C_2$, with $\ell$ an odd prime. Identify $G/H$ with $\mathbb F_\ell$. If Frobenius acts on the rotation subgroup by $r\mapsto r^u$, the elements of the generating coset act on sheets as $x\mapsto ux+c$ and $x\mapsto-ux+c$.
Coefficient $Q$-Frobenius sends the eigenvalue ratio of a rotation to its $Q$th power, while the normalizer of the pair of eigenlines in $\operatorname{PGL}_2(\overline K)$ contributes only identity or inversion; hence
\begin{equation}\label{eq:dihedral-multiplier}
 u\equiv\pm Q\pmod\ell.
\end{equation}
If $u\ne\pm1$, every element has exactly one fixed sheet. Thus nonexceptionality requires $u=\pm1$. Then exactly one of the two slopes is $1$, and the set $\{a\in G\tau:\chi(a)\ne1\}$ has size $M=\ell$. Since $Q$ is odd, \eqref{eq:support-lower} gives $\ell\geqslant Q+1$. Together with \eqref{eq:dihedral-multiplier}, this forces $Q=\ell-1$, so $\ell=Q+1$ is even, a contradiction.

For $(2,3,3)$, the tetrahedral degree-four ramification partitions are
$\{\!\{2,2\}\!\}$, $\{\!\{3,1\}\!\}$, and $\{\!\{3,1\}\!\}$.
The order-$2$ branch point is $K$-rational and its fiber has exactly two
geometric points. For $(2,3,4)$, the octahedral ramification partitions are
$\{\!\{2,1,1\}\!\}$, $\{\!\{3,1\}\!\}$, and $\{\!\{4\}\!\}$;
the order-$3$ branch point is rational and its fiber again has exactly two
geometric points. Both are excluded by the preceding observation. Finally, for $(2,3,5)$ the order-$5$ branch points in degrees six and ten have ramification
partitions $\{\!\{5,1\}\!\}$ and $\{\!\{5,5\}\!\}$, respectively, and are excluded in the same way. In degree five, assume for contradiction that the generating coset is nonexceptional. Here $|G|=60$ and the ramification locus of $C\to C/G$ has three $G$-orbits, while $Q$ is odd. Proposition~\ref{prop:local-sector} gives
\(M\geqslant Q+1, \qquad M(Q+1)\leqslant180,\)
so $Q\leqslant12$. Tameness leaves only $Q=7$ or $11$. The rational order-five branch point has one totally ramified point, so after $K$-M\"obius normalization the normalized rational function is a degree-five polynomial. The order-$3$ branch point has ramification partition $\{\!\{3,1,1\}\!\}$ with a unique rational ramification point, and a further affine normalization gives
\(F(X)=aX^5+bX^4+cX^3, \qquad ac\ne0.\)
If $F$ were a permutation polynomial, then
\[
 \sum_{x\in K}F(x)^2=\sum_{y\in K}y^2=0.
\]
But
\[
 F(X)^2=a^2X^{10}+2abX^9+(b^2+2ac)X^8+2bcX^7+c^2X^6.
\]
Over $\mathbb F_7$ only the exponent $6$ is divisible by $6$, giving the sum $-c^2\ne0$; over $\mathbb F_{11}$ only the exponent $10$ is divisible by $10$, giving $-a^2\ne0$. This closes the tame case.

Next consider the $p$-semi-elementary and characteristic-$2$ dihedral cases. Choose a derangement $a\in G\tau$ if the generating coset is nonexceptional. Proposition~\ref{prop:local-sector} then makes all $Q+1$ points of $\operatorname{Fix}(\mathcal F_a)$ ramified.

In the affine case with $n=1$, the ramification locus is only $\{\infty\}$, an immediate contradiction. For $n>1$, the finite ramification orbit is $\mathcal O=V$. The Frobenius lift preserves the unique common fixed point of $V$ and the other fixed point $0$ of $H$; hence $gH\mapsto g(0)$ identifies $G/H$ with $\mathcal O$ equivariantly for the arithmetic action. A derangement has no fixed point in $\mathcal O$, but the ramification locus outside it contains only $\infty$, again a contradiction.

In the characteristic-$2$ dihedral case, the reflection subgroup $H=C_2$ has a unique fixed point $P$. Every lift normalizing $H$ fixes $P$, so $gH\mapsto g(P)$ is an arithmetic-equivariant identification of the sheet set with the wild-reflection orbit. A derangement has no fixed point there. Only the two rotation fixed points remain in the ramification locus, fewer than $Q+1$. This proves exceptionality in both families.

For the defining-characteristic Lie-type cases, Proposition~\ref{prop:rank-one-nap} applies. It simultaneously covers all Borel, Cartan, subfield, cases with $H\cong A_4,S_4$, or $A_5$, and Lie-type novelty cases in the wild classification.

The only remaining non-Lie wild case is characteristic three with geometric group $A_5$. By Proposition~\ref{prop:char3-A5}, the three $K$-M\"obius classes are represented by $W_5,W_6,W_{10}$. The fiber $W_5^{-1}(0)$ consists of the two points $0,-1$; the pole fiber of $W_6$ consists of $0,\infty$; and the pole fiber of $W_{10}$ consists exactly of the two roots of $X^2+X-1$, each with multiplicity five. Each is a two-point geometric fiber above a $K$-rational value, so none can be a permutation. The wild classification now completes the proof.
\end{proof}

\subsection{Globalization and exact extension degrees}

\begin{theorem}[Permutation--exceptionality equivalence]\label{thm:no-accidental}
Let $K$ be a finite field, and let $f\in K(X)$ be separable of degree greater than $1$. If the Galois closure of $f$ has genus zero, then
\[
 f:\mathbf P^1(K)\longrightarrow\mathbf P^1(K)\text{ is bijective}
 \quad\Longleftrightarrow\quad
 f\text{ is exceptional over }K.
\]
\end{theorem}

\begin{proof}
We prove the nontrivial implication by induction on $\deg f$. If $f$ is $K$-indecomposable, this is Theorem~\ref{thm:primitive-nap}. Otherwise write
\(f=g\circ h\)
with both factors of degree greater than one over $K$. Since $f$ is injective on the finite set $\mathbf P^1(K)$, the function $h$ is injective and hence bijective; then $g=f\circ h^{-1}$ is bijective as well. By Theorem~\ref{thm:complete-decomposition}, both factors are separable and have Galois closure of genus zero. The induction hypothesis shows that both factors are exceptional, and the exceptional-tower criterion \cite[Lemma~3.5]{GMS2003} then implies that their composite is exceptional.

Conversely, suppose that $f$ is exceptional over $K$. By Definition~\ref{def:exceptionality}~(\ref{item:exceptionality-1}), $f$ permutes $\mathbf P^1(L)$ for infinitely many finite extensions $L/K$. If two distinct points of $\mathbf P^1(K)$ had the same image under $f$, that collision would persist over every finite extension of $K$, contradicting bijectivity over any one of those extensions. Thus $f$ is injective on the finite set $\mathbf P^1(K)$ and hence bijective.
\end{proof}

Return now to $k=\mathbb F_q$. Choose a semilinear
$\mathbf P^1$-datum $\mathcal D=(\kappa/k,C,G,H,\varphi)$ over $k$
with $f\in\mathcal M_{\mathcal D}$ and $\kappa=\mathbb F_{q^d}$.
For $m\geqslant1$, put $k_m:=\mathbb F_{q^m}$, $c:=(m,d)$, and
$\kappa_m:=\kappa k_m$.

\begin{lemma}[Base change of semilinear data]
\label{lem:base-change-coset}
One has
\begin{equation}\label{eq:basechange-constants}
 \kappa\cap k_m=\mathbb F_{q^c},
 \qquad
 \kappa_m
 =\mathbb F_{q^{\operatorname{lcm}(d,m)}},
\end{equation}
and therefore $[\kappa_m:k_m]=d/c$. Let
$C_m:=C\times_{\kappa}\kappa_m$. The automorphism $\varphi^m$ extends
uniquely over the compositum to a $q^m$-semilinear automorphism
$\varphi_m$ of $C_m$ fixing $k_m$. Then
$\mathcal D^{(m)}:=(\kappa_m/k_m,C_m,G,H,\varphi_m)$ is a semilinear
$\mathbf P^1$-datum over $k_m$, and
$f\in\mathcal M_{\mathcal D^{(m)}}$ when $f$ is viewed over $k_m$.

Under restriction from the base-changed Galois closure to $\kappa(C)$,
the arithmetic group and point stabilizer are
\begin{equation}\label{eq:basechange-groups}
 A_m
 =
 \langle G,\varphi^m\rangle
 =
 \langle G,\varphi^c\rangle,
 \qquad
 U_m
 =
 \langle H,\varphi^m\rangle
 =
 \langle H,\varphi^c\rangle.
\end{equation}
The designated generating Frobenius coset is $G\varphi^m$.
\end{lemma}

\begin{proof}
Put $\Omega_m:=\kappa(C)k_m$. Its full constant field is
$\kappa_m=\kappa k_m$, and the standard intersection and compositum
formulas for finite fields give \eqref{eq:basechange-constants}, and hence
$[\kappa_m:k_m]=\operatorname{lcm}(d,m)/m=d/c$.

The automorphism $\varphi^m$ acts trivially on
$\kappa\cap k_m=\mathbb F_{q^c}$, so it extends over the compositum by
acting as the identity on $k_m$; on $\kappa_m$ this is the
$q^m$-Frobenius. It normalizes $G$ and $H$, and
$\varphi_m^{d/c}\in H$.

Restriction to $\kappa(C)$ identifies the arithmetic group and point
stabilizer with
\[
 \operatorname{Gal}\!\left(
 \kappa(C)/\mathbb F_{q^c}(\mathbf t)
 \right)
 \qquad\text{and}\qquad
 \operatorname{Gal}\!\left(
 \kappa(C)/\mathbb F_{q^c}(\mathbf x)
 \right),
\]
respectively. Since the image of $\varphi$ generates the cyclic quotient
$A/G$ of order $d$, one has
$\langle\varphi^mG\rangle=\langle\varphi^cG\rangle$; the same argument
applies to $U/H$, proving \eqref{eq:basechange-groups}.

The group $G\leqslant A_m$ is transitive on the sheet set
$G/H\cong A/U$, so $A_m$ is transitive as well, with point stabilizer
$U_m=A_m\cap U$. Its action is the restriction of the faithful action
$A\curvearrowright A/U$, and is therefore faithful. Hence
$\operatorname{core}_{A_m}(U_m)=1$. Thus $\mathcal D^{(m)}$ satisfies
Definition~\ref{def:semilinear-P1-datum}; consequently $\Omega_m$ is the
arithmetic Galois closure of the base-changed rational function, and the
descended $k_m$-M\"obius class contains the base change of $f$.

As a consistency check with Proposition~\ref{prop:fixed-pair-core-order},
if $\bar\sigma$ denotes the pair-Frobenius class of the original datum,
then $\operatorname{ord}(\bar\sigma)=d$ and
$\operatorname{ord}(\bar\sigma^m)=d/(d,m)=[\kappa_m:k_m]$. Thus the
general Frobenius-order formula gives the same relative full
constant-field degree.
\end{proof}

Define
\[
\begin{aligned}
 \mathscr D_{\rm perm}(f)
 &:=\{m\geqslant1:f\text{ permutes }\mathbf P^1(\mathbb F_{q^m})\},\\
 \mathscr D_{\rm exc}(f)
 &:=\{m\geqslant1:f\text{ is exceptional over }\mathbb F_{q^m}\},\\
 \mathcal E_f
 &:=\{\,g\geqslant1:g\mid d,\ f\text{ is exceptional over }\mathbb F_{q^g}\,\}.
\end{aligned}
\]

\begin{theorem}[Exact extension-degree criterion]\label{thm:exact-support}
Let $f\in k(X)$ be separable of degree greater than $1$ with Galois closure of genus zero, and retain the notation above.  Then $d\notin\mathcal E_f$.  If $g\in\mathcal E_f$ and $h$ is a positive integer with $h\mid g$, then $h\in\mathcal E_f$.  For every $m\geqslant1$,
\begin{equation}\label{eq:gcd-support}
 m\in\mathscr D_{\rm perm}(f)
 \quad\Longleftrightarrow\quad
 m\in\mathscr D_{\rm exc}(f)
 \quad\Longleftrightarrow\quad
 (m,d)\in\mathcal E_f.
\end{equation}
In particular, $\mathscr D_{\rm perm}(f)=\mathscr D_{\rm exc}(f)$.
\end{theorem}

\begin{proof}
The first equivalence in \eqref{eq:gcd-support} follows from Theorem~\ref{thm:no-accidental}. Put $c=(m,d)$, $u=m/c$, and $N=d/c$. Lemma~\ref{lem:base-change-coset} gives $A_m=A_c$ and $U_m=U_c$, but the designated generating cosets need not be equal. Rather, $G\varphi^m=G(\varphi^c)^u$, where $(u,N)=1$, so $G\varphi^m$ and $G\varphi^c$ are both generating cosets of the same cyclic quotient $A_c/G\cong C_N$. Under the common identification $A_c/U_c\cong G/H$, the two base changes therefore have the same arithmetic/geometric permutation triple; only the designated generator of the cyclic quotient changes. By Lemma~\ref{lem:exceptionality-dictionary}, exceptionality is independent of that choice of generating coset. Hence $f$ is exceptional over $\mathbb F_{q^m}$ if and only if it is
exceptional over $\mathbb F_{q^c}$, which proves the second equivalence.

If $h\mid g\mid d$ and $g\in\mathcal E_f$, then Definition~\ref{def:exceptionality}~(\ref{item:exceptionality-1}) gives infinitely many finite extensions of $\mathbb F_{q^g}$ on which $f$ is a permutation. These extensions are also finite extensions of $\mathbb F_{q^h}$, so $h\in\mathcal E_f$. Thus $\mathcal E_f$ is closed under taking divisors. Finally, for $m=d$ one has
\(A_d=G,\)
and the identity in the generating coset fixes $|G/H|=\deg f>1$ sheets. Therefore $f$ is not exceptional over $\mathbb F_{q^d}$, so $d\notin\mathcal E_f$.
\end{proof}
We henceforth write $\mathscr D(f):=\mathscr D_{\rm perm}(f)=\mathscr D_{\rm exc}(f)$.

\begin{corollary}[Periodicity, density, and forbidden divisors]
Let $\mathcal B_f$ be the set of divisibility-minimal elements of
$\{g:g\mid d,\ g\notin\mathcal E_f\}$. Then
\begin{equation}\label{eq:forbidden-divisors}
 \mathscr D_{\rm perm}(f)
 =\mathscr D_{\rm exc}(f)
 =\{m\geqslant1:b\nmid m\text{ for every }b\in\mathcal B_f\}.
\end{equation}
The common set of extension degrees is purely periodic and has natural density
\begin{equation}\label{eq:support-density}
 \delta_f
 =\frac1d\sum_{g\in\mathcal E_f}\phi\!\left(\frac dg\right).
\end{equation}
If $\mathscr D(f)\ne\varnothing$, then it contains every $m$ with $(m,d)=1$, and in particular it contains $m=1$.

Define
\[
\begin{aligned}
 \mathcal R_f
 &:=\{j\in\mathbb Z/d\mathbb Z:\gcd(j,d)\in\mathcal E_f\},
 \qquad \gcd(0,d):=d,\\
 \Pi_f
 &:=\{h\in\mathbb Z/d\mathbb Z:\mathcal R_f+h=\mathcal R_f\}.
\end{aligned}
\]
Then the least positive period of the common set of extension degrees is
\begin{equation}\label{eq:support-least-period}
 p_f=[\mathbb Z/d\mathbb Z:\Pi_f].
\end{equation}
\end{corollary}

\begin{proof}
By Theorem~\ref{thm:exact-support}, membership depends only on $(m,d)$. Since $\mathcal E_f$ is closed under taking divisors, its complement among the divisors of $d$ is upward closed; therefore $(m,d)\notin\mathcal E_f$ exactly when some $b\in\mathcal B_f$ divides $m$, proving \eqref{eq:forbidden-divisors}. Pure periodicity with period $d$ is immediate. Among the residue classes modulo $d$, exactly $\phi(d/g)$ have gcd equal to $g$, which gives \eqref{eq:support-density}. If $\mathscr D(f)\ne\varnothing$, closure under taking divisors gives $1\in\mathcal E_f$, proving the coprime-degree assertion.

A positive integer $\nu$ is a period exactly when its residue modulo $d$ stabilizes $\mathcal R_f$ additively, that is, when $[\nu]_d\in\Pi_f$. Since $\mathbb Z/d\mathbb Z$ is cyclic, the least positive lift of the subgroup $\Pi_f$ is its index, proving \eqref{eq:support-least-period}.
\end{proof}

\begin{corollary}[Prime-power constants and extension degrees under decomposition]
The following hold.
\begin{enumerate}[label=(\arabic*),ref=\arabic*,font=\itshape]
\item Suppose $d=\ell^a$ is a prime power and the common set of extension degrees is nonempty. Then there is a unique integer $0\leqslant r<a$ such that
\[
 \mathcal E_f=\{1,\ell,\ldots,\ell^r\},
\]
and
\[
 \mathscr D_{\rm perm}(f)
 =\mathscr D_{\rm exc}(f)
 =\{m\geqslant1:\ell^{r+1}\nmid m\}.
\]
Consequently
\[
 \delta_f=1-\ell^{-(r+1)},
 \qquad
 p_f=\ell^{r+1}.
\]
In particular, if $d=\ell$ is prime and $\mathscr D(f)\ne\varnothing$, then it is exactly the set of extension degrees not divisible by $\ell$.
\item For any decomposition $f=f_s\circ\cdots\circ f_1$ over $k$, not
necessarily complete, one has
\[
 \mathscr D_{\rm perm}(f)
 =\bigcap_{i=1}^s\mathscr D_{\rm perm}(f_i),
 \qquad
 \mathscr D_{\rm exc}(f)
 =\bigcap_{i=1}^s\mathscr D_{\rm exc}(f_i).
\]
\end{enumerate}
\end{corollary}

\begin{proof}
If $d=\ell^a$, its positive divisors are $1,\ell,\ldots,\ell^a$. Since $\mathcal E_f$ is closed under taking divisors and does not contain $d$, any nonempty $\mathcal E_f$ has the displayed form. The condition on $\mathscr D(f)$ then says precisely that $\ell^{r+1}$ does not divide $m$, from which the density and least period follow.

For a decomposition, a composite of maps of the finite set $\mathbf P^1(\mathbb F_{q^m})$ is bijective if and only if every factor is bijective. The identity for $\mathscr D_{\rm exc}$ follows by repeated application of the exceptional-tower equivalence \cite[Lemma~3.5]{GMS2003}. The arithmetic normal closure of each factor may be taken inside $\Omega$, so its full constant-field degree divides $d$; hence the intersection formula is compatible with the preceding gcd-$d$ description. This proves both assertions.
\end{proof}

\section{Permutation extension degrees of the classified families}\label{sec:family-support}

Theorem~\ref{thm:exact-support} describes the common set of permutation and exceptional extension degrees in general. We compute it for every classified indecomposable form; the formulas apply at every extension degree, whether or not the rational function remains indecomposable after base change. The affine family is governed by a finite Frobenius module; all remaining families reduce to explicit order, parity, or divisibility conditions.

We retain the notation $\mathscr D(f)$ introduced above; its permutation--exceptionality equality is Theorem~\ref{thm:no-accidental}.

\subsection{Affine extension degrees from Frobenius modules}

Consider the $p$-semi-elementary family in the normalized form of Theorem~\ref{thm:affine-wild}. Write $G=V\rtimes\mu_n$ and $H=\mu_n$, where $\mu_1=1$, and for $n>1$ the translation space $V$ is naturally an $\mathbb F_{p^{\operatorname{ord}_n(p)}}$-vector space. The normalized Frobenius lift acts by conjugation on the translation group as $v\mapsto b^{-1}v^q$ for some $b\in\mu_n$. Replacing the lift by another representative of the same arithmetic point-stabilizer coset allows us to take $\Phi(v)=v^q$ \emph{on the translation module}. This convention does not assert a global point-side normalization of the semilinear map on $C$.

Since $\Phi\zeta\Phi^{-1}=\zeta^q$ for $\zeta\in\mu_n$, the subgroup $\mu_n$ is normal in $\langle\mu_n,\Phi\rangle\leqslant\operatorname{GL}_{\mathbb F_p}(V)$. Put $d_V:=\operatorname{ord}_{\langle\mu_n,\Phi\rangle/\mu_n}(\Phi\mu_n)$.
For an affine family $f$, let $G_f$ and $A_f$ denote the geometric and arithmetic monodromy groups of its Galois closure, and let $A_f^{(m)}$ denote the arithmetic monodromy group after base change to $\mathbb F_{q^m}$.

\begin{theorem}[Affine Frobenius-module criterion]\label{thm:affine-module-support}
With this notation,
\[
 G_f\cong V\rtimes\mu_n,
 \qquad
 A_f\cong V\rtimes\langle\mu_n,\Phi\rangle,
 \qquad
 [\kappa_f:k]=d_V.
\]
After base change to $\mathbb F_{q^m}$,
\[
 A_f^{(m)}\cong
 V\rtimes\langle\mu_n,\Phi^m\rangle,
 \qquad
 [A_f^{(m)}:G_f]
 =\frac{d_V}{(d_V,m)}.
\]
For every $m\geqslant1$, the following are equivalent:
\begin{enumerate}[label=\textup{(\roman*)},ref=\roman*]
\item $m\in\mathscr D(f)$;
\item $\ker(1-\zeta\Phi^m)=0$ for every $\zeta\in\mu_n$;
\item $\prod_{\zeta\in\mu_n}\det_{\mathbb F_p}(1-\zeta\Phi^m)\neq0$.
\end{enumerate}
If $L_V(Z)$ is the subspace polynomial of $V$, then these conditions are also equivalent to $\gcd\bigl(L_V(Z),Z^{q^m}-\zeta^{-1}Z\bigr)=Z$ for every $\zeta\in\mu_n$, and it suffices to test one representative from each $\mu_n$-conjugacy class in the coset $\mu_n\Phi^m$.
\end{theorem}

\begin{proof}
The translation subgroup acts freely and transitively and identifies the sheet set $G/H$ with $V$. If $U$ is the arithmetic point stabilizer, then $A=V\rtimes U$. The kernel of the conjugation action $U\to\operatorname{Aut}(V)$ is normal in $U$ and centralizes $V$, hence is normal in $A$; since it lies in $U$, core-freeness makes it trivial. Thus $U$ acts faithfully on $V$. Its image is exactly $\langle\mu_n,\Phi\rangle$, because $H$ acts as $\mu_n$ and the chosen Frobenius lift acts as $\Phi$. Hence $U/H\cong\langle\mu_n,\Phi\rangle/\mu_n$, giving the monodromy and constant-field assertions; in particular, $d_V=d=[\kappa:k]$. After base change the same argument replaces $\Phi$ by $\Phi^m$.

At extension degree $m$, an element in the generating arithmetic coset acts on $V$ as $x\mapsto a+\zeta\Phi^m(x)$, with $a\in V$ and $\zeta\in\mu_n$. Its fixed-point equation is
\((1-\zeta\Phi^m)x=a.\)
Thus every element of the generating coset has exactly one fixed sheet if and only if every $1-\zeta\Phi^m$ is bijective: if one of these operators is singular, then taking $a=0$ already gives more than one fixed sheet. Proposition~\ref{prop:coset-package} and Theorem~\ref{thm:no-accidental} give the first equivalence, and finite-dimensional linear algebra gives the determinant test.

A nonzero vector fixed by $\zeta\Phi^m$ is exactly a nonzero common root of $L_V(Z)$ and $Z^{q^m}-\zeta^{-1}Z$, which proves the gcd criterion. Finally
\[
 \eta(\zeta\Phi^m)\eta^{-1}
 =\zeta\eta^{1-q^m}\Phi^m,
\]
for $\eta\in\mu_n$, so conjugacy compresses the scalar tests exactly as asserted. In fact the scalar-conjugacy classes are indexed by $\mu_n/(\mu_n)^{1-q^m}$, so one representative from each of the $\gcd(n,q^m-1)$ classes suffices. This is distinct from independence of the Frobenius lift: replacing $\Phi$ by $\eta\Phi$ with $\eta\in\mu_n$ does not change the coset $\mu_n\Phi^m$, since $(\eta\Phi)^m=\eta^{1+q+\cdots+q^{m-1}}\Phi^m$.
\end{proof}

\begin{corollary}[Split, inversion, and additive formulas]
In the notation above:
\begin{enumerate}[label=(\arabic*),ref=\arabic*,font=\itshape]
\item If $q^m\equiv1\pmod n$, then
\[
 m\in\mathscr D(f)
 \Longleftrightarrow
 V\cap\mathbb F_{q^{mn}}=\{0\}.
\]
\item If $q^m\equiv-1\pmod n$ and $n$ is odd, then
\[
 m\in\mathscr D(f)
 \Longleftrightarrow
 V\cap\mathbb F_{q^m}=\{0\}.
\]
\item If $q^m\equiv-1\pmod n$ and $n$ is even, then
\[
 m\in\mathscr D(f)
 \Longleftrightarrow
 V\cap\mathbb F_{q^{2m}}=\{0\}.
\]
\item If $n=1$ and the additive case is $k$-indecomposable, then
\[
 \mathscr D(f)=\{m:d_V\nmid m\}.
\]
If $d_V>1$, its least period is $d_V$ and its density is $1-1/d_V$; if $d_V=1$, $\mathscr D(f)=\varnothing$.
\end{enumerate}
\end{corollary}

\begin{proof}
In the split case all factors commute and
\[
 \prod_{\zeta\in\mu_n}(1-\zeta\Phi^m)=1-\Phi^{mn}.
\]
For inversion and odd $n$, multiplication by $2$ is invertible on $\mathbb Z/n\mathbb Z$, so the whole linear coset is one scalar-conjugacy class. For inversion and even $n$, one has
\((\zeta\Phi^m)^2=\Phi^{2m}.\)
Here $p$ is automatically odd because $(n,p)=1$. Conversely, on $\ker(1-\Phi^{2m})$ the operator $\Phi^m$ is an involution, hence has a nonzero $\pm1$-eigenvector; since $n$ is even, $\pm1\in\mu_n$. For the additive case, $k$-indecomposability makes $V$ a simple $\mathbb F_p[\Phi]$-module. Thus the minimal polynomial of $\Phi$ is irreducible with nonzero constant term; its roots are Frobenius conjugates, so they all have the same multiplicative order, namely $\operatorname{ord}(\Phi)=d_V$. This gives the stated divisibility criterion.
\end{proof}

\subsection{Extension-degree formulas for the non-Lie cases}

\begin{proposition}[Extension-degree formulas for the non-Lie cases]\label{prop:closed-nonlie-support}
For the tame and small-characteristic non-Lie classes, the complete set of extension degrees is as follows.
\begin{enumerate}[label=(\arabic*),ref=\arabic*,font=\itshape]
\item\label{item:closed-nonlie-support-1} For $X^\ell$,
\[
 \mathscr D(f)
 =\{m:\operatorname{ord}_\ell(q)\nmid m\}.
\]
\item For a nonsplit R\'edei class,
\[
 \mathscr D(f)
 =\{m:\operatorname{ord}_\ell(-q)\nmid m\}.
\]
\item For every tame Dickson class and for the characteristic-$2$ wild dihedral class,
\[
 \mathscr D(f)
 =\{m:\operatorname{ord}_\ell(q^2)\nmid m\}.
\]
\item\label{item:closed-nonlie-support-4} For the $V_4/A_4$ quartic,
\[
 \mathscr D(f)=\{m:3\nmid m\}.
\]
\item The primitive polyhedral cases, as well as $W_5,W_6,W_{10}$ in characteristic $3$, have no extension degrees.
\end{enumerate}
Whenever $\mathscr D(f)$ in (\ref{item:closed-nonlie-support-1})--(\ref{item:closed-nonlie-support-4}) has the form $\{m:c\nmid m\}$ with $c>1$, its least period is $c$ and its density is $1-1/c$.
\end{proposition}

\begin{proof}
The split cyclic assertion is the power-map criterion on $\mathbb F_{q^m}^\times$. For the R\'edei form, odd extension degrees identify the quotient line with the norm-one torus of order $q^m+1$, while even extension degrees split the quadratic pair; hence the obstruction is
\(\ell\mid q^m-(-1)^m,\)
which is equivalent to $\operatorname{ord}_\ell(-q)\mid m$.

For Dickson polynomials use
\[
 D_\ell\left(z+\frac az,a\right)
 =z^\ell+\left(\frac az\right)^\ell.
\]
The rational function is a permutation over $\mathbb F_{q^m}$ exactly when $\ell$ divides neither $q^m-1$ nor $q^m+1$, equivalently when $\ell\nmid q^{2m}-1$. The Klein-four formula follows from the permutation triple $(A_4,V_4,C_3,1)$: every nontrivial $V_4$-coset consists of $3$-cycles with one fixed sheet, whereas for $3\mid m$ the arithmetic group after base change is $V_4$ and its designated coset contains the identity.

For the primitive polyhedral pairs, the geometric group or its only outer extension always contains a generating-coset element with a number of fixed points different from one. Concretely, transpositions give $2$ fixed points in the degree-$4$ $A_4$ outer action, $3$ in degree $5$, $0$ in degree $6$, and $4$ in degree $10$; $S_4$ has no nontrivial outer extension. The characteristic-$3$ representatives also have the rational two-point fibres used in the proof of Theorem~\ref{thm:primitive-nap}. The period and density assertions are immediate.
\end{proof}

\subsection{Extension degrees for the Lie-type cases}

For the geometrically primitive Lie-type cases, fix a finite extension
$k_m=\mathbb F_{q^m}$ of the ground field. Let $L$ be the geometric group,
let $H$ be the maximal geometric point stabilizer, and let $x$ be a
Frobenius lift for the base-changed datum over $k_m$. Then $U=\langle H,x\rangle$ is maximal in $B=\langle L,x\rangle$: any intermediate subgroup has intersection with $L$ equal to either $H$ or $L$, and the two cases give $U$ or $B$. Thus the primitive hypotheses in the Guralnick--M\"uller--Saxl exceptional-group classification apply directly. For subfield cases with exponent $r\ne p$, the field-automorphism construction is the case of \cite[Theorem~3.15]{GMS2003}; field-diagonal forms use \cite[Theorem~3.29(b)(i)]{GMS2003} together with the centralizer condition for this Frobenius lift.

For the remaining formulas write
\[
 Q=p^\alpha,
 \qquad
 k=\mathbb F_{p^s},
 \qquad
 h_0=(\alpha,s),
 \qquad
 d_F=\frac{\alpha}{h_0}.
\]
For extension degree $m$ put
\[
 h_m=(\alpha,sm)=h_0(d_F,m),
 \qquad
 e_m=\frac{d_F}{(d_F,m)}.
\]
For the odd-PSL two-form families below, the pure-field and field-diagonal forms mean those with Frobenius outer images $\bar\phi^{\,s}$ and $\bar\delta\bar\phi^{\,s}$, respectively. Their full constant-field degrees are $d_F$ and $\operatorname{lcm}(2,d_F)$. A square-class parameter labels these forms only after the chosen quotient-coordinate normalization; the outer-image description is intrinsic. The least period can be strictly smaller than the full constant-field degree.

\begin{proposition}[Extension degrees for Cartan and the remaining Lie cases]\label{prop:lie-cartans-support}
The defining-characteristic Lie-type cases have the following sets of extension degrees.
\begin{enumerate}[label=(\arabic*),ref=\arabic*,font=\itshape]
\item\label{item:lie-cartans-support-1} Every Borel case, every geometrically primitive case with stabilizer $A_4,S_4$, or $A_5$, every even-characteristic split-Cartan case, every odd full-$\operatorname{PGL}_2(Q)$ Cartan case, and every odd square-subfield PSL case
\[
 Q=Q_0^2,
 \qquad
 H=\operatorname{PGL}_2(Q_0)
\]
has no extension degrees.
\item\label{item:lie-cartans-support-2} Let $Q=2^\alpha$ with $\alpha>1$ and let $H=D_{2(Q+1)}$. Then
\[
 \mathscr D(f)=
 \begin{cases}
 \{m:(m,\alpha)=1\},&\alpha\text{ odd and }(\alpha,s)=1,\\
 \varnothing,&\text{otherwise}.
 \end{cases}
\]
In the nonempty case the least period is $\operatorname{rad}(\alpha)$ and the
density is $\phi(\alpha)/\alpha$.
\item\label{item:lie-cartans-support-3} For the displayed odd-PSL split-Cartan forms $T_{Q,\eta}$ of Proposition~\ref{prop:split-cartan}, the form with outer image $\bar\phi^{\,s}$ has no extension degrees. The field-diagonal twist has
\[
 \mathscr D(f)=\{m:m\text{ odd}\}
\]
exactly when $\alpha$ and $d_F$ are even; otherwise its set of extension degrees is empty. In the nonempty case the least period is $2$ and the density is $1/2$.
\item\label{item:lie-cartans-support-4} For the nonsplit-Cartan pair $H=D_{Q+1}$ in $\operatorname{PSL}_2(Q)$, the form with outer image $\bar\phi^{\,s}$ has no extension degrees. The field-diagonal twist has
\[
 \mathscr D(f)=\{m:(m,2\alpha)=1\}
\]
exactly when $p=3$, $\alpha>1$ is odd, and $(\alpha,s)=1$; otherwise its set of extension degrees is empty. In the nonempty case the least period is $\operatorname{rad}(2\alpha)$ and the density is $\phi(2\alpha)/(2\alpha)$.
\end{enumerate}
\end{proposition}

\begin{proof}
Fix an extension degree $m$. If the arithmetic monodromy group after base change to $k_m$ equals the geometric group, its generating coset contains the identity and is not exceptional. At the $Q=3$ Borel endpoints, the degree-four arithmetic image is contained in $S_4$: for geometric $S_4$ the identity applies, and for geometric $A_4$ an outer transposition fixes two sheets. Thus these endpoints also have no extension degrees. In all remaining geometrically primitive cases in part~(\ref{item:lie-cartans-support-1}), the hypotheses of \cite[Theorem~3.29]{GMS2003} apply; the Borel, $A_4,S_4,A_5$, even split-Cartan, and odd square-subfield PSL cases do not occur in its exceptional list.

For a classified full $\operatorname{PGL}_2(Q)$ Cartan pair in odd characteristic, take the pure field automorphism $x$ in the generating coset. It fixes $H$. Its fixed subgroup $\operatorname{PGL}_2(p^h)$ is not contained in the dihedral stabilizer; if $c\in C_G(x)\setminus H$, then $cH$ is a second fixed sheet. Hence these PGL Cartan cases have no extension degrees either.

We prove parts~(\ref{item:lie-cartans-support-2})--(\ref{item:lie-cartans-support-4}) from Propositions~\ref{prop:split-cartan} and~\ref{prop:nonsplit-cartan}. After base change to $k_m=\mathbb F_{p^{sm}}$, put $h_m:=(\alpha,sm)$ and $e_m:=\alpha/h_m$, and write $sm=h_mu$; then $(u,e_m)=1$. A field-automorphism image is a coprime power of the canonical generator over $\mathbb F_{p^{h_m}}$. For an odd-$\operatorname{PSL}_2$ diagonal form, $(\bar\delta\bar\phi^{\,s})^m=\bar\delta^{\,m}\bar\phi^{\,sm}$. For odd $m$, if $\psi=\bar\delta\bar\phi^{\,h_m}$ then
\begin{equation}\label{eq:cartan-common-coprime-power}
 \bar\delta\bar\phi^{\,sm}=\psi^w,
 \qquad (w,\operatorname{lcm}(2,e_m))=1,
\end{equation}
with $w=u$ when $e_m$ is even and the odd member of $\{u,u+e_m\}$ otherwise; for even $m$ the diagonal component disappears.

The small split endpoints require separate treatment. At $Q=3$, the stable chain $C_2\lneq V_4\lneq A_4$ on the PSL side, or $V_4\lneq D_8\lneq S_4$ on the PGL side, gives a factor of degree two, which has no permutation extension degrees by Theorem~\ref{thm:no-accidental}. At $Q=5$ on the PSL side, $V_4\lneq A_4\lneq A_5$ is stable under every normalizing Frobenius lift. Its outer factor has the natural degree-five $A_5$ action; after any base extension its arithmetic image is $A_5$ or $S_5$. The generating coset contains respectively the identity, fixing five sheets, or a transposition, fixing three. That factor, and hence the composite, has no permutation extension degrees. The PGL split endpoint at $Q=5$ has already been excluded by the preceding fixed-sheet argument. This preserves the empty-set assertions for all displayed small-endpoint forms without inferring them merely from decomposability.

For the remaining Cartan pairs, $\mathrm N_G(H)=H$. Let $\tau_m$ be
the normalizing Frobenius lift for the datum over $k_m$, and let $\tau_c$
be the chosen normalizing lift of the corresponding canonical generator
over $\mathbb F_{p^{h_m}}$. Then $\tau_mG=(\tau_cG)^w$, where $w$ is
coprime to the order of $\tau_cG$ in the cyclic quotient.
Since both lifts normalize $H$, their quotient
$\tau_m\tau_c^{-w}$ lies in $\mathrm N_G(H)=H$. Hence $\langle G,\tau_m\rangle=\langle G,\tau_c\rangle$ and $\langle H,\tau_m\rangle=\langle H,\tau_c\rangle$, while $G\tau_m=G\tau_c^w$.
With $w$ as in~\eqref{eq:cartan-common-coprime-power}, Theorems~\ref{thm:exact-support} and~\ref{thm:no-accidental} transport the canonical fixed-field tests to every extension.

\medskip
\noindent\emph{Nonsplit Cartan.}
Let $c\mid\alpha$, put $r=p^c$, $n=\alpha/c$, and write
$z=W+W^{-1}\in\mathbb F_r$. Away from $z=\pm2$, either $W^r=W$ or
$W^r=W^{-1}$. Direct substitution in~\eqref{eq:NQ} gives
\begin{equation}\label{eq:nonsplit-cartan-fixed-field-evaluation}
 N_Q(z)=
 \begin{cases}
  0,&W^r=W,\\
  0,&W^r=W^{-1}\text{ and }n\text{ is even},\\
  z+2,&W^r=W^{-1}\text{ and }n\text{ is odd}.
 \end{cases}
\end{equation}
Also $N_Q(\infty)=0$.

Suppose first that $p=2$. If $r=2$ and $n$ is odd, then $N_Q(0)=\infty$, $N_Q(1)=1$, and $N_Q(\infty)=0$, so $N_Q$ permutes $\mathbf P^1(\mathbb F_2)$. In every other case it does not. Indeed, if
$r>2$, choose $v\in\mathbb F_r^\times\setminus\{1\}$ and put
$x=v+v^{-1}\ne0$; then~\eqref{eq:nonsplit-cartan-fixed-field-evaluation}
gives $N_Q(x)=N_Q(\infty)=0$. If $r=2$ and $n$ is even, the quadratic point
over $z=1$ satisfies $W^Q=W$, so $N_Q(1)=N_Q(\infty)=0$. Thus $N_Q$ permutes $\mathbf P^1(\mathbb F_{2^c})$ if and only if $c=1$ and $\alpha$ is odd. Consequently the even nonsplit case has nonempty $\mathscr D(f)$ exactly when
$\alpha$ is odd and $(\alpha,s)=1$, and then $\mathscr D(N_Q)=\{m:(m,\alpha)=1\}$. This proves part~(\ref{item:lie-cartans-support-2}), including the empty case when the
defining exponent is even.

Now let $p$ be odd. From Proposition~\ref{prop:nonsplit-cartan},
\[
 M_{Q,\eta}(x)^2
 =\eta N_Q\!\left(-\frac{x^2}{\eta}-2\epsilon\right),
 \qquad \epsilon=(-1)^{(Q-1)/2}.
\]
For $x\ne0$, the discriminant of the corresponding quadratic equation for $W$ has square class $x^2+4\epsilon\eta$.
If $n$ is even, then $\epsilon=1$ and
\eqref{eq:nonsplit-cartan-fixed-field-evaluation} forces nonpermutation: when
$-\eta$ is square there are two rational points above the pole endpoint
$z=2$; otherwise, writing $\chi_2$ for the quadratic character of $\mathbb F_r$, the identity
\[
 \sum_{x\in\mathbb F_r}\chi_2(x^2-A)=-1\qquad(A\ne0)
\]
produces a nonzero split value and hence a collision
$M_{Q,\eta}(x)=M_{Q,\eta}(-x)=0$ (for $r=3$ take $x=\pm1$ directly).
If $n$ is odd and $\epsilon=1$, the same pole/collision alternative applies.
If $n$ is odd and $\epsilon=-1$, the character sum again gives a nonzero
collision for $r\geqslant7$. At the only remaining endpoint $r=3$, the square form
has $M_{Q,1}(1)=M_{Q,1}(-1)=0$, whereas for the nonsquare form $M_{Q,-1}(0)=\infty$, $M_{Q,-1}(\infty)=0$, and $M_{Q,-1}(\pm1)^2=1$, and oddness makes the last two values opposite. Hence $M_{Q,\eta}$ permutes $\mathbf P^1(\mathbb F_r)$ if and only if $r=3$, $n$ is odd, and $\eta$ is nonsquare. Transporting this criterion with $c=h_m$ shows that the pure form always has
empty set of extension degrees. A diagonal form can have nonempty set of extension degrees only when
$p=3$, $\alpha>1$ is odd, and $(\alpha,s)=1$; in that case a base-field
nonsquare remains nonsquare over $\mathbb F_{3^{sm}}$ precisely for odd $m$, while
$h_m=1$ if and only if $(m,\alpha)=1$. Therefore $\mathscr D(M_{Q,\eta})=\{m:(m,2\alpha)=1\}$, which proves part~(\ref{item:lie-cartans-support-4}) and its period and density assertions.

\medskip
\noindent\emph{Split Cartan.}
For the odd-$\operatorname{PSL}_2$ form $T_{Q,\eta}$, Proposition~\ref{prop:split-cartan}
gives
\[
 T_{Q,\eta}(x)^2
 =\eta S_Q\!\left(\frac{x^2}{\eta}-2\epsilon\right).
\]
If $c\mid\alpha$, $r=p^c$, and $n=\alpha/c$, then direct substitution
in~\eqref{eq:SQ} gives, away from the branch endpoints,
\[
 \begin{aligned}
 n\text{ even}&\Longrightarrow S_Q(z)=z+2,\\
 n\text{ odd and }z\text{ nonsplit}&\Longrightarrow S_Q(z)=0.
 \end{aligned}
\]
Thus for even $n$ one has $T_{Q,\eta}(x)^2=x^2$ away from
$x^2=4\eta$: the square form has two finite rational poles, while the
nonsquare form has no finite rational pole and oddness makes every pair
$\{x,-x\}$ map bijectively to itself. For odd $n$, the Cartan discriminant
has square class $x^2-4\epsilon\eta$; the same character sum gives a nonzero
$x$ with strict nonsplit discriminant, whence
$T_{Q,\eta}(x)=T_{Q,\eta}(-x)=0$ (with $r=3$ checked at $x=\pm1$).
Therefore $T_{Q,\eta}$ permutes $\mathbf P^1(\mathbb F_r)$ if and only if $n$ is even and $\eta$ is nonsquare. The pure form is square after every base change and hence has no extension degrees.
For a diagonal form, $\eta$ remains nonsquare over $k_m$ exactly when $m$ is
odd. For odd $m$, $e_m=d_F/(d_F,m)$ is even exactly when $d_F$ is even. This gives the odd-degree set of extension degrees in
part~(\ref{item:lie-cartans-support-3}) exactly when $d_F$ is even; since $d_F\mid\alpha$, this is
equivalent to the stated condition that $\alpha$ and $d_F$ are even. The
least period and density are $2$ and $1/2$.
\end{proof}

The subfield cases admit one uniform formula when the prime exponent differs from the characteristic.

\begin{proposition}[Prime-exponent subfield extension degrees, $r\neq p$]\label{prop:lie-subfield-support}
Suppose $Q=Q_0^r$, where $Q_0=p^{\alpha/r}$, $r$ is prime, and $r\neq p$, and consider any of the corresponding maximal subfield forms in the wild classification: even characteristic, the odd-PSL odd-prime-exponent pure or diagonal forms, or odd PGL. The odd square-subfield PSL case $r=2$ with stabilizer $\operatorname{PGL}_2(Q_0)$ was already excluded in Proposition~\ref{prop:lie-cartans-support}~(\ref{item:lie-cartans-support-1}). Set $c_r:=\operatorname{ord}_r(p^2)$. For every $m\geqslant1$,
\[
 m\in\mathscr D(f)
 \quad\Longleftrightarrow\quad
 \begin{cases}
 r>3,&p=2,\\
 r\nmid p(p^2-1),&p\text{ odd},
 \end{cases}
 \quad r\mid e_m,
 \quad r\nmid p^{2h_m}-1.
\]
In particular the two odd-PSL quadratic twists have the same complete set of extension degrees.

Equivalently, put $t:=v_r(d_F)$ and $c_*:=c_r/(c_r,h_0)$. If the structural condition fails, or $t=0$, or $c_*=1$, then $\mathscr D(f)=\varnothing$. Otherwise:
\begin{enumerate}[label=(\arabic*),ref=\arabic*,font=\itshape]
\item\label{item:lie-subfield-support-1} if $c_*\nmid d_F$, then
\[
 \mathscr D(f)=\{m:r^t\nmid m\},
\]
with least period $r^t$ and density $1-r^{-t}$;
\item\label{item:lie-subfield-support-2} if $c_*\mid d_F$, then
\[
 \mathscr D(f)=\{m:r^t\nmid m,\ c_*\nmid m\},
\]
with least period $r^tc_*$ and density $(1-r^{-t})(1-c_*^{-1})$.
\end{enumerate}
\end{proposition}

\begin{proof}
Fix an extension degree $m$. Let $L=\operatorname{PSL}_2(Q)$, let $G$
be the geometric group of the chosen row, and write the arithmetic
monodromy group after base change to $k_m=\mathbb F_{p^{sm}}$ as
$B=\langle G,x\rangle$, where $x$ is a Frobenius lift, with point
stabilizer $U$. Put $H_0:=U\cap L=\operatorname{PSL}_2(Q_0)$ and $J:=U\cap G$. Thus $J=\operatorname{PSL}_2(Q_0)$ in the PSL row and $J=\operatorname{PGL}_2(Q_0)$ in the PGL row, and $\mathrm N_G(H_0)=J$. The generator $x\in U$ has field-component order $e_m$.

We first prove the necessity of $r\mid e_m$ without identifying a finite field-diagonal centralizer with an algebraic fixed group. In the standard subfield model a representative of the generating coset normalizing $J$ has the form $x=\operatorname{diag}(d,1)\Phi_{p^{sm}}$, with $d\in\mathbb F_{Q_0}^{\times}$; take $d=1$ for the pure or PGL form and a nonsquare for the diagonal PSL form. Suppose $r\nmid e_m$. Since $\alpha=h_me_m$, one then has
\[
 \mathbb F_{p^{h_m}}\cap\mathbb F_{Q_0}
 =\mathbb F_{p^{h_m/r}}.
\]
Every diagonal projectivity $\operatorname{diag}(c,1)$ with $c\in\mathbb F_{p^{h_m}}^{\times}$ commutes with $x$. At least $(p^{h_m}-1)/2$ of these belong to $L$, whereas only $p^{h_m/r}-1$ can belong to the subfield group $J$. The former number is strictly larger, since the remaining structural subfield exponents are odd primes. Choose such a commuting projectivity outside $J$. Then $x$ fixes both $J$ and its translate by that projectivity in $G/J$, contradicting the unique-fixed-sheet criterion. Hence $r\mid e_m$.

The precise classification \cite[Theorem~3.29(a)(i),(b)(i)]{GMS2003} now gives the stated structural condition and $r\nmid p^{2h_m}-1$. Conversely, these conditions and $r\mid e_m$ give the exceptional subfield action by \cite[Theorems~3.15 and~3.29]{GMS2003}; the order-$r$ field component is present also for the diagonal form, since $r$ is odd. To identify its stabilizer with the one in the present quotient, conjugate the stabilizer $M$ in that action so that $M\cap L=H_0$. Since $\mathrm N_B(H_0)=\langle J,x\rangle=U$, we have $M\leqslant U$. Both groups intersect $L$ in $H_0$ and surject onto $B/L$, so they have the same order. Thus $M=U$, closing the converse for both PSL and PGL rows.

For the closed form, use
\(h_m=h_0(d_F,m), \qquad e_m=d_F/(d_F,m).\)
The condition $r\mid e_m$ is $t>0$ and $r^t\nmid m$. Put $c_0=(c_r,h_0)$, so $c_r=c_0c_*$ and $h_0=c_0h'$ with $(c_*,h')=1$. The condition $r\nmid p^{2h_m}-1$ is equivalent to $c_r\nmid h_m$, and hence to $c_*\nmid(d_F,m)$. If $c_*\mid d_F$, this becomes $c_*\nmid m$; if not, it is automatic. Finally $c_r\mid r-1$, so $(r,c_*)=1$; in case~(\ref{item:lie-subfield-support-2}) the two forbidden divisibilities are consequently independent modulo $r^tc_*$, which gives both the least period and the product density, while case~(\ref{item:lie-subfield-support-1}) has the single forbidden divisor $r^t$.
\end{proof}

\begin{proposition}[The characteristic-two square-subfield case]\label{prop:even-square-subfield-empty-support}
Suppose $p=2$, $Q=Q_0^2$, and $Q_0\ne2$. For the primitive square-subfield pair $G=\operatorname{PGL}_2(Q)=\operatorname{PSL}_2(Q)$ and $H=\operatorname{PGL}_2(Q_0)$, one has
\[
 \mathscr D(f)=\varnothing.
\]
\end{proposition}

\begin{proof}
Base extension changes the generating Frobenius coset but not the geometric pair $(\operatorname{PGL}_2(Q),\operatorname{PGL}_2(Q_0))$, so its structural subfield exponent remains $2$. The characteristic-$2$ exceptional subfield case in \cite[Theorem~3.29(a)]{GMS2003} requires prime exponent greater than $3$. Hence no finite base extension is exceptional, and Theorem~\ref{thm:no-accidental} gives empty permutation set of extension degrees as well.
\end{proof}

\begin{proposition}[Characteristic-prime subfield and novelty extension degrees]\label{prop:lie-novelty-support}
Assume $p$ is odd.
\begin{enumerate}[label=(\arabic*),ref=\arabic*,font=\itshape]
\item Suppose $Q=Q_0^p$ and put $t:=v_p(d_F)$. For geometric $\operatorname{PGL}_2(Q)$, $\mathscr D(f)=\varnothing$ if $t=0$ and otherwise
\[
 \mathscr D(f)=\{m:p^t\nmid m\},
\]
with least period $p^t$ and density $1-p^{-t}$.
\item For geometric $\operatorname{PSL}_2(Q)$ in the same $r=p$ subfield case, the form with outer image $\bar\phi^{\,s}$ has no extension degrees. The field-diagonal form is nonempty exactly when $d_F$ is odd and $t>0$; then
\[
 \mathscr D(f)
 =\{m:2\nmid m,\ p^t\nmid m\},
\]
with least period $2p^t$ and density $\tfrac12(1-p^{-t})$.
\item Among the Lie-type novelties in Theorem~\ref{thm:wild-novelty}, exactly one has nonempty $\mathscr D(f)$. It is the degree-$45$ case with $G=\operatorname{PSL}_2(9)$, $A=M_{10}$, and $H=D_8$. It occurs precisely over base fields $\mathbb F_{3^s}$ with $s$ odd, and then
\[
 \mathscr D(f)=\{m:m\text{ odd}\},
\]
with least period $2$ and density $1/2$. Every other Lie novelty has no extension degrees.
\end{enumerate}
\end{proposition}

\begin{proof}
For $r=p$, use \cite[Theorem~3.29(b)(ii)]{GMS2003}. In PGL the diagonal involution is geometric, so the only condition is that the field-component order $e_m$ of the base-changed Frobenius be divisible by $p$. In PSL the cyclic outer subgroup must also contain the pure diagonal involution; for a field-diagonal generator this is equivalent to the base-change exponent $m$ and its field-component order $e_m$ both being odd. The formulas follow from $e_m=d_F/(d_F,m)$.

Every Lie novelty in Theorem~\ref{thm:wild-novelty} has $[A:G]=2$, so its full constant-field degree is two. Theorem~\ref{thm:exact-support} therefore makes its extension-degree set either empty or exactly the odd positive integers; the latter occurs precisely when the base form is exceptional. It remains to intersect the base-field novelty list with \cite[Theorem~3.29]{GMS2003}. The unique survivor is the split-Cartan field-diagonal case with $G=\operatorname{PSL}_2(9)$, $A=M_{10}$, $H=D_8$, and degree $45$. The other small Cartan cases and the infinite $A_4/S_4$ family have empty sets. This proves the assertion for every extension degree, including those for which the arithmetic action is no longer primitive.
\end{proof}

\subsection{Summary table and small-field conventions}

For quick reference, Table~\ref{tab:master-support} records the extension arithmetic using the classification terminology above. Rows with identical sets of extension degrees are occasionally coalesced: the dihedral line covers both the tame and characteristic-$2$ wild cases, and the primitive-polyhedral line also includes the characteristic-$3$ $A_5$ cases. All rows are restricted to the structural $k$-indecomposable ranges. The odd-PSL Cartan entries refer to the geometrically primitive ranges; the nonmaximal cases at $Q=7,9,11$ are listed under Lie-type novelties. The split-Cartan pairs at $Q=3,5$ are $k$-decomposable and contribute no row; the $Q=3$ indecomposable Lie cases are Borel cases. Theorem~\ref{thm:master-occurrence} separates coalesced rows again when their class counts over $k$ differ. For the affine family the table refers to Theorem~\ref{thm:affine-module-support}, and for the $r\neq p$ subfield case to Proposition~\ref{prop:lie-subfield-support}.

\begingroup
\small
\setlength{\LTpre}{6pt}
\setlength{\LTpost}{6pt}
\begin{longtable}{@{}>{\raggedright\arraybackslash}p{0.18\textwidth}>{\raggedright\arraybackslash}p{0.20\textwidth}>{\raggedright\arraybackslash}p{0.50\textwidth}@{}}
\caption{Extension degrees of the classified indecomposable forms.}\label{tab:master-support}\\
\toprule
Classification case(s) & $k$-form(s) & extension degrees $\mathscr D(f)$ \\
\midrule
\endfirsthead
\toprule
Classification case(s) & $k$-form(s) & extension degrees $\mathscr D(f)$ \\
\midrule
\endhead
cyclic & split $X^\ell$ & $\operatorname{ord}_\ell(q)\nmid m$ \\
cyclic & nonsplit R\'edei & $\operatorname{ord}_\ell(-q)\nmid m$ \\
dihedral & all classified Dickson forms & $\operatorname{ord}_\ell(q^2)\nmid m$ \\
geometrically decomposable $V_4/A_4$ & unique quartic & $3\nmid m$ \\
primitive polyhedral & all forms & empty \\
affine & Ore/Frobenius form & finite-module criterion \\
Lie Borel & all forms & empty \\
even split Cartan & all forms & empty \\
even nonsplit Cartan & unique & $(m,\alpha)=1$ if $\alpha>1$ is odd and $(\alpha,s)=1$; otherwise empty \\
odd PSL split Cartan & pure & empty \\
odd PSL split Cartan & diagonal & odd $m$ if $\alpha$ and $d_F$ are even; otherwise empty \\
odd PGL split Cartan & all forms & empty \\
odd PSL nonsplit Cartan & pure & empty \\
odd PSL nonsplit Cartan & diagonal & $(m,2\alpha)=1$ if $p=3$, $\alpha>1$ is odd, and $(\alpha,s)=1$; otherwise empty \\
odd PGL nonsplit Cartan & all forms & empty \\
odd square-subfield PSL & pure & empty \\
other prime-exponent subfield $r\neq p$ & every form & uniform subfield formula \\
characteristic-$2$ square-subfield $r=p=2$ & unique PGL form & empty \\
odd subfield $r=p$ & PGL & $p^t\nmid m$ if $t>0$; otherwise empty \\
odd subfield $r=p$ & PSL pure & empty \\
odd subfield $r=p$ & PSL diagonal & $2\nmid m$ and $p^t\nmid m$ if $d_F$ is odd and $t>0$; otherwise empty \\
geometrically primitive Lie cases with $H\cong A_4,S_4,A_5$ & all forms & empty \\
Lie-type novelty & $Q=9$, $M_{10}$, degree $45$; exists only for odd $s$ & odd $m$ \\
other Lie-type novelties & all forms & empty \\
\bottomrule\end{longtable}
\endgroup

The accidental small-group identifications introduce no duplication. The groups $\operatorname{PGL}_2(2)\cong S_3$, $\operatorname{PSL}_2(3)\cong A_4$, $\operatorname{PGL}_2(3)\cong S_4$, and $\operatorname{PSL}_2(4)=\operatorname{PGL}_2(4)\cong A_5$, together with $\operatorname{PSL}_2(5)\cong A_5$, recover the same abstract dihedral or polyhedral permutation pairs. Their finite-field classes are assigned uniquely by the characteristic and small-group conventions of Section~\ref{sec:wild-master}. Likewise, the nonmaximal PSL Cartan cases at $Q=7,9,11$ are counted only in the Lie-type novelty theorem and not again in the primitive Cartan cases. The unique nonempty novelty is the degree-$45$ $M_{10}$ form above.

Finally, the explicit formulas in this section refine, rather than replace, Theorem~\ref{thm:exact-support}. Each $\mathscr D(f)$ is a union of complete gcd strata for the relevant full constant-field degree. In particular, the two odd-PSL subfield descents with $r\neq p$ can have different full constant-field degrees while having the same $\mathscr D(f)$; this gives natural examples where the least period is a proper divisor of the full constant-field degree.

\section{Class counts, polynomial representatives, and minimum degrees}\label{sec:arithmetic-closure}

With the extension degrees determined in Section~\ref{sec:family-support}, we complete the arithmetic classification. We first derive exact affine class counts, then combine the finite-field class counts with the formulas at $m=1$ to determine which classes occur over $k$ and which are exceptional, and finally use branch and inertia geometry to characterize polynomial representatives and establish the sharp degree thresholds.

\subsection{Affine class counts and exact enumeration}

Retain the affine notation of Theorem~\ref{thm:affine-wild}: $Q=p^{\operatorname{ord}_n(p)}$ for $n>1$ (and $Q=p$ for $n=1$), $\rho(a)=a^Q$, and $r=\deg_\tau\ell$. With the normalized $\Phi(v)=v^q$ from Theorem~\ref{thm:affine-module-support}, put $k_\rho:=k^\rho=\mathbb F_R$, $m_\rho:=[k:k_\rho]$, $a_n:=(Q-1)/n$, $\delta_n:=(a_n,q-1)$, $\mathcal W_{q,n}:=(k^\times)^{a_n}$, and $L:=|\mathcal W_{q,n}|$.
Let $S=k[\tau;\rho]$. Its center is
$Z(S)=\mathbb F_R[\tau^{m_\rho}]$.
Here $\mathcal W_{q,n}$ is the effective scaling group; $D(v),d(v)$ and $\mathcal C_R,\mathcal P_{R,m_\rho}$ encode the corresponding central and projective fixed-point data.

Let $N_{\rm all}^{\rm aff}(q,n,r)$ and $N_{\rm exc}^{\rm aff}(q,n,r)$ denote respectively the total and exceptional numbers of $k$-M\"obius classes in the affine family with $\deg_\tau\ell=r$.

\begin{lemma}[Effective weighted scaling]\label{lem:affine-effective-scaling}
The weighted $k^\times$-action of Theorem~\ref{thm:affine-wild} factors through the surjection $k^\times\to\mathcal W_{q,n}$, $u\mapsto v=u^{a_n}$, whose kernel has order $\delta_n$. If $b_i:=(Q^r-Q^i)/(Q-1)$, then $v$ acts on a monic Ore polynomial by $c_i\mapsto c_i v^{-b_i}$. Equivalently, this is the skew-variable scaling $\tau\mapsto v\tau$, followed by monic normalization. On the center, with $y:=\tau^{m_\rho}$, put $\mathcal N(v):=N_{k/k_\rho}(v)$; the scaling sends $y$ to $\mathcal N(v)y$. Thus the induced action on roots of a monic central polynomial is $y\mapsto\mathcal N(v)^{-1}y$, and the $k^\times$-orbits are exactly the $\mathcal W_{q,n}$-orbits.
\end{lemma}

\begin{proof}
Since $d_i-D=(Q^i-Q^r)/n=-a_nb_i$, the coefficient action in Theorem~\ref{thm:affine-wild} depends on $u$ only through $v=u^{a_n}$, and its kernel has size $\gcd(a_n,q-1)=\delta_n$. In the Ore ring, the substitution $\tau\mapsto v\tau$ sends $\tau^i$ to $v^{1+Q+\cdots+Q^{i-1}}\tau^i$; division by the transformed leading coefficient gives precisely $c_i\mapsto c_i v^{-b_i}$. Finally, because $\rho$ has order $m_\rho$ on $k$, one has
\[
 (v\tau)^{m_\rho}=N_{k/k_\rho}(v)\tau^{m_\rho}=\mathcal N(v)y.
\]
The remaining two assertions are immediate.
\end{proof}

\begin{proposition}[Affine class counts over $k$]\label{prop:affine-occurrence}
For $r=1$,
\[
 N_{\rm all}^{\rm aff}(q,n,1)=\delta_n,\qquad N_{\rm exc}^{\rm aff}(q,n,1)=\delta_n-1.
\]
For every $r>1$, every monic irreducible $\ell\in k[\tau;\rho]$ gives an exceptional affine family over $k$.
\end{proposition}

\begin{proof}
For $r=1$ write $\ell=\tau+c_0$ with $c_0\in k^\times$. The weighted equivalence is $c_0'=c_0u^{-a_n}$, so there are $\delta_n$ classes. The ground-field extension-degree criterion fails exactly when $-c_0\in(k^\times)^{a_n}$, which is one weighted orbit. For $r>1$, failure of the ground-field extension-degree criterion gives $0\ne w\in V$ and $\zeta\in\mu_n$ with $\Phi(w)=\zeta^{-1}w$. Then $\mathbb F_Qw$ is a nonzero proper $\Phi$-stable $\mathbb F_Q$-subspace, contradicting Ore irreducibility. Thus $1\in\mathscr D(f)$, and Theorem~\ref{thm:no-accidental} gives exceptionality.
\end{proof}

For a nonzero $h\in S$, its \emph{minimal central left multiple}
(MCLM) is the unique monic nonzero element of least $\tau$-degree in
$Z(S)\cap Sh$; thus it is the monic central left multiple of $h$ of
least degree in $\tau$. Its monic central polynomial labels the
corresponding fiber of irreducible right divisors.

\begin{lemma}[Projective MCLM fibers and semilinear equivariance]\label{lem:affine-mclm-projective}
Let $S=k[\tau;\rho]$, with $k_\rho=\mathbb F_R$ and $m_\rho=[k:k_\rho]$. Let $\widehat H(y)\in\mathbb F_R[y]$ be monic irreducible of degree $r>1$ and put $H_c=\widehat H(\tau^{m_\rho})$. Then
\[
 S/SH_c\cong_E M_{m_\rho}(E),\qquad E=\mathbb F_{R^r}.
\]
Monic irreducible right divisors of $H_c$ are naturally parametrized by $\mathbf P^{m_\rho-1}(E)$. If a skew-ring automorphism $\psi$ preserves $SH_c$ and induces $\beta\in\operatorname{Aut}(E)$ on the center, then its action on this projective fiber is induced by an invertible $\beta$-semilinear map on $E^{m_\rho}$.
\end{lemma}

\begin{proof}
Lavrauw--Sheekey Lemma~2 identifies the irreducible central polynomial attached to an irreducible skew polynomial, while Lemma~3 gives the matrix-algebra quotient, the degree of every irreducible right divisor in a fixed central fiber, and the rank-$(m_\rho-1)$ matrix realization \cite[Lemmas~2 and~3]{LavrauwSheekey2013}. A monic irreducible right divisor gives a maximal left ideal, and under the matrix algebra isomorphism this is the annihilator of a one-dimensional kernel; conversely every one-dimensional subspace gives a maximal left ideal whose inverse image has a unique monic irreducible generator.

For semilinear equivariance, compose the induced matrix-algebra automorphism with the inverse of entrywise $\beta$. The result is an $E$-algebra automorphism of the full matrix algebra, hence is inner by Skolem--Noether. Undoing entrywise $\beta$ yields a $\beta$-semilinear implementer $T$. If a projective point is represented by a line $\Lambda\subseteq E^{m_\rho}$, then $\psi(\operatorname{Ann}(\Lambda))=\operatorname{Ann}(T\Lambda)$, which gives the asserted equivariance.
\end{proof}

For $v\in\mathcal W_{q,n}$ put $D(v):=\operatorname{ord}_{\mathbb F_R^\times}(\mathcal N(v))$. Since $v^{D(v)}$ has norm one, choose $a\in k^\times$ with $v^{D(v)}=a/\rho(a)$ and put $d(v):=[\mathbb F_R(a):\mathbb F_R]$. If $a'$ is another such choice, then $\rho(a'/a)=a'/a$, so $a'/a\in k_\rho^\times=\mathbb F_R^\times$ and hence $\mathbb F_R(a')=\mathbb F_R(a)$. Thus $d(v)$ is independent of the Hilbert--90 choice.

\begin{lemma}[Fixed projective points in one central fiber]\label{lem:affine-fixed-projective-fiber}
Assume a monic irreducible central polynomial of degree $r$ is fixed by the scaling attached to $v$. Put $D=D(v)$ and $d=d(v)$. Then $D\mid r$, and the number of fixed projective points in its MCLM fiber is
\[
 \mathcal P_{R,m_\rho}(r;D,d)=
 \begin{cases}
 \displaystyle d\frac{R^{(r/D)(m_\rho/d)}-1}{R^{r/D}-1},&d\mid r/D,\\[3mm]
 0,&d\nmid r/D.
 \end{cases}
\]
\end{lemma}

\begin{proof}
Write $m=m_\rho$, put $c=\mathcal N(v)$, and let $\psi_v$ denote the skew-variable scaling $\tau\mapsto v\tau$. Since the central polynomial is fixed by this scaling, $\psi_v$ induces an automorphism $\beta$ of $E=\mathbb F_{R^r}=\mathbb F_R[y]/(\widehat H)$ with $\beta(y)=cy$. Because $r>1$, the image of $y$ in $E$ is nonzero. Hence $\beta$ has order exactly $D=\operatorname{ord}(c)$. It follows that $D\mid r$ and $E^\beta=E_0:=\mathbb F_{R^{r/D}}$.

Choose $a\in k^\times$ as above, so that $v^D=a/\rho(a)$, and put $d=[\mathbb F_R(a):\mathbb F_R]$. With the convention $\tau b=\rho(b)\tau$, one has
\[
\begin{aligned}
 \psi_v^D(\tau)&=v^D\tau=a\tau a^{-1},
 \qquad \psi_v^D=\operatorname{Int}(a),\\
 T^D&=\xi A_a\qquad(\xi\in E^\times),
\end{aligned}
\]
where $T$ is the invertible $\beta$-semilinear implementer from Lemma~\ref{lem:affine-mclm-projective} and $A_a$ is the image of $a$ in $M_m(E)$.

After scalar extension, the eigenvalues of $A_a$ are the $d$ Frobenius conjugates $a,\rho(a),\ldots,\rho^{d-1}(a)$. When these eigenvalues lie in $E$, their eigenspaces have equal dimension $m/d$: the image of $\tau$, which is invertible because $y=\tau^m$ has nonzero image in $E$, cyclically permutes the eigenspaces through $\tau a\tau^{-1}=\rho(a)$.

Suppose that an $E$-line $L$ is fixed projectively by $T$. Then $T^D(L)=L$, so $A_a(L)=L$ and $L$ is an eigenline of $A_a$, say with eigenvalue $\lambda\in E$. Since $\psi_v$ fixes the coefficient field $k$, one has $\psi_v(A_a)=A_a$, equivalently $TA_a=A_aT$ as semilinear operators. Hence $T(V_\lambda)=V_{\beta(\lambda)}$. Because $T(L)=L\subseteq V_\lambda$, one has $\beta(\lambda)=\lambda$, so $\lambda\in E_0$. The eigenvalue $\lambda$ is a Frobenius conjugate of $a$, so $[\mathbb F_R(\lambda):\mathbb F_R]=d$. Therefore $d\mid[E_0:\mathbb F_R]=r/D$, proving that there are no fixed projective points if $d\nmid r/D$.

Assume now that $d\mid r/D$. Then all $d$ eigenvalues of $A_a$ lie in $E_0$, and the corresponding $T$-stable eigenspaces have $E$-dimension $m/d$. From $T^D=\xi A_a$ and commutation with $T$ one gets $\xi\in E_0^\times$; on an eigenspace one has $T^D=(\xi\lambda)I$ with $\xi\lambda\in E_0^\times$. Rescale $T$ using surjectivity of $N_{E/E_0}$ so that $T^D=1$. Nonabelian Hilbert~90 then conjugates $T$ to coordinatewise $\beta$-Frobenius. Its fixed projective lines are precisely the points of $\mathbf P^{m/d-1}(E_0)$, whose number is
\[
 \#\mathbf P^{m/d-1}(E_0)
 =\frac{R^{(r/D)(m/d)}-1}{R^{r/D}-1}.
\]
There are $d$ such eigenspaces, so multiplying this number by $d$ gives the formula in the statement.
\end{proof}

Define
\[
 \mathcal C_R(r,D):=
 \begin{cases}
 \displaystyle \frac1r\sum_{c\mid r}\mu(c)R^{r/c},&D=1,\\[3mm]
 \displaystyle \frac{\phi(D)}r\sum_{\substack{c\mid r/D\\(c,D)=1}}\mu(c)\bigl(R^{r/(Dc)}-1\bigr),&D>1,\ D\mid r,\\[4mm]
 0,&D>1,\ D\nmid r.
 \end{cases}
\]
For $r=D=2$ this gives the direct value $(R-1)/2$.

For $r=\deg_\tau\ell$, write $\operatorname{Cl}^{\rm aff}_r(q,n)$ for the number of $k$-M\"obius classes of $k$-indecomposable affine families.

\begin{theorem}[Exact enumeration of affine $k$-forms]\label{thm:affine-exact-enumeration}
For every $r>1$,
\begin{equation}\label{eq:affine-burnside-count}
 \operatorname{Cl}^{\rm aff}_r(q,n)=\frac1L\sum_{v\in\mathcal W_{q,n}}\mathcal C_R(r,D(v))\mathcal P_{R,m_\rho}(r;D(v),d(v)).
\end{equation}
Every class counted in \eqref{eq:affine-burnside-count} is exceptional.
\end{theorem}

\begin{proof}
Lemma~\ref{lem:affine-effective-scaling} and Theorem~\ref{thm:affine-wild} reduce equivalence to effective scaling $v\in\mathcal W_{q,n}$. By Lemma~\ref{lem:affine-mclm-projective}, each fixed central irreducible has a projective MCLM fiber; its number is $\mathcal C_R(r,D(v))$, using \cite[Theorem~1.4]{Reis2020Invariant} for $D(v)>1,r>2$ and the stated direct boundary counts. Lemma~\ref{lem:affine-fixed-projective-fiber} counts fixed points in that fiber. The equivariance in Lemma~\ref{lem:affine-mclm-projective} and Burnside give \eqref{eq:affine-burnside-count}; Proposition~\ref{prop:affine-occurrence} gives exceptionality.
\end{proof}

\subsection{Class counts and exceptional subcounts over \texorpdfstring{$k$}{k}}

For a classification case, let $N_{\rm all}$ be the number of separable $k$-indecomposable $k$-M\"obius classes in that case and let $N_{\rm exc}$ be the exceptional subcount. Every case below is understood only in the structural and maximality range of the corresponding classification theorem, with accidental isomorphisms counted according to Section~\ref{sec:wild-master}. The total class numbers are those of the classification theorems; we record the exceptional subcounts, repeating $N_{\rm all}$ where a case has several forms.

Write $Q=p^\alpha$, $k=\mathbb F_{p^s}$, $h_0=(\alpha,s)$, and $d_F=\alpha/h_0$. Here $\mathbf1[P]$ denotes the indicator of the condition $P$.

\begin{theorem}[Class counts over $k$ and exact enumeration]\label{thm:master-occurrence}
The total class counts and exceptional subcounts in the complete classification are as follows. Among the geometrically primitive Lie-type cases, the odd-PSL Borel, odd-PSL split- and nonsplit-Cartan, odd-PSL odd-exponent subfield with $H=\operatorname{PSL}_2(Q_0)$, and the maximal $A_4$-stabilizer case have $N_{\rm all}=2$; every other geometrically primitive Lie-type case has $N_{\rm all}=1$.
\begin{enumerate}[label=(\arabic*),ref=\arabic*,font=\itshape]
\item\label{item:master-occurrence-1} Cyclic prime degree: $N_{\rm all}=2$ and $N_{\rm exc}=\mathbf1[\ell\nmid q-1]+\mathbf1[\ell\nmid q+1]$.
\item Tame dihedral prime degree: $N_{\rm all}=2$ and $N_{\rm exc}=2\mathbf1[\ell\nmid q^2-1]$.
\item\label{item:master-occurrence-3} The $V_4/A_4$ quartic has $N_{\rm all}=N_{\rm exc}=1$. The primitive tame $(A_4,C_3)$ case has $N_{\rm all}=2$ and $N_{\rm exc}=0$; the $(S_4,S_3)$ case and each primitive tame $A_5$ case have $N_{\rm all}=1$ and $N_{\rm exc}=0$.
\item\label{item:master-occurrence-4} In the affine family, Proposition~\ref{prop:affine-occurrence} gives the case $r=1$ and, for $r>1$, Theorem~\ref{thm:affine-exact-enumeration} gives $N_{\rm all}=N_{\rm exc}=\operatorname{Cl}^{\rm aff}_r(q,n)$.
\item\label{item:master-occurrence-5} In characteristic $2$, the wild dihedral case has $N_{\rm all}=1$ and $N_{\rm exc}=\mathbf1[\ell\nmid q^2-1]$. Each characteristic-$3$ $A_5$ case of degree $5,6$, or $10$ has $N_{\rm all}=1$ and $N_{\rm exc}=0$.
\item\label{item:master-occurrence-6} Every Borel case, every even split-Cartan case, every odd full-PGL Cartan case, every Lie case with stabilizer $A_4,S_4$, or $A_5$, and every odd square-subfield PSL case has $N_{\rm exc}=0$.
\item The even nonsplit-Cartan case has $N_{\rm exc}=\mathbf1[\alpha>1\text{ odd and }(\alpha,s)=1]$.
\item For a geometrically primitive odd-PSL split-Cartan case, $N_{\rm all}=2$ and $N_{\rm exc}=\mathbf1[\alpha\text{ even and }d_F\text{ even}]$, with the field-diagonal form exceptional when this is $1$.
\item\label{item:master-occurrence-9} For a geometrically primitive odd-PSL nonsplit-Cartan case, $N_{\rm all}=2$ and $N_{\rm exc}=\mathbf1[p=3,\ \alpha>1\text{ odd},\ (\alpha,s)=1]$, again with the field-diagonal form exceptional.
\item\label{item:master-occurrence-10} For a prime exponent $r_0\ne p$, a PGL subfield case has $N_{\rm all}=1$, while an odd-PSL subfield pair with $H=\operatorname{PSL}_2(Q_0)$ has $N_{\rm all}=2$. If $r_0\mid d_F$, $r_0\nmid p^{2h_0}-1$, and additionally $r_0>3$ for $p=2$, respectively $r_0\nmid p(p^2-1)$ for odd $p$, then their exceptional subcounts are $1$ and $2$, respectively; otherwise both exceptional subcounts are $0$.
\item\label{item:master-occurrence-11} If $p$ is odd and $r_0=p$, the PGL subfield case has $N_{\rm all}=1$ and $N_{\rm exc}=\mathbf1[p\mid d_F]$, while the PSL pair has $N_{\rm all}=2$ and $N_{\rm exc}=\mathbf1[d_F\text{ odd and }p\mid d_F]$. The characteristic-two exponent-$2$ case has $N_{\rm all}=1$ and $N_{\rm exc}=0$ by Proposition~\ref{prop:even-square-subfield-empty-support}.
\item\label{item:master-occurrence-12} Every Lie-type novelty case has $N_{\rm all}=1$ over the finite fields where it occurs. The unique exceptional class is the degree-$45$ $M_{10}$ form with $(G,H)=(\operatorname{PSL}_2(9),D_8)$ over $\mathbb F_{3^s}$ for odd $s$.
\end{enumerate}
\end{theorem}

\begin{proof}
Evaluate the extension-degree formulas at $m=1$. Items~(\ref{item:master-occurrence-1})--(\ref{item:master-occurrence-3}),(\ref{item:master-occurrence-5}) follow from Proposition~\ref{prop:closed-nonlie-support}; (\ref{item:master-occurrence-4}) from Proposition~\ref{prop:affine-occurrence} and Theorem~\ref{thm:affine-exact-enumeration}; (\ref{item:master-occurrence-6})--(\ref{item:master-occurrence-9}) from Proposition~\ref{prop:lie-cartans-support}; (\ref{item:master-occurrence-10}) from Proposition~\ref{prop:lie-subfield-support}; (\ref{item:master-occurrence-11}) from Proposition~\ref{prop:lie-novelty-support}, with the characteristic-$2$ square-subfield case from Proposition~\ref{prop:even-square-subfield-empty-support}; and (\ref{item:master-occurrence-12}) from Proposition~\ref{prop:lie-novelty-support}.
\end{proof}

\subsection{Branch arithmetic and polynomial representatives}

\begin{theorem}[Criterion for polynomial representatives]
\label{thm:polynomializability}
Let $\mathcal D=(\kappa/k,C,G,H,\varphi)$ be a semilinear $\mathbf P^1$-datum over $k$ with $[G:H]>1$. Then $\mathcal M_{\mathcal D}$ contains a polynomial representative over $k$ if and only if there is an orbit $[P]_G\in\mathcal B_G$ and an element $g\in G$ such that $[\varphi(P)]_G=[P]_G$ and $G=H(gG_Pg^{-1})$. Equivalently, a conjugate of an inertia subgroup above a $k$-rational branch point is transitive on $G/H$.
\end{theorem}

\begin{proof}
A rational function is $k$-M\"obius equivalent to a polynomial exactly when a $k$-rational target point has a unique geometric preimage with ramification index equal to the degree. By Theorem~\ref{thm:branch-orbit}, if $P\in\mathcal R_G$ and $I:=G_P$ is its inertia group for $C\to C/G$, then the points in the corresponding $\pi_{\mathcal D}$-fiber are parametrized by $H\backslash G/I$. This set has one element exactly when $G=HgI$, and then \eqref{eq:double-coset-passport} gives total ramification. Frobenius-fixity of $[P]_G$ is exactly rationality of the target branch point; the unique source point is then rational as well.
\end{proof}

\begin{proposition}[Projective transitivity when $H\cong A_4,S_4$, or $A_5$]
\label{prop:exceptional-stabilizer-transitivity}
Among the stabilizers $A_4,S_4,A_5$ in
Proposition~\ref{prop:exceptional-wild}, transitivity on the natural projective
line occurs exactly for
\[
\begin{array}{c|c}
(G,H)&\text{parameters}\\ \hline
(\operatorname{PSL}_2(p),A_4)&p=5\\
(\operatorname{PSL}_2(p),S_4)&p=7,23\\
(\operatorname{PSL}_2(p),A_5)&p=11,19,29,59\\
(\operatorname{PSL}_2(p^2),A_5)&p=3\\
(\operatorname{PGL}_2(p),S_4)&p=5,11.
\end{array}
\]
\end{proposition}

\begin{proof}
Transitivity requires $Q+1\mid |H|$. Intersecting this divisibility condition
with the structural congruence ranges in Proposition~\ref{prop:exceptional-wild}
leaves exactly the displayed parameters.

A nonidentity element of $\operatorname{PGL}_2(p)$ whose order is
prime to $p$ has two distinct geometric fixed points. It has two fixed
points on $\mathbf P^1(\mathbb F_p)$ when this pair is split over
$\mathbb F_p$, and none when the pair is nonsplit. A nonidentity element
of order $p$ has exactly one fixed point on $\mathbf P^1(\mathbb F_p)$. For $A_4$ at $p=5$, Burnside gives
\(\frac{6+3\cdot2}{12}=1.\)
For $S_4$ at $p=7$,
\(\frac{8+8\cdot2}{24}=1,\)
while at $p=23$ every nonidentity element of the embedded $S_4$ is nonsplit,
so the orbit count is $24/24=1$.

For $A_5$, using the $15$, $20$, and $24$ elements of orders $2$, $3$, and
$5$, respectively, the Burnside numerators are
\[
\begin{array}{c|c}
p&\text{numerator}\\ \hline
11&12+24\cdot2\\
19&20+20\cdot2\\
29&30+15\cdot2\\
59&60,
\end{array}
\]
all equal to $60$. For $\operatorname{PSL}_2(9)$ with $H=A_5$, the involutions are split,
the order-$3$ elements each have one rational fixed point, and the
order-$5$ elements have nonsplit fixed-point pairs, so
\(\frac{10+15\cdot2+20}{60}=1.\)

When two $G$-conjugacy classes of one of these stabilizers occur, they form a single $\operatorname{PGL}_2(Q)$-conjugacy class and hence have the same orbit structure on $\mathbf P^1(\mathbb F_Q)$. Finally consider $S_4\leqslant\operatorname{PGL}_2(p)$. At $p=5$, the identity and the six split
order-$4$ elements contribute $18$. If $x$ of the nine involutions are
split, Burnside gives
\(\frac{18+2x}{24}\in\mathbb Z.\)
Since $0\leqslant x\leqslant9$, necessarily $x=3$, and the orbit count is one. At
$p=11$, the order-$3$ and order-$4$ elements are nonsplit; hence
\(\frac{12+2x}{24}\in\mathbb Z\)
forces $x=6$, again giving one orbit.
\end{proof}

For a defining-characteristic Lie pair, the two inertia groups in the Galois quotient $C\to C/G$ are the
Borel $B$ and a nonsplit torus $T_{\rm ns}$. Hence
Theorem~\ref{thm:polynomializability} reduces existence of polynomial representatives to
\[
 G=HgB
 \qquad\text{or}\qquad
 G=HgT_{\rm ns}.
\]
The PGL/even Borel satisfies the second factorization and has a polynomial representative. The odd-PSL Borel is not: its tame inertia has order $(Q+1)/2<[G:B]=Q+1$, and $B$ is not transitive on the projective line.

A split-Cartan normalizer is not transitive and its degree exceeds the tame inertia order. The PGL/even nonsplit torus acts freely and transitively on $\mathbf P^1(\mathbb F_Q)$. For odd $\operatorname{PSL}_2(Q)$ the nonsplit-Cartan normalizer, of order $Q+1$, acts freely and transitively on $\mathbf P^1(\mathbb F_Q)$ exactly when its involutions have no rational fixed points, equivalently when $Q\equiv3\pmod4$. No proper subfield stabilizer is transitive: for prime exponent at least $3$ its order is too small, while in the square-subfield case $Q_0^2+1\nmid Q_0(Q_0^2-1)$.
The tame-inertia factorization is likewise excluded by index. For $H\cong A_4,S_4$, or $A_5$, the Borel factorization occurs exactly in the cases of Proposition~\ref{prop:exceptional-stabilizer-transitivity}. The nonsplit-torus factorization adds no further cases: in PSL, $[G:H]>|T_{\rm ns}|$; in PGL-$S_4$, that factorization would force $p(p-1)\mid24$.

\begin{corollary}[Polynomial representatives across the classification]
\label{cor:complete-polynomializability}
For the non-Lie and basic wild cases, existence of polynomial representatives is given by
\[
\begin{array}{c|c}
\text{classification case}&\text{polynomial representative}\\ \hline
\text{cyclic split}&\text{yes}\\
\text{cyclic nonsplit R\'edei}&\text{no}\\
\text{tame dihedral}&\text{yes, both forms}\\
V_4/A_4&\text{no}\\
(A_4,C_3)&\text{no}\\
(S_4,S_3)&\text{yes}\\
(A_5,A_4),\ \deg5&\text{yes}\\
(A_5,D_{10}),\ \deg6&\text{no}\\
(A_5,S_3),\ \deg10&\text{no}\\
\text{all affine families}&\text{yes}\\
\text{characteristic-$2$ dihedral}&\text{yes}\\
\text{characteristic-$3$ $A_5$, degree $5$}&\text{yes}\\
\text{characteristic-$3$ $A_5$, degrees $6,10$}&\text{no}.
\end{array}
\]
For the defining-characteristic Lie-type cases, the PGL/even Borel and
PGL/even nonsplit-Cartan cases have polynomial representatives, the odd-PSL Borel, all geometrically primitive split-Cartan cases, and all
subfield cases have no polynomial representative, and a geometrically primitive odd-PSL nonsplit-Cartan case has
a polynomial representative exactly when $Q\equiv3\pmod4$. The cases with $H\cong A_4,S_4$, or $A_5$
are exactly those of
Proposition~\ref{prop:exceptional-stabilizer-transitivity}. Among the geometrically decomposable small
Cartan novelty cases only the $Q=7$, degree-$21$ case has a polynomial representative; in
particular the $Q=9$, degree-$36$ case is not. In the infinite
$A_4/S_4$ novelty family, existence of polynomial representatives occurs exactly for $p=11$.

Among these classes, the exceptional classes having a polynomial representative over $k$ are precisely:
\begin{enumerate}[label=(\arabic*),ref=\arabic*,font=\itshape]
\item a split cyclic class with $\ell\nmid q-1$;
\item a tame or characteristic-$2$ dihedral class with $\ell\nmid q^2-1$;
\item an exceptional affine Ore class;
\item an exceptional defining-characteristic nonsplit-Cartan class.
\end{enumerate}
In particular, exceptional nonsplit R\'edei classes, the $V_4/A_4$ quartic, exceptional odd-$\operatorname{PSL}_2$ split-Cartan diagonal forms, all exceptional subfield forms, and the degree-$45$ $M_{10}$ novelty have no polynomial representative.
\end{corollary}

\begin{proof}
For the non-Lie-type cases, this follows from the explicit ramification partitions and
representatives in Sections~\ref{sec:tame-basic}--\ref{sec:wild-master} together with
Theorem~\ref{thm:polynomializability}. The Lie assertions are the
factorization analysis above and
Proposition~\ref{prop:exceptional-stabilizer-transitivity}. The novelty
statements follow from the corresponding Cartan or $A_4$ stabilizer geometry. Intersecting this list with the ground-field extension-degree formulas in Section~\ref{sec:family-support} gives the exceptional sublist.
\end{proof}

\subsection{Sharp minima and prime-degree realization}

Let $d_{\rm exc}(q)$ be the least degree of a separable $k$-indecomposable exceptional class with Galois closure of genus zero, $d_{\rm wild}(q)$ the corresponding wild minimum, $d_{\rm ar}(q)$ the minimum among exceptional $k$-indecomposable maps that are geometrically decomposable, and $d_{\rm aff,2}(q)$ the minimum among exceptional affine classes with $r=\deg_\tau\ell\geqslant2$. Define $\lambda(q):=\min\{\ell:\ell\text{ is an odd prime and }\ell\nmid q-1\}$.
Let $d_{\rm poly}(q)$ be the least degree of an exceptional $k$-indecomposable class containing a polynomial representative over $k$, and let $d_{\rm rat}(q)$ be the least degree of an exceptional $k$-indecomposable class containing no polynomial representative over $k$.

\begin{theorem}[Sharp minimum degrees]\label{thm:sharp-minima}
For $q=p^s$,
\[
\begin{alignedat}{2}
 d_{\rm exc}(q)&=3,\qquad&
 d_{\rm wild}(q)&=
 \begin{cases}
 4,&p=2,\\
 p,&p\text{ odd},
 \end{cases}\\
 d_{\rm ar}(q)&=4,\qquad&
 d_{\rm aff,2}(q)&=p^2.
\end{alignedat}
\]
Moreover,
\[
\begin{aligned}
 d_{\rm poly}(q)&=\min\{\lambda(q),p^2\},\\
 d_{\rm poly}(2^s)&=
 \begin{cases}
 3,&s\text{ odd},\\
 4,&s\text{ even}.
 \end{cases}
\end{aligned}
\]
For odd $p$, the first formula simplifies to $d_{\rm poly}(q)=\lambda(q)$. Finally,
\[
 d_{\rm rat}(q)=
 \begin{cases}
 5,&q=2^s\text{ with }s\text{ odd},\\
 3,&q=2^s\text{ with }s\text{ even},\\
 3,&p\text{ odd},\ p\ne3,\text{ and }q\equiv1\pmod3,\\
 4,&p\text{ odd and the preceding condition does not hold}.
 \end{cases}
\]
\end{theorem}

\begin{proof}
In degree two the arithmetic and geometric permutation groups are both $C_2$, and the generating coset contains the identity, fixing two sheets; hence the function is not exceptional. Degree three is attained by an additive class for $p=3$ and by one of the two cyclic forms otherwise. For a wild degree $d$, the faithful embedding $G\leqslant S_d$ and $p\mid |G|$ imply $d\geqslant p$. The additive degree-$p$ classes attain this bound for odd $p$. In characteristic two, a wild transitive subgroup of $S_3$ must be $S_3=D_6$; its dihedral cubic is not exceptional because $q^2\equiv1\pmod3$. A pure-additive rank-two class gives degree four instead.

Degrees two and three cannot be geometrically decomposable because they are prime. The $V_4/A_4$ quartic for odd $q$, and the pure-additive $n=1,r=2$ quartic for even $q$, therefore give $d_{\rm ar}=4$. Affine rank at least two has degree at least $p^2$. For existence at degree $p^2$, use $k[\tau;v\mapsto v^p]$ and $\ell=\tau^2-\tau+b$. A right linear factor exists exactly when $v^{p+1}-v+b=0$ for some $v\in k$. The function $v\mapsto v^{p+1}-v$ sends both $0$ and $1$ to $0$, so it is not surjective on the finite set $k$. Choosing $-b$ outside its image makes $\ell$ irreducible and attains $d_{\rm aff,2}=p^2$.

For polynomial representatives, Corollary~\ref{cor:complete-polynomializability} leaves only split cyclic, dihedral, affine, and defining-characteristic nonsplit-Cartan cases. Exceptional cyclic or dihedral prime degrees are at least $\lambda(q)$. For odd $p$, one has $\lambda(q)\leqslant p$, and every affine degree is at least $p$; the exceptional nonsplit-Cartan polynomial degrees are larger. If $\lambda(q)<p$, the separable polynomial $X^{\lambda(q)}$ attains the minimum. If $\lambda(q)=p$, use an exceptional additive degree-$p$ class, not the inseparable polynomial $X^p$. In characteristic two, no affine quadratic is exceptional, a rank-two affine quartic is exceptional, and the exceptional nonsplit-Cartan degrees are larger than four. The split cyclic cubic is exceptional exactly when $s$ is odd. This proves $d_{\rm poly}=\min\{\lambda(q),p^2\}$ and its stated characteristic-two specialization.

For $d_{\rm rat}$, if $q=2^s$ with $s$ even, then $q\equiv1\pmod3$ and the nonsplit cyclic cubic is exceptional and has no polynomial representative. If $s$ is odd, the exceptional cubic is the split polynomial class, every exceptional quartic is affine and has a polynomial representative, and the nonsplit cyclic quintic is exceptional because $2^s\not\equiv-1\pmod5$; hence the minimum is $5$. For odd $q$, an exceptional cubic with no polynomial representative occurs exactly when $p\ne3$ and $q\equiv1\pmod3$, again as the nonsplit cyclic class. In all remaining odd cases there is no such cubic, while the $V_4/A_4$ quartic is always exceptional and never has a polynomial representative.
\end{proof}

\begin{corollary}[Every odd prime degree occurs]
For every finite field $\mathbb F_q$, the set of prime degrees of separable $k$-indecomposable exceptional rational functions with Galois closure of genus zero is exactly the set of odd primes.
\end{corollary}

\begin{proof}
A separable exceptional rational function cannot have degree $2$: its arithmetic and geometric monodromy on two sheets are both $C_2$, whose identity fixes two sheets. Let $\ell$ be odd. If $\ell\ne p$, the cyclic classification gives both $k$-forms with no congruence condition; they fail exceptionality respectively for $q\equiv1$ and $q\equiv-1\pmod\ell$, which cannot occur simultaneously. If $\ell=p$, then $p$ is odd and Proposition~\ref{prop:affine-occurrence} gives $p-2\geqslant1$ exceptional additive classes with $r=1$.
\end{proof}

\section{Known exceptional families: recovery and comparison}
\label{sec:historical-recovery}

We compare the classification with several known exceptional families. When the only input is a historical formula or a published genus and geometric pair, the present results sometimes yield an independent reproof and sometimes additional extension-degree or decomposition information. When an earlier paper already proves the same numerical criterion, we claim only a different proof or interpretation; mere normal-form compatibility is stated as such.

\begingroup
\footnotesize
\renewcommand{\arraystretch}{1.05}
\setlength{\tabcolsep}{2.2pt}
\begin{longtable}{|>{\raggedright\arraybackslash}p{0.15\textwidth}|>{\raggedright\arraybackslash}p{0.20\textwidth}|>{\raggedright\arraybackslash}p{0.21\textwidth}|>{\raggedright\arraybackslash}p{0.27\textwidth}|}
\hline
Family & Historical input used here & Argument used here & Conclusion obtained here \\
\hline
\endfirsthead
\hline
Family & Historical input used here & Argument used here & Conclusion obtained here \\
\hline
\endhead
Ding--Zieve quartic & displayed formula and irreducible cubic denominator & $V_4/A_4$ identification and non-Lie extension-degree formula & independent reproof and recovery of the unique odd-characteristic exceptional quartic class; semilinear interpretation of decomposition and extension degrees \\
\hline
Linearized/\newline sublinearized & associated identity and decomposition correspondence & direct description of the full Galois closure and affine/Ore classification & rational full Galois closure and classification placement \\
\hline
M\"oller affine examples & Propositions 16--17 explicit polynomials and parameter conditions & literal affine identification, scalar-order collapse, Frobenius operator of order $p^r-1$ & independent recovery of exceptionality and monodromy; exact arithmetic group, full constants, and extension arithmetic for every permitted scalar order \\
\hline
$\operatorname{PSL}_2/\operatorname{PGL}_2$ polynomial case & published genus, geometric pair, and formula & nonsplit-Cartan descent-class identification and direct extension-degree criterion & uniform independent reproof of the historical extension criteria \\
\hline
Ding--Zieve wild family & displayed formula and Dickson identity & split-Cartan identification, direct extension-degree criterion, ramification and decomposition results & independent reproof; sharp iff criterion, exact extension degrees, and $Q=9$ decomposition behavior over every finite extension \\
\hline
Classical and cubic forms & published normal forms and cubic classification & cyclic, dihedral, and affine extension-degree theory & compatibility and a uniform extension-degree derivation \\
\hline
\end{longtable}
\endgroup

\subsection{Low-degree recovery: the Ding--Zieve quartic}

\begin{corollary}[The Ding--Zieve exceptional quartic from the $V_4/A_4$ case]

Assume that $q$ is odd and that $h(X)=X^3+\alpha X+\beta\in k[X]$ is irreducible. For the formula in \cite[Theorem~1.4\textup{(1)}]{DingZieve2022},
\[
 f^{\mathrm{DZ}}_{\alpha,\beta}(X)
 =\frac{X^4-2\alpha X^2-8\beta X+\alpha^2}{X^3+\alpha X+\beta},
\]
one has
\[
\begin{aligned}
 f^{\mathrm{DZ}}_{\alpha,\beta}&=4L_h,\\
 (A,G,U,H)&\cong(A_4,V_4,C_3,1),\\
 \mathscr D(f^{\mathrm{DZ}}_{\alpha,\beta})&=\{m\geqslant1:3\nmid m\}.
\end{aligned}
\]
Hence these conclusions follow from the results of this paper without the monodromy or permutation proof of \cite{DingZieve2022}.
The rational function is $k$-indecomposable but geometrically decomposable and represents the unique $V_4/A_4$ quartic class. Its set of permutation extension degrees has least period $3$ and natural density $2/3$.
\end{corollary}

\begin{proof}
Comparing the formula in Proposition~\ref{prop:V4-quartic} with the displayed Ding--Zieve formula gives $f^{\mathrm{DZ}}_{\alpha,\beta}=4L_h$ immediately. The same proposition gives the monodromy,
uniqueness, and decomposition assertions, while
Proposition~\ref{prop:closed-nonlie-support} gives the set of extension degrees, period, and density.
\end{proof}

Ding--Zieve's degree-$4$ theorem also contains an even-characteristic additive
case and nonexceptional small-field sporadic cases. We do not reprove that
classification; the corollary isolates the odd exceptional formula and derives its
properties from the present structural theory.

\subsection{The affine construction and exact arithmetic for M\"oller's examples}

Let $R=p^s$ and
\[
 L(Z)=\sum_{i=0}^r a_iZ^{R^i},\qquad
 S(X)=X\left(\sum_{i=0}^r a_iX^{(R^i-1)/n}\right)^n,
\]
so that $L(Z)^n=S(Z^n)$.
Here $n\mid R-1$ and $a_0a_r\ne0$. This is the classical construction by associated linearized and sublinearized
polynomials; see \cite{CoulterHavasHenderson2004}.
For complementary results on exact value distributions of explicit affine
polynomials with rational Galois closure, see Kumallagov
\cite{Kumallagov2026}.

\begin{proposition}[Direct rational full closure for the sublinearized construction]

Assume $\deg S>1$, and after target scaling assume $a_r=1$. Put $e=1$ for $n=1$ and
$e=\operatorname{ord}_n(p)$ for $n>1$, let $Q=p^e$, $h=s/e$, and
$\rho(c)=c^Q$, and set $\ell(\tau):=\sum_{i=0}^r a_i\tau^{hi}\in k[\tau;\rho]$. Then $\mathcal K_\ell=L$ and $F_{\ell,n}=S$. If $V=\ker L$, the full
geometric Galois closure of $S$ is $\mathbf P^1_z$, with $(G,H)=(V\rtimes\mu_n,\mu_n)$ (and $H=1$ for $n=1$). Consequently, every separable $k$-indecomposable exceptional member of this
classical construction is a polynomial representative of the affine family of
Theorem~\ref{thm:affine-wild}.
\end{proposition}

\begin{proof}
Since $R=Q^h$, the identities $\mathcal K_\ell=L$ and $F_{\ell,n}=S$ are
termwise. Let $V$ act on $\bar k(z)$ by translations and $\mu_n$ by
$z\mapsto\zeta z$. With $x:=z^n$ and $t:=L(z)^n=S(x)$, one has $\bar k(z)^{\mu_n}=\bar k(x)$ and $t$ is fixed by
$V\rtimes\mu_n$. Since $L$ is separable,
\[
 [\bar k(z):\bar k(t)]=n|V|=|V\rtimes\mu_n|,
\]
so $\bar k(z)^{V\rtimes\mu_n}=\bar k(t)$. The affine action of
$V\rtimes\mu_n$ on $V$ is faithful and has point stabilizer $\mu_n$; hence
$\operatorname{core}_{V\rtimes\mu_n}(\mu_n)=1$. Therefore $\bar k(z)$ is
the full geometric Galois closure of $\bar k(x)/\bar k(t)$. The final assertion
follows from Theorem~\ref{thm:affine-wild} and the affine extension-degree theory.
\end{proof}

This decomposition correspondence is compatible with
\cite{CoulterHavasHenderson2004}: their complete algorithm first passes to the
least exponent $e$ with $n\mid p^e-1$, namely $e=\operatorname{ord}_n(p)$, and
then uses the corresponding skew-polynomial ring. The additional conclusions are the model of the full Galois closure, its placement in the complete genus-zero classification, and the extension arithmetic supplied by the Frobenius module.

\begin{corollary}[Exact extension arithmetic for M\"oller's affine examples]

Let $\mu(T):=\sum_{i=0}^r m_iT^i\in\mathbb F_p[T]$ be the minimal polynomial of a primitive $N=(p^r-1)$-st root. Let $M$ be the
additive polynomial of \cite[Proposition~16]{Moller2012}, and for $n>1$ let
$f_n$ be a sublinearized polynomial allowed by \cite[Proposition~17]{Moller2012}.
Then necessarily $n\mid p-1$, and moreover
\[
\begin{aligned}
 G_{f_n}&=V\rtimes\mu_n,\\
 A_{f_n}&=V\rtimes\langle\Phi\rangle\cong\operatorname{AGL}_1(p^r),\\
 [\kappa_{f_n}:\mathbb F_p]&=\frac{N}{n},\\
 \mathscr D(f_n)&=\left\{j\geqslant1:\frac{N}{n}\nmid j\right\},
\end{aligned}
\]
where $V=\ker M$ and $\Phi(v)=v^p$.
Thus the least period is $N/n$ and the natural density is $1-n/N$. For the additive
family $M$, the same formulas hold with $n=1$.
\end{corollary}

\begin{proof}
Over $k=\mathbb F_p$ the coefficient automorphism is trivial. Taking
$\ell(\tau)=\mu(\tau)$ in the affine construction gives
$F_{\ell,1}=M$ and $F_{\ell,n}=f_n$ literally.

For $n>1$, let $e=\operatorname{ord}_n(p)$. Since $m_r=1$, M\"oller's divisibility
condition gives $e\mid r$, and it gives $e\mid i$ whenever $m_i\ne0$, hence
$\mu(T)=\nu(T^e)$. If $p\mid e$ then $\mu'=0$, impossible. If $e>1$, take a
primitive $e$-th root $\zeta$ and a root $\lambda$ of $\mu$. From
$\mu(\zeta T)=\mu(T)$ we get $\zeta\lambda=\lambda^{p^j}$ for some
$0<j<r$. With $d=(r,j)$,
\[
 e=\frac{p^r-1}{p^d-1}\geqslant p^{r/2}+1>r,
\]
contrary to $e\mid r$. Thus $e=1$ and $n\mid p-1$.

The minimal polynomial of $\Phi$ on $V$ is $\mu$, so $\Phi$ has order $N=p^r-1$. Its centralizer is $\langle\Phi\rangle$, and since
$\mu_n\leqslant\mathbb F_p^\times$, one has $\mu_n\leqslant\langle\Phi\rangle$. Theorem~\ref{thm:affine-module-support}
therefore gives the displayed arithmetic group and constant-field degree.
For extension degree $j$, its fixed-vector criterion fails exactly when
$1-\zeta\Phi^j$ has nonzero kernel for some $\zeta\in\mu_n$, equivalently when
$\lambda^j\in\mu_n$, i.e., when $(N/n)\mid j$. For $M$ one takes
$\mu_1=1$, giving the same conclusion with $n=1$.
\end{proof}

The proof uses M\"oller's explicit constructions and their parameter conditions, not his proofs of exceptionality or monodromy. Those conditions include $N>1$ in the additive case and $1<n<N$ in the sublinearized case, so $N/n>1$. The affine theory above independently recovers exceptionality and monodromy, and determines the exact arithmetic group, full constant-field degree, and complete extension arithmetic. The fixed-field rationality used by M\"oller for the affine kernel acting freely and transitively on the sheets does
not by itself imply rationality of the full normal closure; the proposition above
proves that stronger statement for the additive/sublinearized branch. No claim is
made here for the genuinely dihedral-stabilizer branch or for all historical
``affine type'' examples.

\subsection{Genus-zero specializations of Lie-type polynomial families}

Guralnick--Zieve \cite{GuralnickZieve2010} obtain in their one-wild-branch
family a full Galois-closure model
\[
 v^Q-v=w^n,
\]
of genus $(Q-1)(n-1)/2$. Thus its intersection with the present genus-zero
classification is exactly $n=1$.

\begin{corollary}[A uniform Cartan reproof for the genus-zero polynomial case]

For $Q=2^\alpha$, $\alpha>1$ odd, the $n=1$ case is the unique
nonsplit-Cartan form $N_Q$, and Proposition~\ref{prop:lie-cartans-support}~(\ref{item:lie-cartans-support-2})
gives
\[
\begin{aligned}
 f\text{ is exceptional over }\mathbb F_{2^s}
 &\Longleftrightarrow(\alpha,s)=1,\\
 \mathscr D(f)
 &=
 \begin{cases}
 \{m\geqslant1:(m,\alpha)=1\},&(\alpha,s)=1,\\
 \varnothing,&(\alpha,s)>1.
 \end{cases}
\end{aligned}
\]
For $Q=3^\alpha$, $\alpha>1$ odd, let $P_\lambda$ denote the $n=1$
specialization of the characteristic-$3$ formula in
\cite[Theorem~4.1]{GuralnickZieve2010}, with the exceptional finite-field
forms classified in \cite[Theorem~4.2]{GuralnickZieve2010}. In that formula, the factor $X^2-\lambda$ gives a ramified pair, each point having
multiplicity $(Q+1)/4$.
This pair identifies the two quadratic twist classes with those of the present
model $M_{Q,\eta}$:
\[
 \lambda\text{ square}\Longleftrightarrow\eta\text{ square},\qquad
 \lambda\text{ nonsquare}\Longleftrightarrow\eta\text{ nonsquare}.
\]
Hence Proposition~\ref{prop:lie-cartans-support}~(\ref{item:lie-cartans-support-4}), without the historical exceptionality proof, gives
\[
\begin{aligned}
 P_\lambda\text{ is exceptional over }\mathbb F_{3^s}
 &\Longleftrightarrow \lambda\text{ nonsquare and }(\alpha,s)=1,\\
 \mathscr D(P_\lambda)&=\{m:(m,2\alpha)=1\}.
\end{aligned}
\]
\end{corollary}

\begin{proof}
The historical genus and geometric pair put the $n=1$ stabilizer in the
nonsplit-Cartan case of Proposition~\ref{prop:nonsplit-cartan}; in characteristic
$2$ that case has one $k$-form. Its criterion after extension to $\mathbb F_{2^{sm}}$ is $(\alpha,sm)=1$; this is impossible for every $m$ if $(\alpha,s)>1$, and otherwise is equivalent to $(\alpha,m)=1$. In characteristic $3$, the historical formula exhibits $X^2=\lambda$ as the unique pair with the stated multiplicity. In the present model (where $4=1$),
\[
 M_{Q,\eta}(X)^2=\eta N_Q\!\left(-\frac{X^2}{\eta}+2\right),
\]
where $Q\equiv3\pmod4$. To compute the corresponding multiplicity, write $z=W+W^{-1}$. At $W=-1$, one has $z+2=(W+1)^2/W$. The factor $W^{Q-1}-1$ in \eqref{eq:NQ} has a simple zero there, and the denominator is nonzero. Thus $N_Q(z)$ has order $Q+1$ in $W+1$, hence order $(Q+1)/2$ in $z+2$. At each point $X^2=\eta$, the parameter $-X^2/\eta+2$ meets $z=-2=1$ simply. Taking the square root gives multiplicity $(Q+1)/4$ for $M_{Q,\eta}$. The other zeros have different multiplicities, so this pair is intrinsic. Its split/quadratic splitting type is invariant
under $k$-M\"obius equivalence and distinguishes the two finite-field
forms. The square-class identification follows, and
Proposition~\ref{prop:lie-cartans-support}~(\ref{item:lie-cartans-support-4}) gives the arithmetic assertions.
\end{proof}

The numerical extension criteria themselves are historical.
For characteristic $2$, Cohen--Matthews prove the criterion in Theorems 1.1--1.2
\cite{CohenMatthews1994}. Lenstra--Zieve give the characteristic-$3$ criterion on
p.~218 \cite{LenstraZieve1996}. The present framework gives a uniform Cartan/semilinear reproof and identifies the historical twist parameter with the nonsplit-Cartan descent class. The Guralnick--Zieve model already shows that $n>1$ has positive full-closure genus. The higher-parameter Lenstra--Zieve family, whose parameter is denoted $m$ in \cite{LenstraZieve1996}, has genus $(m-1)(Q-1)/2$; thus its members with $m>1$ lie outside the present scope.
The two-wild-branch family of \cite{GuralnickRosenbergZieve2010} is likewise
positive-genus and is not claimed here.

\subsection{The Ding--Zieve wild family}

Ding--Zieve use $r=p^{2k}$ and $q=p^\ell$; here $Q=r$, $e=k$, and $s=\ell$.
Their Theorem~1.7 proves, under $\nu_2(s)\leqslant\nu_2(e)$, indecomposability,
exceptionality, and permutation on every odd-degree extension; Proposition~5.1
gives the ramification, and Remark~1.8 gives one $Q=9$ decomposition over
$\mathbb F_{q^2}$. Their proof of exceptionality runs through Proposition~4.1(4.1.4),
whereas our proof factors through the general split-Cartan theorem above. Their
Section~6 also describes the discovery as a genus-zero $\operatorname{PSL}_2$ search with a
dihedral stabilizer of order $Q-1$, exactly the split-Cartan geometry here. We use
their displayed function and the Dickson identity as input.

\begin{corollary}[The Ding--Zieve family from the split-Cartan classification]

Let $p$ be odd, $k=\mathbb F_q$ with $q=p^s$, $Q=p^{2e}$, and let
$a\in k^\times$ be nonsquare. Put
\[
 f^{\mathrm{DZ}}_{Q,a}(X)=
 \frac{E_Q(X,a)^{(Q+1)/2}}{(X^2-4a)^{(Q^2-Q)/4}},
\]
where $E_Q$ is the Dickson polynomial of the second kind. Then:
\begin{enumerate}[label=(\arabic*),ref=\arabic*,font=\itshape]
\item\label{item:unlabelled-the-ding-zieve-family-from-the-split-cartan-classification-1}
\[
 f^{\mathrm{DZ}}_{Q,a}=a^{(Q-1)/2}T_{Q,a}.
\]
Hence, up to the displayed $k$-target scaling, it is the nonsquare diagonal
split-Cartan form with $(G,H)=(\operatorname{PSL}_2(Q),D_{Q-1})$, degree $Q(Q+1)/2$, and one $k$-M\"obius class for fixed $Q$ and $k$.

\item\label{item:unlabelled-the-ding-zieve-family-from-the-split-cartan-classification-2} If $d_F=2e/(2e,s)$, then
\[
\begin{aligned}
 f^{\mathrm{DZ}}_{Q,a}\text{ exceptional}
 &\Longleftrightarrow d_F\text{ even}
 \Longleftrightarrow \nu_2(s)\leqslant\nu_2(e),\\
 \mathscr D(f^{\mathrm{DZ}}_{Q,a})
 &=\{m\geqslant1:m\text{ odd}\}.
\end{aligned}
\]
Thus the hypothesis and odd-extension assertion of
\cite[Theorem~1.7]{DingZieve2023} are recovered, while necessity and failure on every even-degree extension are additional relative to the statement of that theorem. Under these conditions, the rational function is
indecomposable over $k$ and wild. It is geometrically indecomposable exactly for
$Q\ne9$, and it is not $k$-M\"obius equivalent to a polynomial.

\item\label{item:unlabelled-the-ding-zieve-family-from-the-split-cartan-classification-3} Its branch points are $0,\infty$. The fiber over $0$ consists of $Q$
points of index $(Q+1)/2$; the fiber over $\infty$ has three points of indices
\[
 Q,\qquad \frac{Q(Q-1)}4,\qquad \frac{Q(Q-1)}4.
\]
For the Galois quotient $\mathbf P^1\to\mathbf P^1/G$, the inertia orders on
$\mathbf P^1(\mathbb F_Q)$, $\mathbf P^1(\mathbb F_{Q^2})\setminus\mathbf P^1(\mathbb F_Q)$, and its complement are
$Q(Q-1)/2$, $(Q+1)/2$, and $1$, respectively, with $I_1=\operatorname{Syl}_p(I_0)$ and $I_2=1$. This gives a second proof of \cite[Proposition~5.1]{DingZieve2023}.

\item If $Q=9$ and the equivalent conditions in (\ref{item:unlabelled-the-ding-zieve-family-from-the-split-cartan-classification-2}) hold, then
$f^{\mathrm{DZ}}_{9,a}/\mathbb F_{q^m}$ is indecomposable exactly for odd $m$; for
every even $m$, every complete decomposition has inner-to-outer factor degrees
$(3,15)$. Over $\mathbb F_{q^2}$, if $b^2=a$, one such decomposition is
\[
 f^{\mathrm{DZ}}_{9,a}=g_b\circ h_b,
 \quad h_b(X)=\frac{X^3+ab}{X^2+bX+a},
 \quad g_b(Y)=\frac{(Y^3+aY+ab)^5}{Y^6}.
\]
\end{enumerate}
\end{corollary}

\begin{proof}
Put $W_0=Z^{(Q-1)/2}$ and $Y=W_0+W_0^{-1}$. The Dickson identity and
\eqref{eq:TQ-def} give $f^{\mathrm{DZ}}_{Q,1}(Y)=T_Q(Y)$. If $\theta^2=a$, then
$E_Q(\theta X,a)=\theta^Q E_Q(X,1)$, hence $f^{\mathrm{DZ}}_{Q,a}(X)=\theta^Q f^{\mathrm{DZ}}_{Q,1}(X/\theta)=a^{(Q-1)/2}T_{Q,a}(X)$, proving (\ref{item:unlabelled-the-ding-zieve-family-from-the-split-cartan-classification-1}). Proposition~\ref{prop:lie-cartans-support}~(\ref{item:lie-cartans-support-3}) now gives
(\ref{item:unlabelled-the-ding-zieve-family-from-the-split-cartan-classification-2}); the $2$-adic equivalence is immediate from
$d_F=2e/(2e,s)$. The remaining structural assertions follow from
Proposition~\ref{prop:split-cartan}, Theorem~\ref{thm:structural-consequences},
Theorem~\ref{thm:wild-novelty}, and
Corollary~\ref{cor:complete-polynomializability}.

For (\ref{item:unlabelled-the-ding-zieve-family-from-the-split-cartan-classification-3}), the two $G$-orbits in the ramification locus of $\mathbf P^1\to\mathbf P^1/G$ are
$\mathbf P^1(\mathbb F_Q)$ and $\mathbf P^1(\mathbb F_{Q^2})\setminus\mathbf P^1(\mathbb F_Q)$, with inertia orders
$Q(Q-1)/2$ and $(Q+1)/2$; the defining invariant sends them to $\infty$ and $0$.
Since $\gcd(Q-1,(Q+1)/2)=1$, the tame intersections with $H=D_{Q-1}$ are trivial,
so the branch-orbit formula gives the $Q$ points over $0$. On
$G/I_{\rm w}\cong\mathbf P^1(\mathbb F_Q)$ the split-Cartan normalizer has orbit sizes
$2,(Q-1)/2,(Q-1)/2$; the double-coset formula in Theorem~\ref{thm:branch-orbit} gives
$\operatorname{Ram}_{\infty}(f^{\mathrm{DZ}}_{Q,a})
=\{\!\{Q,Q(Q-1)/4,Q(Q-1)/4\}\!\}$. At a wild point, with $t=1/Z$ and
$\sigma(Z)=\lambda Z+c$ in the Borel inertia,
\[
 \sigma(t)-t=\frac{(1-\lambda)t-ct^2}{\lambda+ct}.
\]
Its valuation is $1$ for $\lambda\ne1$ and $2$ for a nontrivial translation,
so $I_1$ is the translation Sylow-$p$ subgroup and $I_2=1$; the other inertia is
tame.

Finally, let $Q=9$. Theorem~\ref{thm:wild-novelty} gives
$G\cong A_6$, $A=M_{10}$, and $H=D_8$. Odd extensions retain the arithmetic
pair, while even extensions reduce to $(G,H)$. To determine complete decompositions, use the Dickson subgroup list for $A_6\cong\operatorname{PSL}_2(9)$. A proper overgroup of $D_8$ has order divisible by eight. The Borel, $A_4$, $A_5$, and nonsplit-Cartan types have orders $36,12,60,10$ and cannot contain it; the split-Cartan normalizer is $D_8$ itself. The only maximal overgroup type is the square-subfield group $\operatorname{PGL}_2(3)\cong S_4$. Since $[S_4:D_8]=3$ and $[A_6:S_4]=15$, every maximal chain has the form
\[
 D_8\lneq S_4\lneq A_6
\]
with indices $3,15$. The intermediate $S_4$ need not be unique, but the chain length and index sequence are. Thus Theorem~\ref{thm:complete-decomposition} gives the
factor-degree statement. The displayed $\mathbb F_{q^2}$ decomposition is precisely
\cite[Remark~1.8]{DingZieve2023}; the decomposition behavior over every finite extension is supplied by the
preceding subgroup-chain argument.
\end{proof}

\subsection{Classical and cubic forms}

The prime-index separable power and R\'edei towers are the prime-order
specializations of Proposition~\ref{prop:cyclic}, while the Dickson tower is
the dihedral family of Proposition~\ref{prop:dihedral}; hence the normal forms
agree. Independently, the extension-degree theory above gives the three
classical extension criteria at once:
\[
\begin{aligned}
 \mathscr D(X^\ell)&=\{m:\operatorname{ord}_\ell(q)\nmid m\},\\
 \mathscr D(R_{\ell,\delta})&=\{m:\operatorname{ord}_\ell(-q)\nmid m\},\\
 \mathscr D(D_\ell(X,a))&=\{m:\operatorname{ord}_\ell(q^2)\nmid m\}.
\end{aligned}
\]
Here $a\in k^\times$. Within the present classification the classical Dickson family appears as two square-class forms in odd characteristic and one in characteristic $2$. Composite indices lie in the decomposition closure, while the characteristic-prime endpoints belong to the inseparable boundary.

For degree $3$, Ding--Zieve \cite[Theorem~1.3]{DingZieve2022} give a conceptual
reformulation of the Ferraguti--Micheli classification. The separable permutation class is represented by $X^3$ when $q\equiv2\pmod3$, by a degree-three R\'edei function when $q\equiv1\pmod3$, and by $X^3-\alpha X$ with $\alpha$ nonsquare when $3\mid q$. These are the split cyclic, nonsplit cyclic, and pure-additive affine case with $r=1$, respectively. The extension-degree theory gives uniformly
\[
 (A,G,U,H)\cong(S_3,C_3,C_2,1),\qquad
 \mathscr D(f)=\{m\geqslant1:m\text{ odd}\}.
\]
The $\alpha=0$ characteristic-$3$ case $X^3$ is the inseparable Frobenius
boundary. No new proof of the historical cubic classification is claimed.
For the broader classification of all degree-three rational functions over
finite fields of odd characteristic under pre- and post-composition by
M\"obius transformations over the same field, see
Hou--Peng--Qiang--Zhao \cite{HouPengQiangZhao2026}.

\section{Concluding remarks}\label{sec:conclusion}

For a semilinear $\mathbf P^1$-datum $\mathcal D$ over $k$, the quotient morphism $\pi_{\mathcal D}:C/H\to C/G$ and its Frobenius descent $\gamma_{\mathcal D}:X_{\mathcal D}\to Y_{\mathcal D}$ determine an intrinsic $k$-M\"obius class $\mathcal M_{\mathcal D}$, and every separable rational function with Galois closure of genus zero belongs to such a class. Different semilinear data can determine the same class; Section~\ref{sec:semilinear-data} gives the exact equality criterion. The Frobenius component is essential: arithmetic and geometric primitivity can differ, producing rational functions that are indecomposable over $k$ but geometrically decomposable. The tame and wild classifications determine all such indecomposable $k$-M\"obius classes of degree greater than one over the prescribed finite field.

For every separable rational function of degree greater than $1$ in the genus-zero class, permutation over a finite field is equivalent to exceptionality over that field. If $d$ is the full constant-field degree, whether $m\in\mathscr D(f)$ depends only on $\gcd(m,d)$ through the divisor set $\mathcal E_f$, which is closed under taking divisors. For the indecomposable classes, the family analysis determines $\mathscr D(f)$ explicitly, together with exact class counts and the existence or nonexistence of polynomial representatives.

Section~\ref{sec:historical-recovery} locates several previously known exceptional families within this classification. For the entries explicitly identified there as independent recoveries, only the stated historical formula or geometric input is used, and the present arguments independently rederive the indicated extension-degree or structural conclusions. When the numerical criterion is already historical, that section says so explicitly. Other entries are recorded only as compatibility with known normal forms or as structural placement inside the present classification.

Every functional decomposition is encoded by a strict $\varphi$-stable subgroup chain
\[
 H=J_0\lneq J_1\lneq\cdots\lneq J_r=G,
\]
and complete decompositions are exactly those for which each interval contains no further $\varphi$-stable subgroup. Every factor is separable and again has Galois closure of genus zero; its arithmetic normal closure and full constant field are determined by the corresponding factor core. Complete decompositions need not have the same length or factor-degree multiset in general. In the affine family, the Jordan--H\"older theorem for the associated Frobenius module gives the same length and factor-degree multiset for all complete decompositions.

Related positive-genus results include permutation--exceptionality criteria
for rational functions induced by equivariant elliptic isogenies
\cite{Fan2026EllipticIsogeny} and the arithmetic classification of separable
indecomposable tame exceptional maps with genus-one Galois closure
\cite{Fan2026GenusOne}. Beyond these settings, it is natural to ask which
parts of the present subgroup-chain formalism extend to positive-genus
Galois closures, where quotient curves need no longer have genus zero.

\end{document}